\documentclass[12pt]{article}
\usepackage{mathrsfs}
\usepackage{amssymb}
\usepackage{color}
\usepackage{amsmath}
\allowdisplaybreaks[4]

\usepackage{amsthm}
\usepackage{amsmath}
\usepackage{graphicx}

 \newtheorem{thm}{Theorem}[section]
 
 \newtheorem{lem}[thm]{Lemma}
 
 \theoremstyle{definition}

 \numberwithin{equation}{section}

\newtheorem{theorem}{Theorem}[section]

\theoremstyle{definition}
\newtheorem{definition}[theorem]{Definition}

\theoremstyle{remark}
\newtheorem{remark}[theorem]{Remark}

\begin{document}
\title{Vanishing Viscosity Limit
 for  Incompressible Viscoelasticity in Two Dimensions}

\author{Yuan Cai\footnote{School of Mathematical Sciences, Fudan University, Shanghai 200433, P. R. China.
Email: ycai14@fudan.edu.cn}
\and
Zhen Lei\footnote{School of Mathematical Sciences, Fudan University; Shanghai Center for Mathematical Sciences, Shanghai 200433, P. R. China.
Email: zlei@fudan.edu.cn}
\and
Fanghua Lin
\footnote{Courant Institute of Mathematics, New York University, NY 10012, USA; and Institute of Mathematical
Sciences of NYU-ECNU at NYU-Shanghai, Shanghai 200433, P. R. China. Email: linf@cims.nyu.edu}
\and
Nader Masmoudi
\footnote{Courant Institute of Mathematical Sciences, New York University, NY 10012, USA. Email: mas-
moudi@cims.nyu.edu}
}
\date{}
\maketitle

\begin{abstract}
This paper studies the inviscid limit of the two-dimensional incompressible viscoelasticity,
which is a system coupling a Navier-Stokes equation with a transport equation for the
deformation tensor. The existence of global smooth solutions near the equilibrium with a
fixed positive viscosity was known since the work of \cite{LLZ2005}. The inviscid
case was solved recently by the second author \cite{Lei16}. While the latter was solely
based on the techniques from the studies of hyperbolic equations, and hence the 2D problem
is in general more challenge than that in higher dimensions, the former was relied
crucially upon a dissipative mechanism. Indeed, after a symmetrization and a linearization
around the equilibrium, the system of the incompressible viscoelasticity reduces to an
incompressible system of damped wave equations for both the fluid velocity and the
deformation tensor. These two approaches are not compatible. In this paper, we prove
global existence of solutions, uniformly in both time $t \in [0, \infty)$ and viscosity
$\mu \geq 0$. This allows us to justify in particular the vanishing viscosity limit for
all time. In order to overcome difficulties coming from the incompatibility between the
purely hyperbolic limiting system and the systems with additional parabolic viscous
perturbations, we introduce in this paper a rather robust method which may apply to a wide
class of physical systems of similar nature. Roughly speaking, the method works in two
dimensional case whenever the hyperbolic system satisfies intrinsically a ``Strong Null
Condition". For dimensions not less than three, the usual null
condition is sufficient for this method to work.

\end{abstract}

\maketitle





\section{Introduction}

One of the common manifestations of anomalous phenomenon in complex fluid comes from the elastic effects.
The different rheological and hydrodynamic properties can be attributed to the special coupling between
the transportation of the internal variable and the induced elastic stress. In the variational energetic formulation,
these properties can be attributed to the competition between the kinetic energy and the internal elastic effects
(see, for instance, \cite{LLZ2005}).

For isotropic, hyperelastic and homogeneous incompressible materials, the motion can be
described by the following (fundamental elastodynamic) system
\begin{equation}\label{model11}
\begin{cases}
\partial_tv + v\cdot\nabla v + \nabla p = \nabla\cdot(\frac{\partial W(F)}{\partial F}F^\top),\\[-4mm]\\
\nabla\cdot v = 0.
\end{cases}
\end{equation}
Here $v$ is the velocity field, $p$ the scalar pressure which is the Lagrangian multiplier due to the incompressibility constraint, $W(F)$ the internal elastic energy density and $F$ the deformation tensor.

The deformation tensor $F$ is often presented in a Lagrangian description using a
time-dependent family of orientation-preserving diffeomorphisms
$x(t, \cdot)$, $0 \leq t < T$. Material points $y$ in the
reference configuration are deformed to the spatial positions
$x(t, y)$ at time $t$. We shall use $y(t,x)$ to denote the inverse of $x(t,\cdot)$.
The flow map $x(t, y)$ is determined as usual by the velocity $v(t,x)$ via the following
ODEs:
\begin{equation*}
\begin{cases}
\frac{d x(t,y)}{d t}=v(t,x(t,y)),\\
x(0)=y.
\end{cases}
\end{equation*}
Such map $x(t,y)$ would be uniquely defined whenever the velocity field $v(t,x)$ is in
appropriate Sobolev space \cite{DiPerna}. The deformation tensor is then defined by
\begin{equation*}
\tilde{F}(t,y)=\frac{\partial x(t,y)}{\partial y}.
\end{equation*}
One simply identifies it as $F(t,x(t,y))=\tilde{F}(t,y)$ in the Eulerian coordinates
$(t,x)$.

It is easy to check the incompressible condition is equivalent to $\nabla\cdot F^\top=0$
(see, for instance,  \cite{LLZ2005}). In addition, one can also deduce that
\begin{equation}\label{compatible-equ}
\begin{cases}
\partial_t F+v\cdot\nabla F=\nabla v F, \\[-4mm]\\
F_{mj}\nabla_mF_{ik}=F_{lk}\nabla_l F_{ij},\quad i, j, m, k, l \in \{1, 2, \cdots, n\}.
\end{cases}
\end{equation}
See for examples, \cite{LeiLZ08, LLZ2005}.
The above \eqref{compatible-equ} is essentially the compatibility condition for the
velocity field and the flow map. In what follows, we use the following notations
\begin{equation*}
(\nabla v)_{ij}=\nabla_jv_i, \quad
(\nabla v F)_{ij}=(\nabla v)_{ik} F_{kj}, \quad
(\nabla \cdot F)_i=\nabla_j F_{ij}, \quad
\end{equation*}
and the summation convention over repeated indices will always be applied.

The equations for elastodynamics \eqref{model11} may then be written equivalently as
\begin{equation}\label{model11-1}
\begin{cases}
\partial_tv + v\cdot\nabla v + \nabla p = \nabla\cdot( \frac{\partial W(F)}{\partial F} F^\top),\\[-4mm]\\
\partial_t F+v\cdot\nabla F=\nabla v F,\\[-4mm]\\
\nabla\cdot v = 0,\quad \nabla\cdot F^\top = 0,
\end{cases}
\end{equation}
with the compatible condition $\eqref{compatible-equ}_2$.

Taking into account of viscosity, one leads to the Oldroyd system of viscoelasticity:
\begin{equation}\label{model11-2}
\begin{cases}
\partial_tv + v\cdot\nabla v + \nabla p = \mu\Delta v + \nabla\cdot( \frac{\partial W(F)}{\partial F} F^\top),\\[-4mm]\\
\partial_t F+v\cdot\nabla F=\nabla v F,\\[-4mm]\\
\nabla\cdot v = 0,\quad \nabla\cdot F^\top = 0.
\end{cases}
\end{equation}
Here $\mu \geq 0$ denotes the fluid viscosity. We notice that the nonlinear coupling structure in \eqref{model11-2} is universal, and it appeared in many physical equations including magneto-hydrodynamic equations and liquid
crystal flows, see \cite{Lin}.

A main goal of this paper is to justify the global-in-time inviscid limit from the
viscoelastic system  \eqref{model11-2} to the elastic system  \eqref{model11-1} in two
dimensions. More precisely, we will show that smooth solutions to \eqref{model11-2} in certain
weighted Sobolev spaces exist uniformly in time $t \geq 0$  and $\mu  \geq 0$. This
allows us to justify the vanishing viscosity limit for all time.

The presence of viscosity requires the use of Eulerian coordinates. Following the standard vector fields method of Klainerman and the ``ghost weights" method of Alinhac,
a number of rather essential difficulties appear due to the incompatibility between
these methods needed for the limiting hyperbolic systems
and the equations in the limiting process that possess additional
parabolic viscous terms. In particular the viscous terms would result in ``bad"
commutators. A reformulation of the system in these coordinates seems necessary in order
for us to identify a stronger notion of null condition, which is essential in 2D case.
With this strong null condition we will be able to do various modifications on
Klainerman's and Alinhac's methods. We shall discuss it in more details in subsection 2.2
below.

\subsection{A Review of Related Results}

The study of   dynamics of  isotropic, hyperelastic and homogeneous materials has a long history.
 Compressible elastodynamic  systems  (commonly referred as elastic waves in literature),
are  quasilinear wave type systems with multiple wave speeds.
For three-dimensional elastic waves, John \cite{John88} showed the existence of almost global solutions
for small displacement (see also \cite{KlainermanSiderisS96}).
On the other hand, John \cite{JohnBook} proved that a genuine nonlinearity condition leads to
formations of finite time singularities for spherically symmetric, arbitrarily small but
non-trivial displacements (see \cite{T98} for large displacement singularities). When the
genuine nonlinearity condition is excluded,  the existence of global small solutions may
be expected even in non-symmetric cases.  The difficulty in obtaining global
solutions lies in the understanding of the interaction between the fast pressure waves and
slow shear waves at a nonlinear level. A breakthrough is due to Sideris \cite{Sideris96, Sideris00} and also Agemi \cite{Agemi00}, under a nonresonance condition which is
physically consistent with the system.
The proof of Sideris  is based on the vector field method of Klainerman
\cite{Klainerman85, Klainerman86} and the weighted Klainerman-Sideris $L^2$ energy
(introduced in their earlier work \cite{KlainermanSiderisS96}). The proof of Agemi
relies on a direct estimate of the fundamental solution. We note that the
nonresonance condition complements John's genuine nonlinearity condition. With an
additional repulsive Poisson term, a global existence was established in \cite{HM17} which
allows a general form for the pressure.

For the incompressible elastodynamics, the only waves presented in the
isotropic systems are shear waves which are linearly degenerate. The
global-well-posedness was obtained by Sideris and Thomases in \cite{SiderisThomas05,
SiderisThomas07} (see \cite{ST06} for a  unified treatment, and \cite{LW} for some
improvement on the uniform time-independent bounds on the highest order energies).
Based on the aforementioned achievements, the theory of global existence of
solutions for the three-dimensional elastic waves with small initial data is
relatively satisfactory.

In  the two dimensional case, the proof of long time existence  for the
elastodynamics is more difficult due to the weaker time decay rate.
The first large time existence result is the recent work  \cite{LSZ13} where the
authors showed the almost global existence for the  two-dimensional incompressible
elastodynamics in Eulerian coordinates. By observing an improved null structure
for the system in Lagrangian coordinates (see also discussions in subsection 2.2),
the second author \cite{Lei16} proved  the global well-posedness using the energy
method of Klainerman and Alinhac's ghost weight approach. Afterwards,
Wang \cite{WangXueCheng2014} gave a new proof of this latter result using
space-time resonance method \cite{GMS12} and a normal form transformation.

When the viscosity is present and strictly positive, the global well-posedness
near equilibrium state was first obtained in \cite{LLZ2005} for the two-dimensional
case. In this case, after a symmetrization and linearization around the equilibrium
state, the system becomes a non-standard (incompressible) damped wave systems for both
velocity field and the deformation tensor, see also \cite{Lin,LZ2008}. This method works
both in 2D and 3D cases. Lei and Zhou \cite{LZ2005} obtained similar results by working
directly on the equations for the deformation tensor through an incompressible limit
process. For many related discussions we refer to  \cite{LeiLZ08} and \cite{CZ2006} and
\cite{HX2010,HL2016,HW2011,Lei10,Lei14,Lei3,M13,QZ2010,ZF} and the references therein.
In all these works, a dissipative structure of the viscoelastic systems (with a
strictly positive viscosity) is a key ingredient to study the long time behavior.
Thus the size of the initial data depends on the viscosity in order to have global in
time existence. Consequently, these arguments can not be applied to study the vanishing
viscosity problem. For the latter one has to deal a nonlinear coupled system of
equations in which both parabolicity and hyperbolicity can't be ignored.

Similar to the study of vanishing viscosity limits for classical fluid
dynamics, one expects that when the fluid viscosity goes to zero, the limit of
solutions to the viscoelastic system converges to a solution to the
elastodynamic system. In the case of Navier-Stokes equations, a lot have
been learned since the work of Kato \cite{Kato72} and Swann \cite{Swann} (see
also a recent version \cite{Masmoudi2007}). These results are not expected to
hold globally in time. If one tries to prove a global in time convergence,
the matter is completely different.

The work of Kessenich \cite{Kessenich} established the global
well-posedness theory for three-dimensional incompressible viscoelastic materials
uniformly in the viscosity and in time. Here, though the presence of viscosity
prevents natural hyperbolic scaling invariance, Kessenich used nevertheless the
scaling operator. His strategy is to apply first this operator directly to the system,
then to deal with the commutators between the scaling operator and the viscosity
terms. Sufficiently fast decay rates in 3D are the key. Another important ingredient in
\cite{Kessenich} is a Hardy's type estimate. It is used to compensate the
derivative loss problem caused by commuting with the viscous terms. In the two-dimensional
case neither of these two key steps can be accomplished easily. One of the main
reasons is that, while the ghost weight of Alinhac seems to be a necessary tool
for the highest order energy estimates in the two dimensional problems, one can not
directly apply it here because it would create extra non-decaying terms involving
commutators with the viscous term.

Let us also discuss some closely related historical works on quasilinear wave
type equations. For quasilinear wave equations in dimension three, and for small
initial data, one can obtain an almost global existence \cite{JohnKlainerman84}.
When the spatial dimensions are not bigger than three, the global existence would depend
on two basic assumptions: the initial data should be sufficiently small and the
nonlinearities satisfy a type of null condition \cite{Sideris00}. For nonlinear wave
equations with sufficiently small initial data and the null condition is not satisfied,
the finite time blow-up was shown by John \cite{John81}, Alinhac \cite{Alinhac00b} in 3D, and
by Alinhac \cite{Alinhac99a,Alinhac99b, Alinhac01b} in the two-dimensional case.
Under the null condition, the  fundamental work on global solutions for three
dimensional scalar wave equation were obtained by Klainerman \cite{Klainerman86} and by
Christodoulou \cite{Christodoulou86}. In two dimensions, the global solutions were
proven by Alinhac \cite{Alinhac01a} under the null condition and under the assumption that the initial data is compactly supported.

\subsection{Difficulties and Key Ideas}
To simplify the presentation, we will focus only on the Hookean elasticity which
corresponds to $W(F)=\frac12|F|^2$. The general case differs only by the cubic and higher order terms,
which won't make much difference in our arguments, see also the comments in
\cite{Lei16,WangXueCheng2014}. In the zero viscosity limit, the
viscoelastic systems tend to a hyperbolic system. One would naturally try to follow the
generalized energy method of Klainerman. An attractive feature of this method is of
course that it suffices to use the weighted Sobolev inequalities involving the invariance
of the system: translations, rotations, scaling and the Lorentz invariance. It avoids the
delicate estimates of fundamental solutions of wave equations \cite{Sideris00}.
Similarly, the Alinhac's ghost weight method may enable one to apply Klainerman's
generalized energy method to the two-dimensional wave equations \cite{Alinhac01a}.
As Alinhac's method seems to be a most valuable tool currently to get the highest
order energy estimates for two dimensional problems in such a way that it could
lead to a critical decay in time. The latter is needed for global in time existence,
see for examples \cite{Alinhac01a, Lei16, LSZ13}.

As in Alinhac's works, we would introduce ``good unknowns" and explore certain damping mechanism for these ``good unknown"
due to the outgoing energy flux when ghost weights are applied (very much like the excess term in the energy monotonicity formulae). In the standard energy estimates, the viscosity may also give rise to some dissipative effects which is a good news.  However, due to the additional viscous terms which violate hyperbolic scaling, it creates
various ``bad commutators" when either the vector field method or the method of ghost
weights are applied (see \eqref{test}). As we mentioned earlier, the ghost weights
are not needed in the 3D case as one has already established the critical decay in
time in estimates of the highest order energies (see \cite{Kessenich}) without using the usual null condition assumption. In addition, there is, see \cite{Kessenich} a Hardy type inequality which is useful for one to get around the difficulties caused by viscosity. Hence for the two-dimensional case, we definitely need a new strategy.

Let us observe more closely how the ghost weight method causes new problems with the highest order energy estimates for the viscoelastic systems:
Let $\omega = x/|x|$, $\sigma = r- t$ and $q(\sigma) = \arctan\sigma$.
Suppose one tries to estimate the highest order energy for the viscoelastic
systems \eqref{model11-2} with the ghost weight $e^{q(\sigma)}$, formally one has
\begin{eqnarray} \label{test}
&&\frac12\frac{d}{dt}\int_{\mathbb{R}^2}(|Z^{\kappa }v|^2+|Z^{\kappa }(F-I)|^2 )e^q dx\nonumber\\
&&\quad +\ \mu\int_{\mathbb{R}^2}|\nabla Z^{\kappa } v|^2 e^q dx + \frac12\int_{\mathbb{R}^2} \frac{|Z^{\kappa } v + Z^{\kappa }(F-I)\omega|^2
+|Z^{\kappa }(F-I)\omega^\perp|^2
}{1 + \sigma^2}e^qdx \nonumber\\
&&= \frac12\mu \int_{\mathbb{R}^2}|Z^{\kappa } v|^2 \Delta e^q dx + \cdots
\end{eqnarray}
Here $Z$ represents a generalized vector field (see section 2 for
precise definitions). Note the estimate is for the difference $F-I$ as we consider
the problem when the deformation tensor perturbs around the (equilibrium) identity matrix.
The two coercive terms on the second line are due to the viscosity and the ghost weight, respectively.  It will be important, and be clear later on, that we observe the quantities $ v + (F-I) \omega $ and $(F-I) \omega^\perp $ are  ``good unknowns".
Suppose, for the sake of arguments, that we can handle the nonlinear terms
(this is far from being trivial and it will require the notion of Strong
Null Condition), we are still facing the difficulty to obtain the expected energy
estimates. Since the right hand side involves the viscous term, it is
not clear at all that how one can treat them.  In fact, these terms can not be
absorbed directly by the coercive terms since a \textit{spatial derivative} is
missing. Moreover, it is not integrable in time as $|\Delta e^q| \backsimeq 1$
near the light cone $r \backsimeq t$.

Our first idea to solve this  difficulty is to take an advantage of the viscous
terms presented in the energy estimates at lower order derivative levels.
To do so, we apply operators \eqref{test},  with $\nabla  Z^{\kappa -1 } $, instead of $Z^{\kappa }  $,
namely one of the derivatives has to be a spatial regular derivative and we combine it with an energy estimate (without the ghost weight) when operators $Z^{\kappa -1 } $ are applied. The good viscous terms in the
latter lower order energy estimates are then used to absorb the commutators from
former one. In the sequel,  we will use $\mathcal  E_\kappa$ to denote
energy estimates with  $\nabla  Z^{\kappa -1 } $   and $  E_{\kappa-1}$  to denote
energy estimates with  all the vector fields  $Z^{\kappa -1  } $. We will call
$\mathcal  E_\kappa$ the modified  generalized energy  and $ E_{\kappa -1 }$ the generalized energy.
In carrying out this idea, there are a few new difficulties coming from the
nonlinear terms. Indeed, the fact that we only estimate the modified energy at the
highest derivatives level with the ghost weights, it produces only good coercive
terms which contain also at least one spacial derivatives that coming from the
giving ghost weights. The latter would not be sufficient to control the all
nonlinear terms. It is at this stage certain special nonlinear structures coming to
play.

Our second key idea is to transform the viscoelastic system to a ``fully nonlinear" one,
together with a transformed fully nonlinear constraint.
It turns out that in this ``fully nonlinear" system,
the good unknowns in the nonlinear terms always possess an extra spatial derivatives and
thus the transformed system satisfies the ``strong null condition" (see the definition in subsection 2.2).
In fact, since $v$ and $F^\top$ are divergence free, there exist potential functions $V$,
$H=(H_1,H_2)$ such that
$$v=\nabla^\perp V, \quad (F-I)^\top=\nabla^\perp H.$$
Then we can reformulate the system of Hookean viscoelasticity
as follows (see section 2 for a detailed derivation):
\begin{equation}\label{VisElas3}
\begin{cases}
\partial_tV-\mu\Delta V-\nabla\cdot H=\nabla^\perp\cdot
\nabla\cdot \Delta^{-1}\big(-\nabla^\perp V\otimes\nabla^\perp V+\nabla^\perp H\otimes\nabla^\perp H \big), \\[-4mm]\\
\partial_t H -\nabla V=\nabla^\perp H\nabla V, \\
\end{cases}
\end{equation}
with the constraint
\begin{equation} \label{constr}
\nabla^\perp\cdot H=\nabla^\perp H_2 \cdot \nabla H_1.
\end{equation}

As in \cite{Lei16}, the strong null condition would
also mean that in these nonlinear terms, there are always good unknowns in each
individual term. The resulting nonlinear structure permits one to perform various
integration by parts, and to obtain desired decay estimates. For related discussions
on this strong null condition in more general setting for nonlinear wave equations, we refer to a forthcoming preprint \cite{CLM} on a simplified wave model.

Now we state the main result of this paper  as follows:
\begin{thm}\label{thm}
Let $M> 0$, $0<\gamma< \frac18$ be two given constants, $(\nabla V_0,\nabla H_0) \in H^{\kappa-1}_\Lambda$
and $(V_0,H_0)\in H^{\kappa-1}_\Lambda$ with $\kappa\geq 12$.
Suppose that $H_0$ satisfies the constraint \eqref{constr} and
$$\|(\nabla V_0,\nabla H_0)\|_{H^{\kappa-1}_\Lambda}
+\|(V_0,H_0)\|_{H^{\kappa-1}_\Lambda}\leq M,
\quad \|(V_0,H_0)\|_{H^{\kappa-3}_\Lambda}\leq \epsilon.$$
There exists a positive constant $\epsilon_0<e^{-M}$ which depends on $M,\ \kappa,\ \gamma$
such that, if $\epsilon \leq \epsilon_0$, then the incompressible Hookean viscoelastic systems \eqref{VisElas3} with initial data
$$V(x,0)=V_0(x),\quad H(x,0)=H_0(x)$$
has a unique global classical solution such that
\begin{align*}
&\mathcal{E}_{\kappa}(t) + E_{\kappa-1}(t)
+\sum_{|\alpha|+|a|\leq \kappa-1}
\mu \int_0^t\int_{\mathbb{R}^2}\big( |\Delta
\widetilde{S}^\alpha\Gamma^aV(\tau)|^2 +\ |\nabla \widetilde{S}^\alpha\Gamma^a V(\tau)|^2\big) dxd\tau \nonumber\\
&\qquad\leq C_0 M^2 \langle t\rangle ^{\gamma}, \nonumber\\
&E_{\kappa-3}(t)+ \sum_{|\alpha| +|a| \leq \kappa-3}
\mu\int_0^t\int_{\mathbb{R}^2} |\nabla \widetilde{S}^\alpha\Gamma^a V(\tau)|^2 dxd\tau
\leq C_0\epsilon^2e^{C_0M}
\end{align*}
for some $C_0>1$ uniformly for $0\le t<\infty$ and uniformly for $\mu \geq 0$.
\end{thm}
Here $E_{\kappa-1}$ and $E_{\kappa-3}$ are generalized energy, $\mathcal{E}_{\kappa}$ is new modified generalized energy.
$\widetilde{S}$ and $\Gamma$ are generalized vector fields. For more detailed
discussions, we refer readers to section 2.
\begin{remark}
Here we only need to assume that the viscosity is smaller than a given constant, say $\mu\leq 1$.
When $\mu\geq 1$, one can use the parabolic method of \cite{LeiLZ08, LLZ2005}
to get the uniform bound. In the following arguments, we will always make this assumption.
\end{remark}

\begin{remark}
One can easily adapt our method to the three dimensional case. In fact, the conclusion in
\cite{Kessenich} could be improved slightly that the uniform bound (in terms of the
viscosity) for the highest order energy holds, see also \cite{LW}.
\end{remark}

\begin{remark}
When there is no viscosity, namely $\mu=0$, the system is reduced to the two-dimensional
incompressible elastodynamics. In this case, our proof of global existence also work, and
it can be substantially simplified as there is no need to use the modified
$\mathcal{E}_\kappa$.
\end{remark}

\begin{remark}
The uniform global a priori estimates allow one to justify the vanishing
viscosity limit by a usual compactness argument, see for examples,
\cite{Kato72,Masmoudi2007,Swann}.
\end{remark}

Let us end this introduction by discussing a couple of additional technical
difficulties that one has to resolve in proving the above theorem.
The first one is the issue of derivative loss due to the presence of
viscous terms, whenever one performs the weighted energy estimates.
Heuristically, for the system of elastodynamics, under some smallness
assumption, one can verify that $X_{\kappa-1}\lesssim E_{\kappa-1}$.
Here, $X_{\kappa-1}$ represents the weighted $L^2$ generalized energy.
We need to clarify here that these quantities here are not
the ones from \cite{Lei16} rather they resemble what were defined in
\cite{SiderisThomas07}  (see section 2 for precise definitions).
However, when the viscosity is present, one can only show that $X_{\kappa -
2}\lesssim E_{\kappa-1}$. Consequently, when one deals with energies outside of
the light cone, one has to be extra careful.
The transformation of the original system into a fully nonlinear one turns out
to be useful here. Its advantage as discussed above is the presence
of an extra spatial derivatives in nonlinear terms. It provides more flexibilities in
using the weighted $L^2$ energies along with the integration by parts. In the 3D case \cite{Kessenich},
Kessenich obtained one extra spatial derivative using a Hardy type inequality  along with the weighted $L^\infty-L^2$ estimate.
But in the 2D case, the Hardy's inequality has an additional logarithmic factor, and it is no longer being useful. To estimate
$X_{\kappa-1}$, we introduce a modified weighted energy $Y_{\kappa-1}$. The latter
is useful to capture a better decay property of the good unknowns as in Alinhac's
works. The estimates for $Y_{\kappa-1}$ are similar to $X_{\kappa-1}$.
Thus the derivative loss problem persists for $Y_{\kappa-1}$ in treating
the highest-order energy estimates as well. Fortunately, at this stage we can
borrow the full ghost weight energies to close the estimates.

The newly formulated \eqref{VisElas3} elastodynamic system becomes a nonlocal
``fully nonlinear" system. Generally speaking, for quasilinear or ``fully nonlinear"
systems, one needs certain symmetries to avoid the derivative loss. For
\eqref{VisElas3}, a careful and lengthy examination of the nonlinearities shows
that the system indeed possesses the desired symmetry.

The proposed method also needs decay in time estimates for the lower-order energies
as usual. For the 2D case,  solutions often decay like $\langle t\rangle^{-\frac
12}$ for wave equations. Since the viscoelastic system satisfies the usual null
condition, one obtains a critical decay for energies and hence the implication of an
almost global existence result, see \cite{LSZ13}.
For the global existence of classical solutions, a "strong null structure" used
here for the system may be needed. One of the contributions of this article is to
show the viscoelastic system possess a Strong Null Condition under Eulerian coordinates. Here it is worth to
point out that, for the scalar quasi-linear wave equations which satisfy
the usual null condition, Alinhac \cite{Alinhac01a} used a Hardy-type inequality
for compactly supported solutions to overcome the issue with critical decays.
Here, due to the nonlocality of terms in the system, the compact support property of the
initial data would not be preserved.

The remaining part of this paper is organized as follows.
In the following section, we will formulate the system of incompressible
elastodynamics in Eulerian coordinate and present its basic
properties. In section 3, we will give some linear and nonlinear estimates,
then the weighted $L^2$ norm and some $L^\infty$ norm will be given.
The last section corresponds to the various higher-order and lower-order energy estimates.

\section{Equations and Its Basic Properties}

In this section, we will rigorously introduce  the concept of ``strong null condition" and will reformulate the system as a ``fully nonlinear" one  in which the  ``strong null condition"  can be verified explicitly. Then we introduce some necessary notations and discuss the vector fields applied to the system.

\subsection{The Equations of Motion}
Due to the presence of viscous term,  we will consider the problem in Eulerian coordinates.
Here, partial derivatives with respect to Eulerian coordinates
are written as $\partial_t=\frac{\partial}{\partial t}$, $\partial_i=\frac{\partial}{\partial x_i}.$
Spatial derivatives are abbreviated as $\nabla=(\partial_1,\partial_2)$.
For convenience, we also use the following notations
\begin{equation}\nonumber
\omega = \frac{x}{r},\quad r = |x|,\quad \omega^\perp= (\omega^\perp_1,\omega^\perp_2)=(-\omega_2,\omega_1),\quad \nabla^\perp =(-\partial_2, \partial_1).
\end{equation}

We shall consider the equations of motion for incompressible Hookean elasticity (general nonlinear elasticity can be treated similarly), which corresponds to the Hookean strain energy functional $W(F) =\frac{1}{2}|F|^2$. When the deformation tensor  perturbs  around its equilibrium,
$F=I+G$, the incompressible viscoelastic system \eqref{model11-2} can be rewritten as
\begin{equation}\label{VisElas1}
\begin{cases}
\partial_tv -\mu\Delta v-\nabla\cdot G= -\nabla p-v\cdot\nabla v + \nabla\cdot(GG^\top),\\[-4mm]\\
\partial_tG-\nabla v =-v\cdot\nabla G+ \nabla v G,\\[-4mm]\\
\nabla\cdot v = 0,\quad \nabla\cdot G^\top=0.
\end{cases}
\end{equation}
In  the two dimensional case, it's easy to see that \eqref{compatible-equ}$_2$ is equivalent to
\begin{equation}\label{constr1}
(\nabla^\perp\cdot  G)_i=G_{l2}\nabla_l G_{i1}-G_{l1}\nabla_l  G_{i2}.
\end{equation}
Before reformulation of the system, let us first explicitly introduce the strong null condition.

\subsection{Strong Null Condition and Reformulation of the Viscoelastic System \eqref{VisElas1}-\eqref{constr1}}

The ``strong null condition" is a more restricted notion of the null condition, which was
 originally introduced and applied in \cite{Lei16} in the proof of the global well-posedness of
incompressible elastodynamics.

We start with the following scalar quasilinear wave equations
\begin{equation}\label{QW}
\partial_t^2 u - \Delta u = Q(\partial u, \partial^2u).
\end{equation}
Here $Q$ is a bilinear form.

\begin{definition}(Strong Null Condition)
We say $Q$ satisfies the strong null condition if
\begin{equation*}
Q(\partial u, \partial^2u) = Q_1(\partial u, g(\partial u)) + \mathcal{R},
\end{equation*}
where ther reminder $\mathcal{R}$ satisfies
\begin{equation*}
|\mathcal{R}| \lesssim \frac{|\partial u||\partial Z u|}{1+t},\quad r \geq \frac{t+1}{2}.
\end{equation*}
Here $g$ is good known in the sense of Alinhac \cite{Alinhac01a}:
$$g(u) = \omega\partial_tu + \nabla u.$$
\end{definition}

\begin{remark}
One can compare the strong null condition with the null condition. We say that $Q$ satisfies null condition if
\begin{equation*} 
Q(\partial u, \partial^2u) = Q_1(\partial u, g(\partial u))+ Q_2(g, \partial^2u) + \mathcal{R},
\end{equation*}
where the reminder term $\mathcal{R}$ satisfies
$$|\mathcal{R}| \lesssim \frac{|\Gamma u||\partial^2u| + |\partial u||\partial\Gamma u|}{1 + t},\quad r \geq \frac{t + 1}{2}.$$
The main point is that by null condition we mean that $Q$ contains at least one good unknowns of  $g(u)$ or $g(\partial u)$. By the \textit{Strong Null Condition}, it requires that  $Q$ contains the good unknown of $g(\partial u)$. In general a quasilinear wave equation \eqref{QW} with null condition may not satisfy the strong null condition.
\end{remark}

In \cite{Lei16}, it was observed that writing \eqref{VisElas1}-\eqref{constr1} in Lagrangian
coordinates
\begin{equation}\label{ElastodynamicsL}
\begin{cases}
D_t^2x - \Delta_y x + F^{-T}\nabla_y p = 0,\\[-4mm]\\
\det(\nabla_y x) = 1,
\end{cases}
\end{equation}
and after applying a curl free Riesz operator, one may discover that \eqref{ElastodynamicsL} satisfies the ``strong null condition". Here $D_t,\ \nabla_y$
are derivatives with respect to Lagrangian coordinates.

Now we give a couple of more examples of physical systems for which the strong null condition are valid.

For the following two the dimensional fully nonlinear wave equations which are considered in \cite{CLM}:
\begin{equation}\label{fullynonwave}
(\partial_t^2-\Delta)u=N_{\alpha\beta\mu\nu}\partial_{\alpha} \partial_{\beta} u \partial_{\mu}\partial_{\nu} u,
\end{equation}
where $N_{\alpha\beta\mu\nu}$ satisfies the condition:
\begin{equation*}
N_{\alpha\beta\mu\nu} X_{\alpha}X_{\beta}X_{\mu}X_{\nu}=0
\end{equation*}
for all $X\in\Sigma$, where
$\Sigma=\{X\in \mathbb{R}^{+} \times\mathbb{R}^2: X_0^2=X_1^2+X_2^2  \}$. The equations \eqref{fullynonwave} satisfy the strong null condition. Moreover, it was shown in \cite{CLM} that a class of quasilinear wave equations where the null condition is satisfied can be transformed into \eqref{fullynonwave}.

For ideal magnetohydrodynamic systems:
\begin{equation} \label{MHD_0}
\begin{cases}
\partial_t v+v\cdot\nabla v+\nabla p=b\cdot\nabla b,\\[-4mm]\\
\partial_t b+v\cdot\nabla b=b\cdot\nabla v,\\[-4mm]\\
\nabla\cdot v=0\quad \nabla\cdot b=0.
\end{cases}
\end{equation}
Considering the problem in a nonzero magnetic background $e$, which is set to be $(1,0,...,0)\in\mathbb{R}^n$ without loss of generality. Introducing the following good unknowns:
\begin{equation}
\Lambda^{\pm}=v\pm (b-e).\nonumber
\end{equation}
Then \eqref{MHD_0} can be rewritten  as
\begin{equation*}
\begin{cases}
\partial_t \Lambda^{+}- e\cdot \nabla \Lambda^{+}
+\Lambda^{-}\cdot\nabla\Lambda^{+} + \nabla p = 0,\\[-4mm]\\
\partial_t \Lambda^{-} + e\cdot \nabla \Lambda^{-}
+\Lambda^{+}\cdot\nabla\Lambda^{-} + \nabla p =0,\\[-4mm]\\
\nabla\cdot \Lambda^{+} = 0,\quad \nabla\cdot \Lambda^{-}=0.
\end{cases}
\end{equation*}
It's obvious now the strong null condition is satisfied \cite{CL}.

Inspired by the above examples, we believe there is a large body of physical systems where the strong null condition is satisfied.

Coming back to the system \eqref{VisElas1}-\eqref{constr1}, following Lei-Sideris-Zhou
\cite{LSZ13}, we call $v + G\omega$ and $G\omega^\perp$ good unknowns. They are in the
similar spirit as the concept of good unknowns $g$ of Alinhac \cite{Alinhac01a}. One
writes the nonlinear terms in the momentum equation as
\begin{eqnarray}\nonumber
v\cdot\nabla v - \nabla\cdot(GG^T) &=
& (v + G\omega)\cdot\nabla v - (G\omega)_ j( \nabla_j v + \nabla_jG\omega)\\ [-4mm]\nonumber\\\nonumber
&& -(G\omega^\perp)_j \nabla_j G\omega^\perp \\ [-4mm]\nonumber\\\nonumber
&=&Q_1(g,\nabla v)+
Q_2(G\omega,g(\nabla u))
+Q_2(g, g(\nabla u))\nonumber.
\end{eqnarray}
We note that the system now is of first-order (if we ignoring the viscosity term). It does
explain why there is one spatial derivative less on the unknowns in the nonlinear
term. Obviously, $Q_1$ (which is a transport term) must present and thus system
\eqref{VisElas1}-\eqref{constr1} doesn't explicitly exhibit the strong null structure. One
can observe a similar fact for the $G$-equation in \eqref{VisElas1}-\eqref{constr1}.

We reformulate the system in order to show the strong null structure explicitly. Since $v$ and $G^\top$ are divergence free, there exist potential functions $V$ and
$H=(H_1,H_2)$ such that
$$v=\nabla^\perp V, \quad G^\top=\nabla^\perp H.$$
Then one has
\begin{lem}
For classical solutions,
the system \eqref{VisElas1}  is equivalent to \eqref{VisElas3}:
\begin{equation*}
\begin{cases}
\partial_tV-\mu\Delta V-\nabla\cdot H=\nabla^\perp\cdot
\nabla\cdot \Delta^{-1}\big(-\nabla^\perp V\otimes\nabla^\perp V+\nabla^\perp H\otimes\nabla^\perp H \big), \\[-4mm]\\
\partial_t H -\nabla V=\nabla^\perp H\nabla V, \\
\end{cases}
\end{equation*}
and \eqref{constr1} is  equivalent to \eqref{constr}:
\begin{equation*}
\nabla^\perp\cdot H=\nabla^\perp H_2 \cdot \nabla H_1.
\end{equation*}
Here $\nabla^\perp\cdot\nabla\cdot\Delta^{-1}(\nabla^\perp V\otimes\nabla^\perp V)$ and $\nabla^\perp\cdot\nabla\cdot\Delta^{-1}(\nabla^\perp H\otimes\nabla^\perp H)$ are given by:
\begin{equation*}
\begin{cases}
\nabla^\perp\cdot\nabla\cdot\Delta^{-1}(\nabla^\perp V\otimes\nabla^\perp V)
=\nabla^\perp_i\nabla_j\Delta^{-1}(\nabla_i^\perp V \nabla_j^\perp V), \\[-4mm]\\
\nabla^\perp\cdot\nabla\cdot\Delta^{-1}(\nabla^\perp H\otimes\nabla^\perp H)
=\nabla^\perp_i\nabla_j\Delta^{-1}(\nabla_i^\perp H\cdot \nabla_j^\perp H).
\end{cases}
\end{equation*}
\end{lem}

Before proving the above lemma, let us check first that the above system satisfies the so-called  "strong null condition".
The good quantities here are $V + H\cdot\omega$ and $H\cdot\omega^\perp$. We can calculate that
\begin{eqnarray*}
&&\nabla^\perp_i V\nabla^\perp_j V - \nabla^\perp_i H\cdot\nabla^\perp_j H \\
&&=
(\nabla^\perp_i V+\nabla^\perp_i H\ \cdot\omega)\nabla^\perp_j V
-\nabla^\perp_i H\ \cdot\omega (\nabla^\perp_j V+\nabla^\perp_j H\ \cdot\omega) \\
&&\quad-\nabla^\perp_i H\ \cdot\omega^\perp
\nabla^\perp_j H\ \cdot\omega^\perp .
\end{eqnarray*}
The strong null condition has clearly showed up on the right hand side of the above
equation since all good quantities have an extra spatial derivative. The presence of the
extra zero-order nonlocal Riesz type operator
$ \nabla^\perp\cdot\nabla\cdot\Delta^{-1} $ in \eqref{VisElas3} is an extra issue that we
have to deal with later.

\begin{remark}
At first glance, the resulting system  \eqref{VisElas3} seems to be more complicated than
the original one \eqref{VisElas1}. The nonlinearities of  \eqref{VisElas3} have one more
derivative than that of \eqref{VisElas1}), which makes \eqref{VisElas3} a ``fully nonlinear
system" (in the inviscid case), see also a related formulation in
\cite{WangXueCheng2014}.
But the key point is that, together with the use of the modified generalized
energy $\mathcal{E}_\kappa$ , we can yet apply the ghost weight method along with the
strong null condition in this formulation. Moreover, we can avoid the derivative loss in
deriving the estimates $X_{\kappa-2}\lesssim E_{\kappa-1}$ and
$Y_{\kappa-2}\lesssim E_{\kappa-1}$.
\end{remark}

\begin{proof}
We begin by rewriting the first equation of \eqref{VisElas1} as
\begin{align*}
\partial_tv -\mu\Delta v-\nabla\cdot G= -\nabla p-\nabla\cdot(v\otimes v) + \nabla\cdot(GG^\top).
\end{align*}
Using  $(V,H)$  instead of  $(v,G)$ and applying $\nabla^\perp\cdot$ to the above equation, one has
\begin{equation*}
\Delta(\partial_tV-\mu\Delta V-\nabla\cdot H)=\nabla^\perp\cdot
\nabla\cdot\big(-\nabla^\perp V\otimes\nabla^\perp V
+\nabla^\perp H\otimes \nabla^\perp H \big).
\end{equation*}
Applying $ \Delta^{-1}$  to the above  the equation  yields the first equation of \eqref{VisElas3}.

For each component of \eqref{VisElas1}$_{2}$, the same substitution gives
\begin{align*}
\partial_t \nabla_i^\perp H_j-\nabla_j\nabla_i^\perp V
=-\nabla_l^\perp V\nabla_l \nabla_i^\perp H_j
 +\nabla_l\nabla_i^\perp V  \nabla_l^\perp  H_j.
\end{align*}
Note that
$$-\nabla^\perp_l V \nabla_l H=\nabla_l V \nabla_l^\perp H,$$
hence,
\begin{align*}
\nabla_i^\perp(\partial_t  H_j-\nabla_j V)
&=-\nabla_l^\perp V\nabla_l \nabla_i^\perp H_j
 +\nabla_l\nabla_i^\perp V  \nabla_l^\perp  H_j \\
&=\nabla_l V\nabla_l^\perp \nabla_i^\perp H_j
 +\nabla_l\nabla_i^\perp V  \nabla_l^\perp  H_j \\
&=\nabla_i^\perp  (\nabla_l V\nabla_l^\perp H_j),
\end{align*}
which infers
\begin{align*}
\partial_t  H_j-\nabla_j V
=\nabla_l V\nabla_l^\perp H_j.
\end{align*}
Thus the second equation of \eqref{VisElas3} is obtained.

For \eqref{constr1}, the same substitution gives
\begin{align*}
\nabla_i^\perp(\nabla^\perp\cdot H )
=\nabla_l^\perp H_2  \nabla_l \nabla_i^\perp H_1
-\nabla_l^\perp H_1  \nabla_l \nabla_i^\perp H_2.
\end{align*}
By the identity
$$-\nabla^\perp_l H_1 \nabla_l H_2=\nabla_l H_1 \nabla_l^\perp H_2, $$
we deduce that
\begin{align*}
\nabla_i^\perp(\nabla^\perp\cdot H )
&=\nabla_l^\perp H_2  \nabla_l \nabla_i^\perp H_1
-\nabla_l^\perp H_1  \nabla_l \nabla_i^\perp H_2 \\
&=\nabla_l^\perp H_2  \nabla_l \nabla_i^\perp H_1
+\nabla_l H_1  \nabla_l^\perp \nabla_i^\perp H_2 \\
&=\nabla_i^\perp(\nabla_l^\perp H_2  \nabla_l  H_1),
\end{align*}
which infers \eqref{constr}.

In all the above argument,  the calculation can be reversed  if the solution has enough regularity.
Hence \eqref{VisElas1} and \eqref{constr1} are equivalent to \eqref{VisElas3} and \eqref{constr} for classical solutions.
\end{proof}

\subsection{Commutation Properties}
Now let us take a look at the various vector fields which play a central role in the proofs.
Since the application of  the vector field theory is now classical, we will give a
sketch of rough ideas, and indicate the differences with the classical theory.
For related discussions, we refer to \cite{Kessenich,Lei16,Sideris00}.

The scaling operator is defined by
\begin{equation*}
S = t\partial_t +r\partial_r.
\end{equation*}
Here, due to the scaling of $V$ and $H$,  we will use the modified scaling operator $\widetilde{S}$
which is defined as:
\begin{equation*}
\widetilde{S} = S-1.
\end{equation*}
Applying  the scaling operator $ S$ to \eqref{VisElas3} and \eqref{constr},  we get
\begin{equation}\label{Eq-Scale1}
\begin{cases}
\partial_t\widetilde{S}V-\mu\Delta (\widetilde{S}-1)V-\nabla \cdot \widetilde{S}H \\
\quad=\nabla^\perp\cdot
\nabla\cdot \Delta^{-1}\big(-\nabla^\perp\widetilde{S}V\otimes\nabla^\perp V+\nabla^\perp \widetilde{S} H\otimes \nabla^\perp H \big)\\
\qquad +\nabla^\perp\cdot
\nabla\cdot \Delta^{-1}\big(-\nabla^\perp V\otimes\nabla^\perp\widetilde{S}V+\nabla^\perp H\otimes \nabla^\perp \widetilde{S} H \big), \\
\partial_t \widetilde{S}H -\nabla \widetilde{S}V=
\nabla^\perp \widetilde{S}H\nabla V
+ \nabla^\perp H\nabla \widetilde{S}V,
\end{cases}
\end{equation}
and
\begin{equation} \label{Eq-Scale2}
\nabla^\perp\cdot \widetilde{S} H
=\nabla^\perp \widetilde{S} H_2 \cdot \nabla H_1
+\nabla^\perp H_2 \cdot \nabla \widetilde{S} H_1.
\end{equation}
We see from the above expressions that, when $\mu=0$, the modified scaling operator
commutates well with the inviscid systems, but when $\mu>0$, there is an
extra term $-1$ coming from the commutation between the viscosity term and the scaling
operator. This extra commutator term is troublesome and requires some extra care.

In the 2D case, the rotation operator is defined by:
\begin{equation*}
\Omega =x^\perp\cdot\nabla=\partial_\theta.
\end{equation*}
Applying the rotation operator to \eqref{VisElas3} and \eqref{constr}, we get
\begin{equation}
\begin{cases}\label{Eq-Rotation1}
\partial_t\widetilde{\Omega}V
-\mu\Delta\widetilde{\Omega}V-\nabla \cdot \widetilde{\Omega}H
 \\
\quad=\nabla^\perp\cdot
\nabla\cdot \Delta^{-1}\big(-\nabla^\perp\widetilde{\Omega}V\otimes\nabla^\perp V+\nabla^\perp \widetilde{\Omega} H\otimes \nabla^\perp H \big)\\
\qquad +\nabla^\perp\cdot
\nabla\cdot \Delta^{-1}\big(-\nabla^\perp V\otimes\nabla^\perp\widetilde{\Omega}V+\nabla^\perp H\otimes \nabla^\perp \widetilde{\Omega} H \big), \\
\partial_t \widetilde{\Omega}H -\nabla \widetilde{\Omega}V
=\nabla^\perp \widetilde{\Omega}H\nabla V
+ \nabla^\perp H\nabla \widetilde{\Omega}V,
\end{cases}
\end{equation}
and
\begin{equation}\label{Eq-Rotation2}
\nabla^\perp\cdot \widetilde{\Omega} H
=\nabla^\perp \widetilde{\Omega} H_2 \cdot \nabla H_1
+\nabla^\perp H_2 \cdot \nabla \widetilde{\Omega} H_1,
\end{equation}
where
\begin{equation}
\begin{cases}
\widetilde\Omega V =\Omega V , \\
\widetilde\Omega H= \Omega H-H^\perp. \nonumber
\end{cases}
\end{equation}
Hence, the rotation operator commutates well with the system.
In view of this, we will  separate the scaling operator from regular derivatives and rotation operator:
Let $\Gamma$ be any of the following differential operators
\begin{equation*}
\Gamma \in \{\partial_t,\partial_1,\partial_2,
\widetilde{\Omega}\}.
\end{equation*}
Following the above arguments,
repeatedly using \eqref{Eq-Scale1}, \eqref{Eq-Scale2} and \eqref{Eq-Rotation1}, \eqref{Eq-Rotation2},
we have
\begin{equation}\label{VisElas-Gamma0}
\begin{cases}
\partial_t\widetilde{S}^\alpha\Gamma^aV
-\mu\Delta\sum\limits_{l=0}^\alpha C_\alpha^l (-1)^{\alpha-l} \widetilde{S}^l\Gamma^aV
-\nabla\cdot\widetilde{S}^\alpha\Gamma^a H=f^1_{\alpha a} ,\\
\partial_t\widetilde{S}^\alpha\Gamma^a H -\nabla\widetilde{S}^\alpha\Gamma^aV=f^2_{\alpha a},
\end{cases}
\end{equation}
and
\begin{equation}\label{constr-Gamma0}
\nabla^\perp\cdot\widetilde{S}^\alpha\Gamma^a H=f^3_{\alpha a},
\end{equation}
where
\begin{equation}\label{VisElas-Gamma-f123}
\begin{cases}
f^1_{\alpha a}
=\sum\limits_{\tiny\begin{matrix}b+c=a\\ \beta+\gamma=\alpha\end{matrix}}
C_{\alpha}^\beta C_a^b
\nabla^\perp\cdot\nabla\cdot \Delta^{-1}
\big(-\nabla^\perp \widetilde{S}^\beta\Gamma^b V\otimes\nabla^\perp\widetilde{S}^\gamma\Gamma^c V\\
\quad\quad\quad\quad\quad\quad +\ \nabla^\perp \widetilde{S}^\beta\Gamma^b H\otimes
\nabla^\perp \widetilde{S}^\gamma\Gamma^c H \big),\\[-4mm]\\
f^2_{\alpha a}=\sum\limits_{\tiny\begin{matrix}b+c=a\\ \beta+\gamma=\alpha\end{matrix}}
C_{\alpha}^\beta C_a^b
(\nabla^\perp \widetilde{S}^\beta\Gamma^b H\nabla \widetilde{S}^\gamma\Gamma^c V),\\
f^3_{\alpha a}=\sum\limits_{\tiny\begin{matrix}b+c=a\\ \beta+\gamma=\alpha\end{matrix}}
C_{\alpha}^\beta C_a^b
(\nabla^\perp \widetilde{S}^\beta\Gamma^b H_2 \cdot \nabla\widetilde{S}^\gamma\Gamma^c H_1).
\end{cases}
\end{equation}
Here $\alpha \in \mathbb{N}$ and  $\Gamma^a$ stands for  $\Gamma^a=\Gamma^{a_1}...\Gamma^{a_4}$, where
$a$ is multi-index $a=(a_1,a_2,a_3,a_4)\in \mathbb{N}^4$. We indicate that the generalized vector field $Z$ used in section 1.2 refers to
\begin{equation*}
Z \in \{\partial_t,\partial_1,\partial_2,
\widetilde{\Omega},\widetilde{S}\}.
\end{equation*}
 We also use the abbreviation $\Gamma^k V=\{\Gamma^a V: |a|\leq k \}$ and $\Gamma^k H=\{\Gamma^a H: |a|\leq k \}$.
The binomial coefficient $C_a^b$ is given by
$$C_a^b = \frac{a!}{b!(a - b)!}.$$
We remark that the above commutation relation \eqref{VisElas-Gamma0}-\eqref{VisElas-Gamma-f123} is essential in all of the subsequent argument.
Schematically, we write the following commutation relationship
\begin{equation*}
[\Gamma,\Gamma]=\partial,\quad [\Gamma, S]=\partial.
\end{equation*}
This fact is frequently implicitly used through the whole argument.

In order to simplify the presentation, we abbreviate
$\widetilde{S}^\alpha\Gamma^a V$ as $V^{(\alpha,a)}$  and
abbreviate $\widetilde{S}^\alpha\Gamma^a H$ as $H^{(\alpha,a)}$.
Also, we denote $V^{(\alpha,|a|)}=\{ V^{(\alpha,b)}: |b|\leq |a| \}$,
$H^{(\alpha,|a|)}=\{ H^{(\alpha,b)}: |b|\leq |a| \}$.
Thus \eqref{VisElas-Gamma0} and \eqref{constr-Gamma0} can be written as
\begin{equation}\label{VisElas-Gamma}
\begin{cases}
\partial_tV^{(\alpha, a)}
-\mu\Delta\sum\limits_{l=0}^\alpha C_\alpha^l (-1)^{\alpha-l}V^{(l, a)}
-\nabla\cdot H^{(\alpha, a)}=f^1_{\alpha a} ,\\
\partial_t H^{(\alpha, a)} -\nabla V^{(\alpha, a)}=f^2_{\alpha a}
\end{cases}
\end{equation}
and
\begin{equation}\label{constr-Gamma}
\nabla^\perp\cdot H^{(\alpha, a)}=f^3_{\alpha a}.
\end{equation}

We will also use the notation $  f^{\alpha a}_{ij}  $  to denote
\begin{equation}\label{fij}
f^{\alpha a}_{ij}=\sum\limits_{\tiny\begin{matrix}b+c=a\\ \beta+\gamma=\alpha\end{matrix}}C_{\alpha}^\beta C_a^b
(\partial_iV^{(\beta,b)} \partial_jV^{(\gamma,c)}
-\partial_iH^{(\beta,b)}\cdot \partial_jH^{(\gamma,c)}),
\end{equation}
where $1\leq i,j \leq 2$.  Hence, $ f_{\alpha a}^{1}   =    R^\perp_i  R_j    f_{\alpha a}^{ij} $
where  $ R^\perp_i  R_j  =   \nabla^\perp_i   \nabla_j    \Delta^{-1}.$

\subsection{Some Notations}
Now we explain some important concepts and notations used through the paper.
The spatial derivatives can be decomposed into radial and angular components:
\begin{equation}
\label{der-decomp}
\nabla=\omega\partial_r+\frac{\omega^\perp}{r}\partial_\theta,
\end{equation}
where $\partial_r=\omega\cdot\nabla, \partial_\theta=x^\perp \cdot \nabla$.
This fact plays an important role in the following argument.

We will use Klainerman's generalized energy which is defined, for $\kappa \geq 1$,  by
\begin{equation*}\label{Energy}
E_\kappa(t) = \sum_{|\alpha|+|a| \leq \kappa}
\|U^{(\alpha,a)}(t,\cdot)\|_{L^2}^2,
\end{equation*}
where $U=(V,H)$. Moreover, we introduce the modified generalized energy:
\begin{equation*}\label{EnergyA}
\mathcal{E}_\kappa(t) = \sum_{|\alpha|+|a|+1 \leq \kappa}
\|\nabla U^{(\alpha,a)}(t,\cdot)\|_{L^2}^2.
\end{equation*}
Here the word ``modified generalized energy"  is used to
insist on the fact that one of the derivatives has to be a regular derivative.
 The use  of the modified energy
$\mathcal{E}_{\kappa}$  at the highest derivative level is imposed by the ghost weight method
and will  lead to some difficulties.

We also use the weighted energy norm of Klainerman-Sideris \cite{KlainermanSiderisS96}:
\begin{equation*}\label{WEnergy}
X_\kappa(t) = \sum_{|\alpha|+|a|+1 \leq \kappa}
\| \langle r-t\rangle \nabla U^{(\alpha,a)} \|^2_{L^2},
\end{equation*}
in which we denote $\langle \sigma\rangle = \sqrt{1 + \sigma^2}$.

In addition, we introduce a new weighted energy for good quantities $V + H\cdot\omega$ and $H\cdot\omega^\perp$:
\begin{equation*}
Y_\kappa(t) = \sum_{|\alpha|+|a|+1 \leq \kappa}\big(\| r (\partial_rV^{(\alpha,a)}
+\partial_rH^{(\alpha,a)}\cdot\omega) \|^2_{L^2}
+\| r \partial_rH^{(\alpha,a)}\cdot\omega^\perp \|^2_{L^2}\big).
\end{equation*}
The weighted energy $Y_\kappa$ is used to describe the good decay properties of the good unknowns
$V+ H\cdot\omega$ and $H\cdot\omega^\perp$ near the light cone.
We emphasize that we need  to treat the derivative loss in  the sequel when estimating $X_\kappa$ and $Y_\kappa$.

To describe the space of  initial data, we introduce (see \cite{Sideris00})
\begin{equation}\nonumber
\Lambda = \{\nabla, \widetilde\Omega,r\partial_r-1\},
\end{equation}
and
\[
H^\kappa_\Lambda =\{(f,g):\sum_{|a|\le \kappa}\|\Lambda^a f\|_{L^2}+\|\Lambda^a g\|_{L^2}<\infty\},
\]
with the norm
\[
\|(f,g)\|_{H^\kappa_\Lambda} =\sum_{|a|\le \kappa}\big(\|\Lambda^a f\|_{L^2}+\|\Lambda^a g\|_{L^2}\big),
\]
for scalar, vector or matrix function $f$ and $g$.
Then as in \cite{Sideris00}, we define
\[
{H}^\kappa_\Gamma(T)=\{(f,g):[0,T)\to \mathbb{R}\times\mathbb{R}^2, (f,g)\in\cap_{j=0}^\kappa C^j([0,T);H^{\kappa-j}_\Lambda)\}.
\]
Solutions will be constructed in the space ${H}^\kappa_\Gamma(T)$.

Throughout the whole paper, we will use $A\lesssim B$ to denote $A\leq C B$ for some positive absolute constant $C$,
whose meaning may change from line to line. We remark that, without specification, the constant only depends on $\kappa$, but never depends on $\mu$ or $t$.

For the global existence result, we will establish the following \textit{a priori} estimate
\begin{align}\label{PriorEsti1}
&\mathcal{E}_{\kappa}(t) + E_{\kappa-1}(t)
+\sum_{|\alpha|+|a|\leq \kappa-1}
\mu \int_0^t\int_{\mathbb{R}^2} |\Delta V^{(\alpha,a)}(\tau)|^2
+ |\nabla V^{(\alpha,a)}(\tau)|^2 dxd\tau\nonumber\\
&\lesssim
\int_0^t \langle \tau \rangle ^{-1}(\mathcal{E}_\kappa(\tau)+E_{\kappa-1}(\tau)) E_{\kappa-3}^{\frac 12}(\tau)d\tau
+\mathcal{E}_{\kappa}(0)+E_{\kappa-1}(0) ,
\end{align}
and
\begin{align}\label{PriorEsti2}
&E_{\kappa-3}(t)+ \sum_{|\alpha| +|a| \leq \kappa-3}
\mu\int_0^t\int_{\mathbb{R}^2} |\nabla V^{(\alpha,a)}(\tau)|^2 dxd\tau  \nonumber\\
&\lesssim E_{\kappa-3}(0)+\int_0^t \langle \tau\rangle^{-\frac 32}
E_{\kappa-3}(\tau) E_{\kappa-1}^{\frac 12}(\tau) d\tau,
\end{align}
for $\kappa\geq 12$. Once the above estimates are obtained, the main result holds by a  standard continuity method.
For the details, one can consult the differential version in \cite{Lei16}.

So from now on,  our main goal is to prove the two \textit{a priori}
estimates \eqref{PriorEsti1} and \eqref{PriorEsti2}.
In Theorem \ref{thm}, by taking appropriately large $C_0$ and small $\gamma$, we can
assume that $E_{k-3}\ll 1$, which is always assumed in the following argument.

\subsection{Energy Estimate Scenario}
To see the underlying ideas more clearly in these long computations, we sketch the energy
estimates in various scenarios as follows.

\textbf{The modified energy estimate}:
\begin{align*}
&\mathcal{E}_\kappa+\underbrace{\mathcal{D}_{\kappa+1}}_{
\textrm{modified dissipative energy}}+\underbrace{G_{\kappa}}_{
\textrm{ghost weight energy}}
+\underbrace{Linear Commutator_1}_{
\textrm{due to viscosity and scaling operator}}
\nonumber\\[-4mm] \nonumber\\
&
+\underbrace{LinearCommutator_2}_{
\textrm{due to viscosity and ghost weight}}
\lesssim C+{Nonlinear-terms}_1.
\end{align*}
$LinearCommutator_1$: it can be absorbed by $\mathcal{D}_{\kappa+1}$.\\
$LinearCommutator_2$: is absorbed by $\mathcal{D}_{\kappa+1}$ and $D_{\kappa}$ (note that
$D_{\kappa}$ will in \textbf{the standard higher order energy estimate} $E_{\kappa-1}$).\\
${Nonlinear-terms}_1$: derivative loss problems are present due the highly fully nonlinear
effect, nonlocal effect and the use of ghost weights. After a long, delicate integration by
parts procedure, one can save one derivative.  It is important that the null condition is
satisfied. Thus one can continue to employ $G_k$ to improve the decay rate to the critical
rate. Here one needs also to take care of the derivative loss in dealing with
$X_{\kappa-1}+Y_{\kappa-1}\lesssim E_{\kappa}$.

\textbf{The standard higher order energy estimate}:
\begin{align*}
&E_{\kappa-1}+\underbrace{D_{\kappa}}_{
\textrm{dissipative energy}}
+\underbrace{LinearCommutator_3}_{
\textrm{due to viscosity and scaling operator}}
\nonumber\\[-4mm] \nonumber\\
&\lesssim C+ {Nonlinear-terms}_2.
\end{align*}
$LinearCommutator_3$: absorbed by $D_{\kappa}$.
\\
${Nonlinear-terms}_2$:  $G_\kappa$ is used to improve the decay (note $G_\kappa$ has been
used in the
\textbf{modified energy estimate} $\mathcal{E}_{\kappa}$). One also takes care of the
derivative loss in derivations of $X_{\kappa-1}+Y_{\kappa-1}\lesssim E_{\kappa}$.

\textbf{The lower order energy estimate}:
\begin{align*}
&E_{\kappa-3}+\underbrace{D_{\kappa-2}}_{
\textrm{dissipative energy}}
+\underbrace{LinearCommutator_4}_{
\textrm{due to viscosity and scaling operator}}
\nonumber\\[-4mm] \nonumber\\
&\lesssim \epsilon^2+{Nonlinear-terms}_3.
\end{align*}
$LinearCommutator_4$: absorbed by $D_{\kappa-2}$.\\
${Nonlinear-terms}_3$: by a strong null condition to improve the decays.\\

\section{Estimates for the Special Quantities}
In this section, we are going to estimate the weighted
$L^2$ energy $X_{\kappa}$ and $Y_{\kappa}$. The weighted energy $X_\kappa$ was first introduced by Klainerman and Sideris \cite{KlainermanSiderisS96} for proving almost global solutions of three-dimensional quasilinear wave equations and later on used in \cite{LSZ13} for proving almost global  existence for the  two-dimensional incompressible elastodynamic system.
The energy $Y_\kappa$ is new and used here to estimate the good unknowns. Due to the fact
that $r$ is not an $\mathcal{A}_2$ weight in two dimension, the a modified one is introduced
in \cite{Lei16} to get the global solution of two-dimensional incompressible elastodynamics.
Here by transforming the original quasilinear system \eqref{VisElas1} into a fully nonlinear one
\eqref{VisElas3}, we can simply use the earlier ones introduced by Klainerman-Sideris for
$X_\kappa$.  This advantage is based on the inherent structure of the system, which enables
us to simplify the proofs a lot.

\subsection{Sobolev-type Inequalities}
The following weighted Sobolev-type inequalities will be used to prove the decay of solutions in $L^\infty$ norm.
A much stronger version of \eqref{K-S-3} appeared in \cite{Lei16}. Here since we are able to transform the original system into a ``fully nonlinear" one, the form of \eqref{K-S-3} is enough for us here.
\begin{lem}\label{31}
For all $f\in H^2(\mathbb{R}^2)$, there holds
\begin{align}
\label{K-S-1}
r|f(x)|^2&\lesssim \sum_{a=0,1} \Large[ \|\partial_r \Omega^a f\|_{L^2}^2+\|\Omega^a f\|_{L^2}^2 \Large],\\
\label{K-S-2}
r{\langle t-r\rangle}^2|f(x)|^2&\lesssim \sum_{a=0,1}\Large[\|\langle t-r\rangle \partial_r \Omega^a f\|_{L^2}^2
+\|\langle t-r\rangle\Omega^a f\|_{L^2}^2\Large],\\
\label{K-S-3}
\langle t\rangle \| f\|_{L^\infty(r\leq \langle t\rangle/2)}
&\lesssim \sum_{|a|\le2}\|\langle t-r\rangle\partial^a f\|_{L^2},
\end{align}
provided the right-hand side is finite.
\end{lem}
The proof of this lemma can be found in \cite{LSZ13} (the three-dimensional version can be found in \cite{SiderisThomas07}),
we omit the details here.

\subsection{Estimate of the Good  Quantities}
In this section, we are going to explore the good properties of some special combinations of unknowns.
Both the linearities and the nonlinearities will be investigated.
The exploration of these special quantities is a prerequisite
 for the estimate of weighted $L^2$ energy $X_\kappa$ and $Y_\kappa$. On the other hand,
they are also crucial for the energy estimate which will be conducted
in the subsequent two sections.

In order to simplify the presentation, we first introduce some notations.
Suppose that $(V,H)\in H^\kappa_{\Gamma}$ solves \eqref{VisElas3} and \eqref{constr}.
Define
\begin{eqnarray*}
&&L_\kappa=\sum_{|\alpha|+|a|\leq k}|U^{(\alpha,a)}|,\\
&&N_{\kappa+1}= \sum_{|\alpha|+|a|\leq k}
( t|f^1_{\alpha a}|+t|f^2_{\alpha a}|+(t+r)|f_{\alpha a}^3|),\\
&&\mathcal{N}_{\kappa+2}= \sum_{|\alpha|+|a|\leq k}t|\nabla \cdot f^2_{\alpha a}|,
\end{eqnarray*}
where $L_\kappa$ represents some linear quantity,
$N_{\kappa+1}$ and $\mathcal{N}_{\kappa+2}$ represent some nonlinear quantities.

\begin{remark}
The term  $N_{\kappa+1}$ will be used when we
multiply  the systems \eqref{VisElas-Gamma} and \eqref{constr-Gamma} by some $t$ or $r$ factor.
  The term $\mathcal{N}_{\kappa+2}$  will appear  due to the presence  of viscosity   (see Lemma \ref{lemLinear}).
\end{remark}
\begin{remark}
One can also use a stronger  version of $N_{\kappa+1}$ by defining
\begin{equation*}
N_{\kappa+1}= \sum_{|\alpha|+|a|\leq k}
( t|f^1_{\alpha a}|+(t+r)|f^2_{\alpha a}|+(t+r)|f_{\alpha a}^3|).
\end{equation*}
However, one cannot include $r|f^1_{\alpha a}|$ in $N_{\kappa+1}$ since $r$ is not an  $\mathcal{A}_2$ weight for singular integral in two space dimensions.
\end{remark}

Now we are going to analyze the linear part of the system  and  establish several estimates. Before doing so, we need an elementary iteration lemma.
\begin{lem}(Iteration lemma) \label{lemItera}
Let $\{f_l\},\ \{g_l\},\ \{F_l\}$ be three nonnegative sequences,
where $0\leq l\leq \kappa$.
Suppose that
\begin{equation*}
f_0+g_0\lesssim F_0,
\end{equation*}
and for all $1\leq l\leq \kappa$,
\begin{equation*}
f_l+g_l-g_{l-1}\lesssim F_l.
\end{equation*}
Then there holds
\begin{equation*}
\sum_{0\leq m\leq l}(f_m+g_m)\lesssim \sum_{0\leq m\leq l}F_m,
\end{equation*}
for all $0\leq l\leq \kappa$.
\end{lem}
\begin{remark}
This lemma plays a role in
dealing with the commutators between the viscosity term
and the scaling operator and will be frequently used through the whole paper.
\end{remark}
\begin{proof}
We prove the lemma by induction on $l$.
Obviously, the lemma is correct when $l=0$.
Let $1\leq l \leq \kappa$.
We assume the lemma is correct for $l-1$.
This means that
\begin{equation}\label{Ite1}
\sum_{0\leq m\leq l-1}(f_m+g_m)\leq C\sum_{0\leq m\leq l-1}F_m.
\end{equation}
On the other hand, we have
\begin{equation}\label{Ite2}
f_l+g_l-g_{l-1}\leq CF_l.
\end{equation}
Multiplying \eqref{Ite1} by 2, then adding \eqref{Ite2}, we get
\begin{equation*}
2\sum_{0\leq m\leq l-2}(f_m+g_m)+f_l+g_l+2f_{l-1}+g_{l-1}
\leq 2C\sum_{0\leq m\leq l-1}F_m+CF_l,
\end{equation*}
which is the required estimate for $l$.
Thus the lemma is proved.
\end{proof}

Now we are ready to state two lemmas for the special linear quantities. These two
lemmas are requisite for the estimate of the weighted $L^2$ norm $X_\kappa$ and $Y_\kappa$ .
\begin{lem}\label{lemLinear}
Suppose that $(V,H)\in H^{\kappa-1}_{\Gamma}$ solves \eqref{VisElas3} and \eqref{constr}.
Then for all $|\alpha|+|a|\leq \kappa-3$,
there holds
\begin{align*}
&\|r\partial_rV^{(\alpha,a)}+t\nabla\cdot H^{(\alpha,a)}\|_{L^2}^2
+\|\mu t\Delta V^{(\alpha,a)}\|_{L^2}^2
\nonumber\\[-4mm]\nonumber\\
&\leq \nu X_{|\alpha|+|a|+1}
+C \langle \mu\rangle^2\|L_{|\alpha|+|a|+1} \|^2_{L^2}
+C_{\nu} \mu^2\|L_{|\alpha|+|a|+2} \|^2_{L^2}
\nonumber\\[-4mm]\nonumber\\
&\quad+C \|N_{|\alpha|+|a|+1} \|^2_{L^2}
+C_{\nu} \mu^2\|\mathcal{N}_{|\alpha|+|a|+2} \|^2_{L^2},
\end{align*}
provided the right hand side is finite,
where $\nu$ can be any positive constant,
$C_\nu$ is a constant which depends only on $\alpha, a$ and $\nu$,
$C$ depends only on $\alpha$ and $a$.
\end{lem}
\begin{remark}
While the lemma becomes trivial if there is no viscosity, the   viscous version is nontrivial.
The viscosity is the main reason we only have an $L^2$ bound  rather than a pointwise
bound  as in the next lemma.
\end{remark}

\begin{remark}
The terms on the left hand side are of order $|\alpha|+|a|+1$ except for the viscous term, but
on the right hand side, the order is $|\alpha|+|a|+2$.
This means that we will encounter the problem of losing derivatives in future discussions.
If $\mu=0$, one has no derivative loss problem.
\end{remark}

\begin{proof}
Denote
\begin{equation*}
J=\|r\partial_rV^{(\alpha,a)}+t\nabla\cdot H^{(\alpha,a)}\|_{L^2}^2
+\sum\limits_{l=0}^\alpha  (C_\alpha^l)^2\|\mu t\Delta V^{(l,a)}\|_{L^2}^2.
\end{equation*}
We first claim that there holds the following fact:
\begin{align}\label{ll}
J&\leq 6\sum\limits_{l=0}^{\alpha-1} (C_\alpha^l)^2\|\mu t\Delta V^{(l, a)}\|^2_{L^2}
+\nu\|\langle t-r\rangle \nabla\cdot H^{(\alpha,a)}\|_{L^2}^2  \nonumber\\
&\quad+C \langle \mu\rangle^2\|L_{|\alpha|+|a|+1} \|^2_{L^2}
+C_{\nu} \mu^2\|L_{|\alpha|+|a|+2} \|^2_{L^2}
\nonumber\\[-4mm]\nonumber\\
&\quad+C \|N_{|\alpha|+|a|+1} \|^2_{L^2}
+C_{\nu} \mu^2\|\mathcal{N}_{|\alpha|+|a|+2} \|^2_{L^2},
\end{align}
where $\nu$ can be any positive constant,
$C_\nu$ is a constant depending only on $\alpha, a$ and $\nu$.

Once this assertion \eqref{ll} becomes true, note the assumption $\mu\leq1$,
one can immediately see that Lemma \ref{lemLinear}
is proved by applying Lemma \ref{lemItera} to \eqref{ll}.
Thus it suffices to prove  \eqref{ll}.

Multiplying the first equation of \eqref{VisElas-Gamma} by $t$
and using the scaling operator, one gets
\begin{equation*}
r\partial_rV^{(\alpha,a)}+t\nabla\cdot H^{(\alpha,a)}
+\mu t\Delta\sum\limits_{l=0}^\alpha C_\alpha^l (-1)^{\alpha-l}V^{(l, a)}
=SV^{(\alpha,a)}-tf^1_{\alpha a}.
\end{equation*}
Taking $L^2$ norm for the above equation, one has
\begin{align} \label{ll0}
J&=-2\int_{\mathbb{R}^2}
(r\partial_rV^{(\alpha,a)}+t\nabla\cdot H^{(\alpha,a)})\cdot
\mu t\Delta\sum\limits_{l=0}^{\alpha} C_\alpha^l (-1)^{\alpha-l}V^{(l, a)}dx \nonumber\\
&\quad+\|SV^{(\alpha,a)}-tf^1_{\alpha a}\|^2_{L^2}.
\end{align}
To prove \eqref{ll}, we need to  deal with
the right hand of \eqref{ll0}.

By separating the highest order terms from the lower order ones,
\eqref{ll0} can be organized as:
\begin{align} \label{ll1}
J&=\underbrace{-2\int_{\mathbb{R}^2}
(r\partial_rV^{(\alpha,a)}+t\nabla\cdot H^{(\alpha,a)})\cdot
\mu t\Delta\sum\limits_{l=0}^{\alpha-1} C_\alpha^l (-1)^{\alpha-l}V^{(l, a)}dx
+\|SV^{(\alpha,a)}-tf^1_{\alpha a}\|^2_{L^2}}_{J_3}\nonumber\\
&\quad\underbrace{-2\int_{\mathbb{R}^2} r\partial_rV^{(\alpha,a)}\cdot\mu t\Delta V^{(\alpha,a)}dx}_{J_2}
\underbrace{-2\int_{\mathbb{R}^2} t\nabla\cdot H^{(\alpha,a)}\cdot\mu t\Delta V^{(\alpha,a)}dx}_{J_1}.
\end{align}
Here $J_3$ refers to the lower order term,
$J_1$ and $J_2$ refer to the highest order terms.

By H\"{o}lder's inequality, $J_3$
can be estimated by
\begin{eqnarray*}
&&\frac12\|r\partial_rV^{(\alpha,a)}+t\nabla\cdot H^{(\alpha,a)}\|^2_{L^2}
+2\sum\limits_{l=0}^{\alpha-1} (C_\alpha^l)^2 \|\mu t\Delta V^{(l, a)}\|^2_{L^2}\\\nonumber
&&\quad\quad\quad\quad\ +\ 2\|SV^{(\alpha,a)}\|^2_{L^2}+2\|tf^1_{\alpha a}\|^2_{L^2}.
\end{eqnarray*}
For $J_2$, one can deduce from integration by parts to get that
\begin{equation*}
J_2=-2\int_{\mathbb{R}^2} r\partial_rV^{(\alpha,a)}\cdot\mu t\Delta V^{(\alpha,a)}dx=0.
\end{equation*}

It remains to estimate $J_1$, which can not be treated simply by Cauchy inequality due to an extra $t$-factor.
We will refer to the inherent structure of the systems.

Applying  divergence operator to the second equation of \eqref{VisElas-Gamma},
one gets
\begin{equation*}
\Delta V^{(\alpha,a)}=
\partial_t\nabla\cdot H^{(\alpha,a)}-\nabla\cdot f^2_{\alpha a}.
\end{equation*}
Inserting the above expression into $J_1$,
and employing the scaling operator, we have
\begin{align*}
J_1&=-2\int_{\mathbb{R}^2} t\nabla\cdot H^{(\alpha,a)}
\cdot \mu(t\partial_t\nabla\cdot H^{(\alpha,a)}-t\nabla\cdot f^2_{\alpha a} )dx\\
&=-2\int_{\mathbb{R}^2} t\nabla\cdot H^{(\alpha,a)}
\cdot\mu(-r\partial_r\nabla\cdot H^{(\alpha,a)}+S\nabla\cdot H^{(\alpha,a)}
-t\nabla\cdot f^2_{\alpha a})dx.
\end{align*}
In view of the fact that
\begin{align*}
2\int_{\mathbb{R}^2} t\nabla\cdot H^{(\alpha,a)}
\cdot\mu r\partial_r\nabla\cdot H^{(\alpha,a)}dx
=-2\int_{\mathbb{R}^2} \mu t|\nabla\cdot H^{(\alpha,a)}|^2 dx,
\end{align*}
we get
\begin{align*}
&J_1\leq
-2\int_{\mathbb{R}^2} t\nabla\cdot H^{(\alpha,a)}
\cdot\mu(S\nabla\cdot H^{(\alpha,a)}-t\nabla\cdot f^2_{\alpha a})dx.
\end{align*}
Now we need to estimate the integral in different regions, separately.
To do this, define a radial cut-off function $\varphi\in C^\infty(\mathbb{R}^2)$ which satisfies
\begin{equation}\nonumber
\varphi = \begin{cases}1,\quad {\rm if}\ \frac{3}{4} \leq r \leq \frac{6}{5}\\
0, \quad {\rm if}\ r < \frac{2}{3}\ {\rm or}\ r >
\frac{5}{4}\end{cases},\quad |\nabla\varphi| \lesssim 1.
\end{equation}
For each fixed $t \geq 1$, let $\varphi^t(x) = \varphi(x/\langle
t\rangle)$. Clearly, one has
$$\varphi^t(x) \equiv 1\ \ {\rm for}\ \frac{3\langle t \rangle}{4} \leq r \leq
\frac{6\langle t \rangle}{5},\quad \varphi^t(x) \equiv 0\ \ {\rm
for}\ r \leq \frac{2\langle t \rangle}{3}\ {\rm or}\ r \geq
\frac{5\langle t \rangle}{4}$$ and $$|\nabla\varphi^t(x)| \lesssim
\langle t\rangle^{-1}.$$
Consequently,
\begin{align*} 
J_1&\leq-2\int_{\mathbb{R}^2} t\nabla\cdot H^{(\alpha,a)}
\cdot\mu(S\nabla\cdot H^{(\alpha,a)}-t\nabla\cdot f^2_{\alpha a})dx \nonumber\\
&=\underbrace{-2\int_{\mathbb{R}^2} (1-\varphi^t(x)) t\nabla\cdot H^{(\alpha,a)}
\cdot\mu(S\nabla\cdot H^{(\alpha,a)}-t\nabla\cdot f^2_{\alpha a})dx}_{J_{11}} \nonumber\\
&\quad \underbrace{-2\int_{\mathbb{R}^2} \varphi^t(x) t\nabla\cdot H^{(\alpha,a)}
\cdot\mu(S\nabla\cdot H^{(\alpha,a)}-t\nabla\cdot f^2_{\alpha a})dx}_{J_{12}}.
\end{align*}
We now estimate $J_{11}$. Note on the support of $1-\varphi^t(x)$,
we have $t\lesssim \langle t-r\rangle $.
Thus one can estimate $J_{11}$ as follows:
\begin{equation*}
J_{11}\leq \nu\|\langle t-r\rangle \nabla\cdot H^{(\alpha,a)}\|_{L^2}^2
+\mu^2 C_{\nu}\big( \|S\nabla\cdot H^{(\alpha,a)}\|^2_{L^2}
+\|t\nabla\cdot f^2_{\alpha a}\|^2_{L^2} \big),
\end{equation*}
where $\nu$ can be any positive constant,
$C_\nu$ is a constant depending on $\nu$.

For $J_{12}$, employing the first equation of \eqref{VisElas-Gamma},
we have
\begin{align*}
J_{12}&=-2\int_{\mathbb{R}^2} \varphi^t(x) t
\big[\partial_tV^{(\alpha,a)}
-\mu\Delta\sum\limits_{l=0}^\alpha C_\alpha^l (-1)^{\alpha-l}V^{(l, a)}
-f^1_{\alpha a}\big]\\
&\quad\quad\quad\quad \cdot\mu(S\nabla\cdot H^{(\alpha,a)}-t\nabla\cdot f^2_{\alpha a})dx\\
&=\underbrace{-2\int_{\mathbb{R}^2} \varphi^t(x) t
\partial_tV^{(\alpha,a)}
\cdot\mu(S\nabla\cdot H^{(\alpha,a)}-t\nabla\cdot f^2_{\alpha a})dx}_{J_{121}}\\
&\quad+\underbrace{2\int_{\mathbb{R}^2} \varphi^t(x)
(\mu t\Delta\sum\limits_{l=0}^\alpha C_\alpha^l (-1)^{\alpha-l}V^{(l, a)}+tf^1_{\alpha a})
\cdot\mu(S\nabla\cdot H^{(\alpha,a)}-t\nabla\cdot f^2_{\alpha a})dx}_{J_{122}}.
\end{align*}
$J_{122}$ can be directly bounded as:
\begin{align*}
J_{122}&\leq\frac 14\sum\limits_{l=0}^\alpha (C_\alpha^l)^2 \|\mu t\Delta V^{(l, a)}\|^2_{L^2}
+5\mu^2\|S\nabla\cdot H^{(\alpha,a)}-t\nabla\cdot f^2_{\alpha a}\|^2_{L^2}
+\|t f^1_{ \alpha a}\|^2_{L^2}\\
&\leq\frac 14\sum\limits_{l=0}^\alpha (C_\alpha^l)^2 \|\mu t\Delta V^{(l, a)}\|^2_{L^2}
+10\mu^2 \|S\nabla\cdot H^{(\alpha,a)}\|_{L^2}\\
&\quad\quad\quad +\ 10\mu^2 \|t \nabla\cdot f^2_{ \alpha a}\|^2_{L^2}
+\|t f^1_{\alpha a}\|^2_{L^2}.
\end{align*}
At last, we write
\begin{align*}
J_{121}&=2\int_{\mathbb{R}^2} \varphi^t(x) \mu (r\partial_rV^{(\alpha,a)} -SV^{(\alpha,a)})
\cdot(S\nabla\cdot H^{(\alpha,a)}-t\nabla\cdot f^2_{\alpha a})dx \\
&=\underbrace{2\int_{\mathbb{R}^2} \varphi^t(x) \mu r\partial_rV^{(\alpha,a)}
\cdot(S\nabla\cdot H^{(\alpha,a)}-t\nabla\cdot f^2_{\alpha a})dx}_{J_{1211}} \\
&\quad\underbrace{-2\int_{\mathbb{R}^2} \varphi^t(x) \mu SV^{(\alpha,a)}
\cdot(S\nabla\cdot H^{(\alpha,a)}-t\nabla\cdot f^2_{\alpha a})dx}_{J_{1212}} .
\end{align*}
$J_{1212}$ can be bounded by
\begin{equation*}
J_{1212}\leq
2 \|SV^{(\alpha,a)}\|^2_{L^2}
+\mu^2\|S\nabla\cdot H^{(\alpha,a)}\|^2_{L^2}+\mu^2\|t\nabla\cdot f^2_{\alpha a}\|^2_{L^2}.
\end{equation*}
For $J_{1211}$, note on the support of $\varphi^t(x)$, we have $\langle t\rangle\sim r $.
Hence one deduces that
\begin{align*}
J_{1211}
&=2\int_{\mathbb{R}^2} \varphi^t(x) \mu r\partial_rV^{(\alpha,a)}
\cdot(\nabla \cdot \widetilde{S}H^{(\alpha,a)}-t\nabla\cdot f^2_{\alpha a})dx\\
&=-2\int_{\mathbb{R}^2} \nabla \big(\varphi^t(x) \mu r\partial_rV^{(\alpha,a)}\big)
\cdot(\widetilde{S}H^{(\alpha,a)}-t f^2_{\alpha a})dx\\
&\leq2\int_{\mathbb{R}^2} |\nabla\varphi^t(x) \mu r\partial_rV^{(\alpha,a)}
\cdot(\widetilde{S}H^{(\alpha,a)}-t f^2_{\alpha a})|dx\\
&\quad +2\int_{\mathbb{R}^2} \varphi^t(x) \mu |\nabla V^{(\alpha,a)}|
\cdot|\widetilde{S}H^{(\alpha,a)}-t f^2_{\alpha a}|dx\\
&\quad +2\int_{\mathbb{R}^2} \varphi^t(x) \mu r \partial_r\nabla V^{(\alpha,a)}
\cdot(\widetilde{S}H^{(\alpha,a)}-t f^2_{\alpha a})dx\\
&\leq \mu^2\|\nabla V^{(\alpha,a)}\|^2_{L^2}
 +\frac 14\|\mu t \nabla^2 V^{(\alpha,a)}\|^2_{L^2}
 +C\|\widetilde{S} H^{(\alpha,a)}-t f^2_{\alpha a}\|^2_{L^2}.
\end{align*}
Combing all the above estimates together, we conclude
by the commutation between the generalized operators that
\begin{align*}
&J=\|r\partial_rV^{(\alpha,a)}+t\nabla\cdot H^{(\alpha,a)}\|_{L^2}^2
+\sum_{l=0}^\alpha (C_\alpha^l)^2 \|\mu t\Delta V^{(l,a)}\|_{L^2}^2\\
&\leq \frac12\|r\partial_rV^{(\alpha,a)}+t\nabla\cdot H^{(\alpha,a)}\|^2_{L^2}
+\frac 12 \|\mu t\Delta V^{(\alpha, a)}\|^2_{L^2}
+3\sum\limits_{l=0}^{\alpha-1} (C_\alpha^l)^2 \|\mu t\Delta V^{(l, a)}\|^2_{L^2} \\
&+\nu\|\langle t-r\rangle \nabla\cdot H^{(\alpha,a)}\|_{L^2}^2
+C \langle \mu\rangle^2\|L_{|\alpha|+|a|+1} \|^2_{L^2}
+C_{\nu} \mu^2\|L_{|\alpha|+|a|+2} \|^2_{L^2}
\nonumber\\[-4mm]\nonumber\\
&+C \|N_{|\alpha|+|a|+1} \|^2_{L^2}
+C_{\nu} \mu^2\|\mathcal{N}_{|\alpha|+|a|+2} \|^2_{L^2}.
\end{align*}
Absorbing the first two terms on the right hand side in the above
yields \eqref{ll}.
Thus the lemma is proved.
\end{proof}

We have the following pointwise estimates.
\begin{lem}\label{Pointwise-Estimate}
Suppose that $(V,H)\in H^{\kappa-1}_{\Gamma}$ solves \eqref{VisElas3} and \eqref{constr}.
Then for all $|\alpha|+|a|\leq \kappa-2$, there holds
\begin{eqnarray}
&&|r(\nabla\cdot H^{(\alpha,a)})\omega
+t\nabla V^{(\alpha,a)}| \lesssim
L_{|\alpha|+|a|+1}+N_{|\alpha|+|a|+1}, \label{LL2}\\
&&|(t\pm r)(\nabla V^{(\alpha,a)}
\pm \nabla\cdot H^{(\alpha,a)}\omega)|  \nonumber\\
&&\quad \lesssim
L_{|\alpha|+|a|+1}+N_{|\alpha|+|a|+1}+ |r\partial_r V^{(\alpha,a)}
+t\nabla\cdot H^{(\alpha,a)}|, \label{LL3}\\
&&r|\partial_rH^{(\alpha,a)}\cdot\omega^\perp|
\lesssim L_{|\alpha|+|a|+1}+N_{|\alpha|+|a|+1},  \label{LL4}\\
&&r|\partial_rV^{(\alpha,a)}
+\partial_r H^{(\alpha,a)}\cdot\omega| \nonumber\\
&&\quad \lesssim L_{|\alpha|+|a|+1}+N_{|\alpha|+|a|+1}+|r\partial_r V^{(\alpha,a)}+t\nabla\cdot H^{(\alpha,a)}|. \label{LL5}
\end{eqnarray}
\end{lem}
\begin{proof}
Multiplying the second equation of \eqref{VisElas-Gamma} by $t$ and using the scaling operator,
we can rearrange the resulting systems as follows:
\begin{equation} \label{LEE1}
r\partial_rH^{(\alpha,a)}
+t\nabla V^{(\alpha,a)}
=SH^{(\alpha,a)}-tf^2_{\alpha a}.
\end{equation}
Employing \eqref{der-decomp}, one has
\begin{align}\label{LEE2}
&r\partial_rH^{(\alpha,a)}
+t\nabla V^{(\alpha,a)}\nonumber\\
&=(r\partial_rH^{(\alpha,a)}\cdot \omega )\omega
+(r\partial_rH^{(\alpha,a)}\cdot \omega^\perp)\omega^\perp +t\nabla V^{(\alpha,a)} \nonumber\\
&=(r\nabla\cdot H^{(\alpha,a)})\omega
-(\Omega H^{(\alpha,a)}\cdot\omega^\perp)\omega
+(r\nabla^\perp\cdot H^{(\alpha,a)})\omega^\perp \nonumber\\
&\quad +(\Omega H^{(\alpha,a)}\cdot\omega)\omega^\perp
+t\nabla V^{(\alpha,a)}  \nonumber\\
&=(r\nabla\cdot H^{(\alpha,a)})\omega
+t\nabla V^{(\alpha,a)}
+rf_{\alpha a}^3\omega^\perp \nonumber\\
&\quad -(\Omega H^{(\alpha,a)}\cdot\omega^\perp)\omega
+(\Omega H^{(\alpha,a)}\cdot\omega)\omega^\perp.
\end{align}
In view of the relation between $S$ and $\widetilde{S}$,  \eqref{LL2} is clear from \eqref{LEE1} and \eqref{LEE2}.
Next, note that
\begin{equation}\label{LEE3}
r\nabla V^{(\alpha,a)}
+t(\nabla\cdot H^{(\alpha,a)})\omega
=(r\partial_r V^{(\alpha,a)}
+t\nabla\cdot H^{(\alpha,a)})\omega
+\Omega V\omega^\perp.
\end{equation}
\eqref{LL3} is a direct consequence of \eqref{LL2}
and \eqref{LEE3}.

The estimate of \eqref{LL4} follows directly from \eqref{constr-Gamma} and \eqref{der-decomp}.
To check \eqref{LL5}, by analog with the above proof, we write
\begin{align*}
&r(\partial_rV^{(\alpha,a)}
+\partial_r H^{(\alpha,a)}\cdot\omega)\\
&=\omega\cdot\Large[r\nabla V^{(\alpha,a)}
+(r\partial_r H^{(\alpha,a)}\cdot\omega)\omega
\Large]\\
&=\omega\cdot\Large[r\nabla V^{(\alpha,a)}
+r\nabla\cdot H^{(\alpha,a)}\omega
-(\Omega H^{(\alpha,a)}\cdot\omega^\perp)\omega \Large],
\end{align*}
from which \eqref{LL5} follows from \eqref{LL3}.
\end{proof}

Next we are going to estimate the nonlinearities. The following lemma says that the
nonlinearities have good pointwise decay property near the light cone if we disregard
the Riesz transform. This lemma is not only used in the estimate of weighted $L^2$ norm
in this section, but also plays one of the key roles in the energy estimate in next section.
\begin{lem}\label{Lem-Good-f}
Let $f_{\alpha a}^2$, $f_{\alpha a}^3$ denote the nonlinearities in \eqref{VisElas-Gamma-f123}. Then
for all $|\alpha|+|a|\leq \kappa-3$, there hold
\begin{equation}\label{good-f2}
|f_{\alpha a}^2|\lesssim \frac 1r
\sum\limits_{\tiny{\begin{matrix}|\beta|+|\gamma|\leq |\alpha|
\\ |b|+|c|\leq |a| \end{matrix}}}
|V^{(|\beta|,|b|+1)}| |H^{(|\gamma|,|c|+1)}|,
\end{equation}
\begin{equation}\label{good-f3}
|f_{\alpha a}^3|\lesssim \frac 1r
\sum\limits_{\tiny{\begin{matrix}|\beta|+|\gamma|\leq |\alpha|
\\ |b|+|c|\leq |a| \end{matrix}}}
|H^{(|\beta|,|b|+1)}| |H^{(|\gamma|,|c|+1)}|,
\end{equation}
\begin{equation}\label{good-f2d}
|\nabla\cdot f_{\alpha a}^2|\lesssim\frac 1r
\sum\limits_{\tiny{\begin{matrix}|\beta|+|\gamma|\leq |\alpha|
\\ |b|+|c|\leq |a| \end{matrix}}}
|V^{(|\beta|,|b|+2)}| |H^{(|\beta|,|c|+2)}|.
\end{equation}

Furthermore, recall the definition of  $f^{ij}_{\alpha a} $ in \eqref{fij}.  Then,  there holds
\begin{align}\label{good-f1}
|f_{\alpha a}^{ij}|\lesssim& \frac 1r
\sum\limits_{\tiny\begin{matrix}|b|+|c|\leq |a|\\
|\beta|+|\gamma|\leq |\alpha|\end{matrix}}
\big(|V^{(|\beta|,|b|+1)}| |V^{(|\gamma|,|c|+1)}|
+|H^{(|\beta|,|b|+1)}| |H^{|\gamma|,|c|+1}|
\big)\nonumber\\
&+\sum\limits_{\tiny\begin{matrix}b+c=a\\
\beta+\gamma=\alpha\end{matrix}}
\Big[|\partial_rV^{(\beta,b)}
+\partial_rH^{(\beta,b)}\cdot\omega|
\big(|\nabla V^{(\gamma,c)}|
+|\nabla H^{(\gamma,c)}|\big)\nonumber\\
&\qquad\qquad+|\partial_rH^{(\beta,b)}\cdot\omega^\perp \partial_rH^{(\gamma,c)}\cdot\omega^\perp|\Big].
\end{align}
\end{lem}
Recall that the introduction of $f_{\alpha a}^{ij}$ came  from $f_{\alpha a}^1$
by dropping the Riesz transforms.
\begin{remark}
Note that all the nonlinearities satisfy the strong null
condition, our estimates always contain one spatial derivative  in the good unknowns or gain $\langle t\rangle^{-1}$ near the light cone.
\end{remark}
\begin{remark}
In the highest order energy estimate of  the next section,
this lemma  can not be
used since it  causes a  derivative loss.
\end{remark}
\begin{proof}
Employing \eqref{der-decomp}, we write
\begin{align*}
f^2_{\alpha a}=&\sum\limits_{\tiny\begin{matrix}b+c=a\\ \beta+\gamma=\alpha\end{matrix}}
C_{\alpha}^\beta C_a^b
(\nabla^\perp H^{(\beta,b)}\nabla V^{(\gamma,c)}) \\
=&\sum\limits_{\tiny\begin{matrix}b+c=a\\ \beta+\gamma=\alpha\end{matrix}}
C_{\alpha}^\beta C_a^b
(\partial_r H^{(\beta,b)}\otimes\omega^\perp
-\frac{1}{r}\partial_\theta H^{(\beta,b)}\otimes\omega )
(\omega \partial_r V^{(\gamma,c)}
+\frac{\omega^\perp}{r}\partial_\theta V^{(\gamma,c)})  \\
=&\frac{1}{r}
\sum\limits_{\tiny\begin{matrix}b+c=a\\ \beta+\gamma=\alpha\end{matrix}}
C_{\alpha}^\beta C_a^b
 (\partial_rH^{(\beta,b)}
 \partial_\theta V^{(\gamma,c)}
-\partial_\theta H^{(\beta,b)}
\partial_rV^{(\gamma,c)}).
\end{align*}
Thus \eqref{good-f2} is clear from the commutation between
$\partial_r$ and $\widetilde{S}$, $\Gamma$.
Note that \eqref{good-f3} can be estimated exactly in the same fashion, we omit the details.

To estimate \eqref{good-f2d}, we use \eqref{der-decomp} to get that
\begin{align*}
\nabla\cdot f_{\alpha a}^2
&=\nabla\cdot\sum\limits_{\tiny\begin{matrix}b+c=a\\ \beta+\gamma=\alpha\end{matrix}}
C_{\alpha}^\beta C_a^b
(\nabla^\perp H^{(\beta,b)}\nabla V^{(\gamma,c)})\\
&=\sum\limits_{\tiny\begin{matrix}b+c=a\\ \beta+\gamma=\alpha\end{matrix}}
C_{\alpha}^\beta C_a^b
(\nabla^\perp_j \nabla_iH^{(\beta,b)}_i\nabla_j V^{(\gamma,c)}
+\nabla^\perp_j H^{(\beta,b)}_i
\nabla_j \nabla_iV^{(\gamma,c)})  \\
&=\frac 1r \sum\limits_{\tiny\begin{matrix}b+c=a\\ \beta+\gamma=\alpha\end{matrix}}
C_{\alpha}^\beta C_a^b
(\partial_r \nabla_iH^{(\beta,b)}_i\partial_\theta V^{(\gamma,c)}
-\partial_\theta \nabla_iH^{(\beta,b)}_i\partial_r V^{(\gamma,c)}) \\
&\quad +\frac 1r \sum\limits_{\tiny\begin{matrix}b+c=a\\ \beta+\gamma=\alpha\end{matrix}}
C_{\alpha}^\beta C_a^b
(\partial_r H^{(\beta,b)}_i
\partial_\theta \nabla_iV^{(\gamma,c)}
-\partial_\theta H^{(\beta,b)}_i
\partial_r \nabla_iV^{(\gamma,c)}).
\end{align*}
By the commutation between the generalized operators,
\eqref{good-f2d} is clear.

To estimate \eqref{good-f1}, similarly we can get by \eqref{der-decomp} to deduce that
\begin{eqnarray*}
&&f_{\alpha a}^{ij}=\sum\limits_{\tiny\begin{matrix}b+c=a\\ \beta+\gamma=\alpha\end{matrix}}C_{\alpha}^\beta C_a^b
\big[ (\omega_i\partial_rV^{(\beta,b)}+\frac 1r\omega_i^\perp\Omega V^{(\beta,b)})
(\omega_j\partial_rV^{(\gamma,c)}+\frac 1r\omega_j^\perp\Omega V^{(\gamma,c)})\\[-6mm]\\
&&\qquad\qquad\qquad-
(\omega_i\partial_rH^{(\beta,b)}+\frac 1r\omega_i^\perp\Omega H^{(\beta,b)})\cdot
(\omega_j\partial_rH^{(\gamma,c)}+\frac 1r\omega_j^\perp\Omega H^{(\gamma,c)})
\big]\\
&&\quad\ =\sum\limits_{\tiny\begin{matrix}b+c=a\\ \beta+\gamma=\alpha\end{matrix}}C_{\alpha}^\beta C_a^b
\big[\omega_i\omega_j(\partial_rV^{(\beta,b)}\partial_rV^{(\gamma,c)}-\partial_rH^{(\beta,b)}\cdot\partial_rH^{(\gamma,c)})\\[-6mm]\\
&&\qquad\qquad\qquad
+\frac 1r\omega_i\omega_j^\perp\partial_rV^{(\beta,b)}
\Omega V^{(\gamma,c)}
+\frac 1r\omega_i^\perp\Omega V^{(\beta,b)}
(\omega_j\partial_rV^{(\gamma,c)}+\frac 1r\omega_j^\perp\Omega V^{(\gamma,c)})\\
&&\qquad\qquad
-\frac 1r\omega_i\omega_j^\perp\partial_rH^{(\beta,b)}\cdot
\Omega H^{(\gamma,c)}
-\frac 1r\omega_i^\perp\Omega H^{(\beta,b)}\cdot
(\omega_j\partial_rH^{(\gamma,c)}+\frac 1r\omega_j^\perp\Omega H^{(\gamma,c)})\big].
\end{eqnarray*}
Thus \eqref{good-f1} is obtained by the following organization
\begin{eqnarray*}
&&\partial_rV^{(\beta,b)}\partial_rV^{(\gamma,c)}-\partial_rH^{(\beta,b)}\cdot\partial_rH^{(\gamma,c)}\\
&&=(\partial_rV^{(\beta,b)}+\partial_rH^{(\beta,b)}\cdot\omega)\partial_rV^{(\gamma,c)}
-\partial_rH^{(\beta,b)}\cdot\omega(\partial_rV^{(\gamma,c)}+\partial_rH^{(\gamma,c)}\cdot\omega)\\
&&\quad-\partial_rH^{(\beta,b)}\cdot\omega^\perp \partial_rH^{(\gamma,c)}\cdot\omega^\perp.
\end{eqnarray*}
This completes the proof of the lemma.
\end{proof}

\subsection{Estimate of the Weighted $L^2$ Energy}
In the sequel, we will show that the weighted energy can be controlled by the generalized energy under the smallness  assumptions on lower order energies.
\begin{lem}\label{Estimate-N}
Suppose that $(V,H)\in H^{\kappa-1}_{\Gamma}$ solves \eqref{VisElas3} and \eqref{constr}.
Then for all $|\alpha|+|a|\leq \kappa-3$,
there holds
\begin{eqnarray*}
&&\big\|N_{|\alpha|+|a|+1}+\mathcal{N}_{|\alpha|+|a|+2} \big\|^2_{L^2} \nonumber\\
&&\lesssim
 E_{|\alpha|+|a|+2}E_{[(|\alpha|+|a|)/2]+4}
 +Y_{|\alpha|+|a|+1}E_{[(|\alpha|+|a|)/2]+3} \nonumber\\
&&\quad +E_{|\alpha|+|a|+2} (X_{[(|\alpha|+|a|)/2]+4}+Y_{[(|\alpha|+|a|)/2]+3}) .
\end{eqnarray*}
\end{lem}
\begin{proof}
In view  of the definition of
$N_{|\alpha|+|a|+1}$ and $\mathcal{N}_{|\alpha|+|a|+2}$, it suffices to prove
\begin{eqnarray*}
&&\big\|t|f^1_{\alpha a}|+t|f^2_{\alpha a}|+(t+r)|f_{\alpha a}^3|
+t|\nabla\cdot f_{\alpha a}^2| \big\|^2_{L^2} \nonumber\\
&&\lesssim
 E_{|\alpha|+|a|+2}E_{[(|\alpha|+|a|)/2]+4}
 +Y_{|\alpha|+|a|+1}E_{[(|\alpha|+|a|)/2]+3} \nonumber\\
&&\quad +E_{|\alpha|+|a|+2} (X_{[(|\alpha|+|a|)/2]+4}+Y_{[(|\alpha|+|a|)/2]+3}) .
\end{eqnarray*}
Let us first treat $\|t|f^2_{\alpha a}|+(t+r)|f^3_{\alpha a}|\|^2_{L^2}$.
Recall  that $f^2_{\alpha a}$ and $f^3_{\alpha a}$ were  defined in
\eqref{VisElas-Gamma-f123}.
We need to estimate the norm in different regions separately. When $r\leq \langle t \rangle/2$,
we have $\langle t\rangle\lesssim \langle t-r \rangle$, thus
\begin{align*}
&\big\| t|f^2_{\alpha a}|+(t+r)|f^3_{\alpha a}|\big\|^2_{L^2(r\leq \langle t \rangle/2)}\\
&\lesssim
\sum_{\tiny\begin{matrix}\beta+\gamma =\alpha\\ b+c=a\end{matrix}}
\big\| \langle t \rangle |\nabla U^{(\beta,b)}|
|\nabla U^{(\gamma,c)}|  \big\|^2_{L^2(r\leq \langle t \rangle/2)}.
\end{align*}
By the symmetry between the multi-index $b$ and $c$ and the symmetry between $\beta$ and $\gamma$ in the above,
we assume $|c|+|\gamma|\leq |b|+|\beta|$ without loss of generality. Thus $|\gamma|+|c|\leq[(\alpha+|a|)/2]$.
Hence thanks to \eqref{K-S-3},
the above can be further bounded by
\begin{align*}
&\sum_{\tiny\begin{matrix}\beta+\gamma =\alpha, b+c=a\\
           |\gamma|+|c|\leq[(\alpha+|a|)/2]\end{matrix}}
\|\nabla U^{(\beta,b)}\|^2_{L^2}
\| \langle t \rangle\nabla U^{(\gamma,c)}\|^2_{L^\infty(r\leq \langle t \rangle/2)}\\
&\lesssim E_{|\alpha|+|a|+1}X_{[(|\alpha|+|a|)/2]+3} .
\end{align*}
For $r\geq \langle t \rangle/2$, one infers by \eqref{good-f2}, \eqref{good-f3} and Sobolev embedding that
\begin{align*}
&\big\| t|f^2_{\alpha a}|+(t+r)|f^3_{\alpha a}|\big\|^2_{L^2(r\geq \langle t \rangle/2)}\\
&\lesssim
\sum\limits_{\tiny{\begin{matrix}|\beta|+|\gamma|\leq |\alpha|
\\ |b|+|c|\leq |a| \end{matrix}}}
\big\| |U^{(|\beta|,|b|+1)}| |U^{(|\gamma|,|c|+1)}|
 \big\|^2_{L^2(r\geq \langle t \rangle/2)} \\
&\lesssim E_{|\alpha|+|a|+1}E_{[(|\alpha|+|a|)/2]+3}.
\end{align*}
Then we turn our attention to $\|tf^1_{\alpha a}\|_{L^2}$.
Recalling that $f^1_{\alpha a}$ is defined in \eqref{VisElas-Gamma-f123}.
By the $L^2$ boundness of the Riesz transform, one has
\begin{equation*}
\|tf_{\alpha a}^1\|_{L^2} \lesssim \sum_{1\leq i,j\leq 2}\|tf_{\alpha a}^{ij}\|_{L^2},
\end{equation*}
where $f_{\alpha a}^{ij}$ is defined in Lemma \ref{Lem-Good-f}. Hence in the following, we focus our attention on $\|tf^{ij}_{\alpha a}\|_{L^2}$.
When $r\leq \langle t \rangle/2$, we can estimate similar to
$\|tf^2_{\alpha a}\|_{L^2(r\leq \langle t \rangle/2)}$
to deduce that
\begin{equation*}
\|t f^{ij}_{\alpha a}\|^2_{L^2(r\leq \langle t \rangle/2)}
\lesssim E_{|\alpha|+|a|+1}X_{[(|\alpha|+|a|)/2]+3}.
\end{equation*}
When $r\geq \langle t \rangle/2$, by \eqref{good-f1},
one has
\begin{align} \label{f_ij_r}
\|t f^{ij}_{\alpha a}\|^2_{L^2(r\geq \langle t \rangle/2)}
\nonumber
&\lesssim
\sum\limits_{\tiny\begin{matrix}|b|+|c|\leq |a|\\ |\beta|+|\gamma|\leq|\alpha|\end{matrix}}
\big\||V^{(|\beta|,|b|+1)}| |V^{(|\gamma|,|c|+1)}|
+|H^{(|\beta|,|b|+1)}| |H^{(|\gamma|,|c|+1)}|
\big\|_{L^2}^2\nonumber\\
&\quad+\sum\limits_{\tiny\begin{matrix}b+c=a\\ \beta+\gamma=\alpha\end{matrix}}
\big\|r|(\partial_rV^{(\beta,b)}
+\partial_rH^{(\beta,b)}\cdot\omega)|
\big(|\nabla V^{(\gamma,c)}|
+|\nabla H^{(\gamma,c)}|\big)\big\|_{L^2}^2\nonumber\\
&\quad+\sum\limits_{\tiny\begin{matrix}b+c=a\\ \beta+\gamma=\alpha\end{matrix}}
\big\|r\partial_rH^{(\beta,b)}\cdot\omega^\perp \partial_rH^{(\gamma,c)}\cdot\omega^\perp\big\|^2_{L^2}.
\end{align}
For the first  and third terms  on the right-hand side of  \eqref{f_ij_r},
one can use traditional Sobolev inequality to deduce that
they are bounded by
$$E_{|\alpha|+|a|+1} E_{[(|\alpha|+|a|)/2]+3}+Y_{|\alpha|+|a|+1}E_{[(|\alpha|+|a|)/2]+3}.$$
The remaining second terms  of \eqref{f_ij_r} needs further work.
Making use of the fact that
\begin{align*}
\partial_\theta(V^{(\alpha,a)}+H^{(\alpha,a)}\cdot\omega)
&=\widetilde{\Omega}V^{(\alpha,a)}
+\widetilde{\Omega}H^{(\alpha,a)}\cdot\omega\\
&=\widetilde{S}^\alpha\widetilde{\Omega}\Gamma^aV
+\widetilde{S}^\alpha\widetilde{\Omega}\Gamma^aH\cdot\omega
\end{align*}
and by \eqref{K-S-1}, one gets
\begin{align*}
&\|r(\partial_rV^{(\gamma,c)}+\partial_rH^{(\gamma,c)}\cdot\omega) \|^2_{L^\infty(r\geq \langle t \rangle/2)}\\[-4mm]\\
&\lesssim \sum_{d=0,1}\big\{
\|\partial_r \Omega^d \Large[ r(\partial_rV^{(\gamma,c)}
+\partial_rH^{(\gamma,c)}\cdot\omega)\Large]\|_{L^2}^2
+\|\Omega^d \Large[r(\partial_rV^{(\gamma,c)}
+\partial_rH^{(\gamma,c)}\cdot\omega) \Large]\|_{L^2}^2\big\}\\
&\lesssim Y_{|\gamma|+|c|+3}
+E_{|\gamma|+|c|+2}.
\end{align*}
This allows us to  control the second  line of \eqref{f_ij_r} as follows:
\begin{eqnarray*}
&&\sum\limits_{\tiny\begin{matrix}b+c=a\\ \beta+\gamma=\alpha\end{matrix}}
\big\|r(\partial_rV^{(\beta,b)}
+\partial_rH^{(\beta,b)}\cdot\omega)
\Large(|\nabla V^{(\gamma,c)}|
+|\nabla H^{(\gamma,c)}|\Large)\big\|^2_{L^2(r\geq \langle t \rangle/2)}\\
&&\lesssim\sum_{\tiny\begin{matrix}\beta+\gamma =\alpha, b+c=a\\ |\beta|+|b|\geq |\gamma|+|c|\end{matrix}}
\|r(\partial_rV^{(\beta,b)}
 +\partial_rH^{(\beta,b)}\cdot\omega) \|^2_{L^2}
\|\nabla U^{(\gamma,c)}\|^2_{L^\infty} \\
&&+\sum_{\tiny\begin{matrix}\beta+\gamma =\alpha, b+c=a\\ |\beta|+|b|< |\gamma|+|c|\end{matrix}}
\|\nabla U^{(\gamma,c)}\|^2_{L^2}
\|r(\partial_rV^{(\beta,b)}
+\partial_rH^{(\beta,b)}\cdot\omega) \|^2_{L^\infty(r\geq \langle t \rangle/2)} \\
[-4mm]\\
&&\lesssim E_{|\alpha|+|a|+1} E_{[(|\alpha|+|a|)/2]+3}
+Y_{|\alpha|+|a|+1}E_{[(|\alpha|+|a|)/2]+3}
+E_{|\alpha|+|a|+1}Y_{[(|\alpha|+|a|)/2]+3}.
\end{eqnarray*}
Finally, we are going to show that
$$\|t\nabla\cdot f_{\alpha a}^2\|^2_{L^2}
\lesssim
E_{|\alpha|+|a|+2} (E_{[(|\alpha|+|a|)/2]+4}+X_{[(|\alpha|+|a|)/2]+4}).
$$
For $r\leq \langle t \rangle/2$, by \eqref{K-S-3}, we have
\begin{align*}
\|t\nabla\cdot f_{\alpha a}^2\|^2_{L^2(r\leq \langle t \rangle/2)}
&\lesssim
\sum\limits_{\tiny\begin{matrix}b+c=a\\ \beta+\gamma=\alpha\end{matrix}}
\big\|\langle t-r\rangle
|\nabla^2 U^{(\beta,b)}| |\nabla U^{(\gamma,c)}|\big\|^2_{L^2(r\leq \langle t \rangle/2)} \\
&\lesssim E_{|\alpha|+|a|+2} X_{[(|\alpha|+|a|)/2]+4}.
\end{align*}
For $r\geq \langle t \rangle/2$, one deduces by \eqref{good-f2d}  that:
\begin{align*}
&\|t\nabla\cdot f_{\alpha a}^2\|^2_{L^2(r\geq \langle t \rangle/2)}
\lesssim \|r\nabla\cdot f_{\alpha a}^2\|^2_{L^2(r\geq \langle t \rangle/2)} \\[-4mm]\\
&\lesssim \sum\limits_{\tiny\begin{matrix}|b|+|c|\leq |a|\\ |\beta|+|\gamma|\leq|\alpha|\end{matrix}}
\big\| | V^{(|\beta|,|b|+2)}| |H^{(|\gamma|,|c|+2)}| \big\|_{L^2(r\geq \langle t \rangle/2)} \\
&\lesssim  E_{|\alpha|+|a|+2} E_{[(|\alpha|+|a|)/2]+4}.
\end{align*}
This finishes the proof of the lemma.
\end{proof}

Now, we state a lemma that allows us to estimate the weighted $L^2$ norms:
\begin{lem}\label{Estimate-Nl}
Suppose that $(V,H)\in H^{\kappa-1}_{\Gamma}$ solves \eqref{VisElas3} and \eqref{constr} with $\kappa \geq12$. Then there hold
\begin{equation}\label{W1}
X_{\kappa-4}+Y_{\kappa-4}\lesssim
E_{\kappa-3}+Y_{\kappa-4}E_{\kappa-3}+E_{\kappa-3}X_{\kappa-4}+E_{\kappa-3}E_{\kappa-4},
\end{equation}
and
\begin{equation}\label{W2}
X_{\kappa-2}+Y_{\kappa-2}\lesssim E_{\kappa-1}+E_{\kappa-1}X_{\kappa-4}
+Y_{\kappa-2}E_{\kappa-4}+E_{\kappa-1}Y_{\kappa-4}+E_{\kappa-1}E_{\kappa-4}.
\end{equation}
\end{lem}

\begin{proof}
For the proof of this lemma, we recall and prove the following simple lemma :

\begin{lem}
\label{grad-div-curl}
For vector $K$, there holds
\[
\|\langle t-r\rangle \nabla K\|_{L^2}\lesssim  \|\langle t-r\rangle\nabla\cdot K\|_{L^2}
+\|\langle t-r\rangle\nabla^\perp\cdot K\|_{L^2}  +\|K\|_{L^2},
\]
provided the right-hand side is finite.
\end{lem}
\begin{proof}
The proof is rather simple and the version for matrix has appeared in \cite{LSZ13}. For completeness we include the proof for vector $K$. It suffices to prove the lemma for $K\in C_0^2(\mathbb{R}^2)$, the general case can  be established
by a completion procedure.

For any vector $K$, we write
\[
|\nabla K|^2=|\nabla\cdot K|^2+|\nabla^\perp\cdot K|^2
-2\partial_1 K_1\partial_2K_2+\partial_2K_1\partial_1K_2.
\]
By integration by parts and Young's inequality, we have
\begin{eqnarray*}
&&\|\langle t-r\rangle\nabla K\|_{L^2}^2-
\|\langle t-r\rangle\nabla\cdot  K\|_{L^2}^2-\|(t-r)\nabla^\perp\cdot K\|_{L^2}^2\\
&&=\int_{\mathbb{R}^2} 2\langle t-r\rangle^2 [-\partial_1(K_1\partial_2K_2)+\partial_2(K_1\partial_1K_2)]dx\\
&&=\int_{\mathbb{R}^2}4(t-r)[-\omega_1K_1\partial_2K_2+\omega_2K_1\partial_1K_2]dx\\
&&\le\frac12\|\langle t-r\rangle\nabla K\|_{L^2}^2+C\|K\|_{L^2}^2.
\end{eqnarray*}
The lemma then follows from the fact that the first term of the right-hand side can be absorbed by the left-hand side.
\end{proof}

 We go back to the proof of  Lemma \ref{Estimate-Nl} and we first show that
\begin{align}\label{F1}
&X_{|\alpha|+|a|+1}+Y_{|\alpha|+|a|+1}\nonumber\\[-4mm]\nonumber\\
&\lesssim E_{|\alpha|+|a|+2}+E_{|\alpha|+|a|+2} (X_{[(|\alpha|+|a|)/2]+4}+Y_{[(|\alpha|+|a|)/2]+3})\nonumber\\[-4mm]\nonumber\\
&+(X_{|\alpha|+|a|+1}+Y_{|\alpha|+|a|+1})E_{[(|\alpha|+|a|)/2]+3}
+E_{|\alpha|+|a|+2}E_{[(|\alpha|+|a|)/2]+4}.
\end{align}
Actually, by Lemma \ref{Estimate-N}, we only need to show
\begin{eqnarray} \label{F2}
X_{|\alpha|+|a|+1}+Y_{|\alpha|+|a|+1}
\lesssim
E_{|\alpha|+|a|+2}+\|N_{|\alpha|+|a|+1}\|^2_{L^2}
+\|\mathcal{N}_{|\alpha|+|a|+2}\|^2_{L^2}.
\end{eqnarray}
In view of the fact that
\begin{eqnarray*}
&&\nabla V^{(\alpha,a)}
=\frac12[\nabla V^{(\alpha,a)}
+(\nabla\cdot H^{(\alpha,a)})\omega] +
\frac12[\nabla V^{(\alpha,a)}
-(\nabla\cdot H^{(\alpha,a)})\omega], \\
&&(\nabla\cdot H^{(\alpha,a)})\omega
=\frac12[\nabla V^{(\alpha,a)}
+(\nabla\cdot H^{(\alpha,a)})\omega] -
\frac12[\nabla V^{(\alpha,a)}
-(\nabla\cdot H^{(\alpha,a)})\omega],
\end{eqnarray*}
we deduce that
\begin{align*}
&\langle t-r\rangle(|\nabla V^{(\alpha,a)}|
 +|\nabla\cdot H^{(\alpha,a)}|) \nonumber\\[-4mm]\nonumber\\
&\lesssim \langle t+r\rangle
|\nabla V^{(\alpha,a)}+ \nabla\cdot H^{(\alpha,a)}\omega|
+|\langle t-r\rangle(\nabla V^{(\alpha,a)}
 - \nabla\cdot H^{(\alpha,a)}\omega)|.
\end{align*}
By \eqref{LL3}, the above can be further bounded by
\begin{align*}
L_{|\alpha|+|a|+1}+N_{|\alpha|+|a|+1}+
|r\partial_r V^{(\alpha,a)}+t\nabla\cdot H^{(\alpha,a)}|.
\end{align*}
Hence by Lemma \ref{lemLinear} and Lemma \ref{grad-div-curl},
we have
\begin{eqnarray*}
&&\|\langle t-r\rangle \nabla V^{(\alpha,a)}\|^2_{L^2}
+\|\langle t-r\rangle \nabla H^{(\alpha,a)}\|^2_{L^2}\nonumber\\[-4mm]\nonumber\\
&&\lesssim
\| \langle t-r\rangle \nabla V^{(\alpha,a)}\|^2_{L^2}
+\|\langle t-r\rangle \nabla\cdot H^{(\alpha,a)}\|^2_{L^2} \nonumber\\[-4mm]\nonumber\\
&&\quad+\|\langle t-r\rangle \nabla^\perp\cdot H^{(\alpha,a)}\|^2_{L^2}
+\|H^{(\alpha,a)}\|^2_{L^2}              \\
&&\lesssim \nu X_{|\alpha|+|a|+1}+\|L_{|\alpha|+|a|+2}\|^2_{L^2}
+C_\nu \|N_{|\alpha|+|a|+1}\|^2_{L^2}
+C_\nu \|\mathcal{N}_{|\alpha|+|a|+2}\|^2_{L^2},
\end{eqnarray*}
for any positive $\nu$. This further implies that
\begin{eqnarray*}
X_{|\alpha|+|a|+1}
\lesssim \nu X_{|\alpha|+|a|+1}+\|L_{|\alpha|+|a|+2}\|^2_{L^2}
+C_\nu\|N_{|\alpha|+|a|+1}\|^2_{L^2}
+C_\nu\|\mathcal{N}_{|\alpha|+|a|+2}\|^2_{L^2}.
\end{eqnarray*}
Taking  $\nu \geq 0 $ small enough, the first term on the right hand side in the above is absorbed
by the left hand side. This yields
\begin{eqnarray} \label{F3}
X_{|\alpha|+|a|+1}
\lesssim \|L_{|\alpha|+|a|+2}\|^2_{L^2}
+\|N_{|\alpha|+|a|+1}\|^2_{L^2}
+\|\mathcal{N}_{|\alpha|+|a|+2}\|^2_{L^2}.
\end{eqnarray}
The estimate for $Y_{|\alpha|+|a|+1}$ in \eqref{F2}
is obvious from \eqref{LL4}, \eqref{LL5},
 \eqref{F3} and Lemma \ref{lemLinear}. Thus \eqref{F2} is proved.

Now we turn to the proof of the first inequality in the lemma:  Let $\kappa \geq 12$, $|\alpha|+|a|+1 \leq \kappa - 4$,
one has $[(|\alpha|+|a|)/2] +4  \leq \kappa -4$. Hence, by \eqref{F1}, we have
\begin{equation*}
X_{\kappa-4}+Y_{\kappa-4}\lesssim E_{\kappa-3}
+Y_{\kappa-4}E_{\kappa-3}+E_{\kappa-3}X_{\kappa-4}+E_{\kappa-3}E_{\kappa-4}.
\end{equation*}
Next, for $|\alpha|+|a|+1 \leq \kappa-2 $, there holds $[(|a|+|\alpha|)/2] + 4
\leq \kappa - 4$. Hence one can derive from \eqref{F1} that
\begin{equation*}
X_{\kappa-2}+Y_{\kappa-2}\lesssim E_{\kappa-1}+E_{\kappa-1}X_{\kappa-4}
+Y_{\kappa-2}E_{\kappa-4}+E_{\kappa-1}Y_{\kappa-4}+E_{\kappa-1}E_{\kappa-4}.
\end{equation*}
\end{proof}

The following lemma gives the control of weighted  generalized energies and weighted good quantity energies in terms of Klainerman's generalized ones. Note that we have one derivative loss with respect to similar estimates in \cite{Lei16,LSZ13,Sideris00}.
\begin{lem}\label{lemmaWeighted-Energy}
Suppose that $(V,H)\in H^{\kappa-1}_{\Gamma}$ solves \eqref{VisElas3} and \eqref{constr} with $\kappa\geq 12$,
and suppose $E_{\kappa-3}\ll1$. Then,  we have
\begin{equation*}
X_{\kappa-4}+Y_{\kappa-4}\lesssim E_{\kappa-3},\quad
X_{\kappa-2}+Y_{\kappa-2}\lesssim E_{\kappa-1}.
\end{equation*}
\end{lem}
\begin{remark}
When $\mu=0$, we can modify Lemma \ref{Estimate-N} and Lemma \ref{Estimate-Nl}
and finally get a non-derivative-loss version of Lemma \ref{lemmaWeighted-Energy}:
\begin{equation*}
X_{\kappa-3}+Y_{\kappa-3}\lesssim E_{\kappa-3},\quad
X_{\kappa-1}+Y_{\kappa-1}\lesssim E_{\kappa-1}.
\end{equation*}
\end{remark}
\begin{proof}
The first estimate follows from \eqref{W1} and the assumption $E_{\kappa-3}\ll1$. The second one follows from \eqref{W2}, the assumption
and the obtained first estimate.
\end{proof}

\subsection{Strengthened $L^\infty$ estimate for the Good Unknowns}
We now complete the decay estimate for the $L^\infty$ of the good unknowns $\partial V+\partial H\cdot\omega$
and $\partial H\cdot\omega^\perp$ near the light cone.
\begin{lem}\label{LemmaGood-L-infy}
Suppose that $(V,H)\in H^{\kappa-1}_{\Gamma}$ solves \eqref{VisElas3} and \eqref{constr} with $\kappa\geq 12$ and suppose $E_{\kappa-3}\ll1$. Then for all $|\alpha|+|a|\leq \kappa-7$ and  for  $ i  = 1,  2$,  we have
\begin{equation} \label{Good-L-infy-1}
\langle t \rangle^{\frac 32}
\big\||\nabla_i V^{(\alpha,a)}+\nabla_iH^{(\alpha,a)}\cdot\omega |+
|\nabla_iH^{(\alpha,a)}\cdot\omega^\perp| \big\|_{L^\infty( r\geq\langle t \rangle/2)}
\lesssim E_{\kappa-3}^{\frac12}.
\end{equation}
\end{lem}
\begin{proof}
In view of  \eqref{der-decomp} and \eqref{K-S-1}, we only need to show
\begin{equation*}
\langle t \rangle^{\frac 32}
\big\||\partial_rV^{(\alpha,a)}+\partial_rH^{(\alpha,a)}\cdot\omega |+
|\partial_rH^{(\alpha,a)}\cdot\omega^\perp| \big\|_{L^\infty( r\geq\langle t \rangle/2)}
\lesssim E_{\kappa-3}^{\frac12}.
\end{equation*}
Note that
\begin{eqnarray*}
&&\partial_\theta(V^{(\alpha,a)}+H^{(\alpha,a)}\cdot\omega)=\widetilde{\Omega}V^{(\alpha,a)}+\widetilde{\Omega}H^{(\alpha,a)}\cdot\omega.
\end{eqnarray*}
By \eqref{K-S-1} and Lemma \ref{lemmaWeighted-Energy}, one gets
\begin{align*}
&\|r^\frac 32(\partial_rV^{(\alpha,a)}+\partial_rH^{(\alpha,a)}\cdot\omega) \|^2_{L^\infty(r\geq \langle t \rangle/2)}\\
&\lesssim \sum_{d=0,1}\big\{
\|\partial_r \Omega^d \Large[ r(\partial_rV^{(\alpha,a)}+\partial_rH^{(\alpha,a)}\cdot\omega)\Large]\|_{L^2}^2
+\|\Omega^d \Large[r(\partial_rV^{(\alpha,a)}+\partial_rH^{(\alpha,a)}\cdot\omega) \Large]\|_{L^2}^2\big\}\\
&\lesssim \sum_{d=0,1}\big\{
\|r(\partial^2_r\widetilde{\Omega}^d V^{(\alpha,a)}
+\partial^2_r\widetilde{\Omega}^dH^{(\alpha,a)}\cdot\omega)\|_{L^2}^2+
\|\partial_r \widetilde{\Omega}^dV^{(\alpha,a)}
+\partial_r\widetilde{\Omega}^dH^{(\alpha,a)}\cdot\omega\|_{L^2}^2 \\
&\qquad\quad
+\|r(\partial_r\widetilde{\Omega}^dV^{(\alpha,a)}
+\partial_r\widetilde{\Omega}^dH^{(\alpha,a)}\cdot\omega) \|_{L^2}^2\big\}
\\[-4mm]\\
&\lesssim Y_{|\alpha|+|a|+3}+E_{|\alpha|+|a|+2}
\lesssim E_{\kappa-3}.
\end{align*}
The first part of \eqref{Good-L-infy-1} is clear from the fact that $r\geq\langle t \rangle/2$. The proof for the remaining part of the inequality is similar. We omit the details.
\end{proof}

\section{Energy Estimate}

This section is devoted to the energy estimate. We split the proof into three subsections, which correspond to the
highest-order modified energy estimate, the highest-order standard energy estimate
and the lower-order standard energy estimate, respectively. Here in this section, both $\mathcal{E}_\kappa$ and $E_\kappa$
will be called energies. To avoid confusion, we will call $E_\kappa$ the standard energy and still call $\mathcal{E}_\kappa$ the modified energy.

\subsection{Higher-order Modified Energy Estimate}

We first take care of the highest-order modified energy estimate.
One need to be very careful about the derivative loss problem.
Ignoring the diffusion, at   first glance, we will always lose one derivative in the highest-order modified energy estimate
due to the ``fully nonlinear" effect. The nonlocal effect and the application of the  ghost weight make this problem
even more complicated. Luckily, a delicate analysis of the nonlinearities shows that the system has the
requisite symmetry which is hidden in the Riesz transform. Actually, we can integrate
by parts in a way which will produce a Laplacian operator in the worst terms (the
worst terms refer to the terms with a  derivative loss.  The other terms
do not have such problems and the null structure is satisfied).
Moreover, after   gaining the  one derivative, the null condition
is  present again. Then we can take full  advantage of this condition by the ghost weight method.

Let $\kappa\geq 12$, $|\alpha|+|a| \leq \kappa-1$, $\sigma =r-t$ and $q(\sigma)=\arctan \sigma$. We write $e^q=e^{q(\sigma)}$ for simplicity.
After applying $\nabla$ to \eqref{VisElas-Gamma},
we take the $L^2$ inner product of the first and
second equations  of the resulting system
 with $\nabla V^{(\alpha,a)} e^q$ and $\nabla H^{(\alpha,a)} e^q$, respectively, then adding them up, we get
\begin{eqnarray}\label{EN-Es1}
&&\frac 12\frac{d}{dt}
\int_{\mathbb{R}^2}(|\nabla V^{(\alpha,a)}|^2
    +|\nabla H^{(\alpha,a)}|^2)e^qdx\nonumber\\
&&-\int_{\mathbb{R}^2}
\mu\nabla\Delta\sum\limits_{l=0}^{\alpha} C_\alpha^l (-1)^{\alpha-l}V^{(l,a)}
\cdot\nabla V^{(\alpha,a)}e^q dx \nonumber\\
&&+\frac 12\sum_{1\leq i \leq 2} \int_{\mathbb{R}^2}\frac{|\nabla_iV^{(\alpha,a)}
+\nabla_iH^{(\alpha,a)}\cdot\omega|^2
+|\nabla_iH^{(\alpha,a)}\cdot\omega^\perp|^2}
{\langle t-r \rangle^2}e^qdx\nonumber\\
&&=\int_{\mathbb{R}^2} (\nabla f^1_{\alpha a} \cdot\nabla V^{(\alpha,a)}
+ \nabla f^2_{\alpha a} :\nabla H^{(\alpha,a)})e^q dx\nonumber\\
&&=I_1+I_2,
\end{eqnarray}
where
\begin{align*}
I_1=&\int_{\mathbb{R}^2}\nabla\big[
\nabla^\perp\cdot
\nabla\cdot \Delta^{-1}\big(-\nabla^\perp V^{(\alpha,a)}\otimes\nabla^\perp V
+\nabla^\perp H^{(\alpha,a)}\otimes \nabla^\perp H \big)\\
&\qquad +\nabla^\perp\cdot
\nabla\cdot \Delta^{-1}\big(-\nabla^\perp V\otimes\nabla^\perp V^{(\alpha,a)}+
\nabla^\perp H\otimes \nabla^\perp H^{(\alpha,a)} \big)
 \big]\cdot
\nabla V^{(\alpha,a)} e^qdx\\
&+\int_{\mathbb{R}^2}\nabla(\nabla^\perp H^{(\alpha,a)}\nabla V
+\nabla^\perp H\nabla V^{(\alpha,a)}):\nabla H^{(\alpha,a)} e^qdx,
\end{align*}
and
\begin{align*}
I_2=&\sum_{\tiny\begin{matrix} \beta+\gamma=\alpha, b + c = a
\\|\beta|+|b|,|\gamma|+|c|<|\alpha|+|a| \end{matrix}}
\Big\{C_\alpha^\beta C_a^b\int_{\mathbb{R}^2} \nabla\big[
\nabla^\perp\cdot
\nabla\cdot \Delta^{-1}\big(-\nabla^\perp V^{(\beta,b)}\otimes\nabla^\perp V^{(\gamma,c)} \\[-6mm]\\
&\qquad\qquad\qquad\qquad+\nabla^\perp H^{(\beta,b)}\otimes
\nabla^\perp H^{(\gamma,c)} \big)
 \big]\cdot \nabla V^{(\alpha,a)} e^qdx\\[-4mm]\\
&\qquad\qquad +C_\alpha^\beta C_a^b
\int_{\mathbb{R}^2} \nabla(\nabla^\perp H^{(\beta,b)}
\nabla V^{(\gamma,c)} ):\nabla H^{(\alpha,a)} e^qdx \Big\}.
\end{align*}
Here $I_1$  consist of  the terms  which contain the highest order derivatives,
namely when all derivatives hit the same factor in the nonlinear term.
To avoid notational confusion, we mention that $\nabla_kH^{(\alpha,a)}\cdot\omega$ and
$\nabla_kH^{(\alpha,a)}\cdot\omega^\perp$
appearing  in the ghost weight energy  \eqref{ghost}
mean $(\nabla_kH^{(\alpha,a)})\cdot\omega$ and
$(\nabla_kH^{(\alpha,a)})\cdot\omega^\perp$.
This notation convention will always be used in the following argument.
Also we denote (see third line of \eqref{EN-Es1})
\begin{align}  \label{ghost}
G_{\kappa}(t)=\sum_{|\alpha|+|a|\leq\kappa-1}\sum_{1\leq i \leq 2} \int_{\mathbb{R}^2}\frac{|\nabla_iV^{(\alpha,a)}
+\nabla_iH^{(\alpha,a)}\cdot\omega|^2
+|\nabla_iH^{(\alpha,a)}\cdot\omega^\perp|^2}
{\langle t-r \rangle^2}e^qdx.
\end{align}

Step 1: estimate of the highest order term $I_1$.

We divide $I_1$ into five terms:
\begin{align*}
I_1=I_{11}+I_{12}+I_{13}+I_{14}+I_{15},
\end{align*}
where
\begin{align*}
&I_{11}=-\int_{\mathbb{R}^2}\nabla\big[\nabla^\perp\cdot
\nabla\cdot \Delta^{-1}(\nabla^\perp V^{(\alpha,a)}\otimes\nabla^\perp V)\big]\cdot\nabla V^{(\alpha,a)} e^qdx,\\
&I_{12}=\int_{\mathbb{R}^2}\nabla\big[
\nabla^\perp\cdot\nabla\cdot \Delta^{-1}(
\nabla^\perp H^{(\alpha,a)}\otimes \nabla^\perp H)\big]
\cdot\nabla V^{(\alpha,a)} e^qdx,\\
&I_{13}=-\int_{\mathbb{R}^2}\nabla \big[\nabla^\perp\cdot\nabla\cdot \Delta^{-1}(
\nabla^\perp V\otimes\nabla^\perp V^{(\alpha,a)})\big]
\cdot\nabla V^{(\alpha,a)} e^qdx,\\
&I_{14}=\int_{\mathbb{R}^2}\nabla\big[\nabla^\perp\cdot\nabla\cdot \Delta^{-1}(
\nabla^\perp H\otimes \nabla^\perp H^{(\alpha,a)})
 \big]\cdot\nabla V^{(\alpha,a)}  e^qdx,\\
&I_{15}=\int_{\mathbb{R}^2}\nabla\big(\nabla^\perp H^{(\alpha,a)}\nabla V
+\nabla^\perp H\nabla V^{(\alpha,a)}\big):\nabla H^{(\alpha,a)} e^qdx.
\end{align*}
Now we transform  $I_{11}$ to $I_{15}$ one by one.   The goal is  take advantage
of the symmetric nature of the original system to
get rid of the derivative loss that appears. Due to the good property of the original system
it is expected that one can get rid of the derivative loss,  however this requires some lengthy
calculations.  For $I_{11}$, we deduce by
integration by parts   that
\begin{align*}
I_{11}&=-\int_{\mathbb{R}^2}\nabla\nabla^\perp\cdot\nabla\cdot \Delta^{-1}
\big(\nabla^\perp V^{(\alpha,a)}\otimes\nabla^\perp V \big)
\cdot\nabla V^{(\alpha,a)} e^q dx\nonumber\\
&=-\int_{\mathbb{R}^2}\nabla_k\nabla^\perp_i\nabla_j \Delta^{-1}
\big(\nabla^\perp_iV^{(\alpha,a)}\nabla^\perp_j V \big)
\nabla_kV^{(\alpha,a)} e^q dx\nonumber\\
&=-\int_{\mathbb{R}^2}\nabla^\perp_i\nabla_j \Delta^{-1}
\big(\nabla_k\nabla^\perp_iV^{(\alpha,a)}\nabla^\perp_j V \big)
\nabla_kV^{(\alpha,a)} e^q dx\nonumber\\
&\quad-\int_{\mathbb{R}^2}\nabla^\perp_i\nabla_j \Delta^{-1}
\big(\nabla^\perp_iV^{(\alpha,a)}\nabla_k\nabla^\perp_j V \big)
\nabla_kV^{(\alpha,a)} e^q dx \nonumber\\
&=\int_{\mathbb{R}^2}\nabla^\perp_i\nabla_j \Delta^{-1}
\big(\nabla_kV^{(\alpha,a)}\nabla^\perp_i\nabla^\perp_j V \big)
\nabla_kV^{(\alpha,a)} e^q dx\nonumber\\
&\quad-\int_{\mathbb{R}^2}\nabla^\perp_i\nabla^\perp_i\nabla_j \Delta^{-1}
\big(\nabla_kV^{(\alpha,a)}\nabla^\perp_j V \big)
\nabla_kV^{(\alpha,a)} e^q dx\nonumber\\
&\quad-\int_{\mathbb{R}^2}\nabla^\perp_i\nabla_j \Delta^{-1}
\big(\nabla^\perp_iV^{(\alpha,a)}\nabla_k\nabla^\perp_j V \big)
\nabla_kV^{(\alpha,a)} e^q dx \nonumber\\
&=\int_{\mathbb{R}^2}\nabla^\perp_i\nabla_j \Delta^{-1}
\big(\nabla_kV^{(\alpha,a)}\nabla^\perp_i\nabla^\perp_j V \big)
\nabla_kV^{(\alpha,a)} e^q dx\nonumber\\
&\quad+\frac{1}{2}\int_{\mathbb{R}^2}
|\nabla V^{(\alpha,a)}|^2
\nabla_j(\nabla^\perp_j V e^q) dx\nonumber\\
&\quad-\int_{\mathbb{R}^2}\nabla^\perp_i\nabla_j \Delta^{-1}
\big(\nabla^\perp_iV^{(\alpha,a)}\nabla_k\nabla^\perp_j V \big)
\nabla_kV^{(\alpha,a)} e^q dx.
\end{align*}
Next, for $I_{12}$  we get  by  integration by parts that
\begin{align*}
I_{12}&=-\int_{\mathbb{R}^2}\nabla\nabla^\perp\cdot\nabla\cdot \Delta^{-1}
\big(\nabla^\perp V\otimes\nabla^\perp V^{(\alpha,a)} \big)
\cdot\nabla V^{(\alpha,a)} e^q dx\nonumber\\
&=-\int_{\mathbb{R}^2}\nabla_k\nabla^\perp_i\nabla_j \Delta^{-1}
\big(\nabla^\perp_i V\nabla^\perp_jV^{(\alpha,a)} \big)
\nabla_kV^{(\alpha,a)} e^q dx\nonumber\\
&=-\int_{\mathbb{R}^2}\nabla^\perp_i\nabla_j \Delta^{-1}
\big(\nabla_k\nabla^\perp_i V\nabla^\perp_jV^{(\alpha,a)} \big)
\nabla_kV^{(\alpha,a)} e^q dx\nonumber\\
&\quad-\int_{\mathbb{R}^2}\nabla^\perp_i\nabla_j \Delta^{-1}
\big(\nabla^\perp_i V\nabla_k\nabla^\perp_jV^{(\alpha,a)} \big)
\nabla_kV^{(\alpha,a)} e^q dx\nonumber\\
&=-\int_{\mathbb{R}^2}\nabla^\perp_i\nabla_j \Delta^{-1}
\big(\nabla_k\nabla^\perp_i V\nabla^\perp_jV^{(\alpha,a)} \big)
\nabla_kV^{(\alpha,a)} e^q dx\nonumber\\
&\quad+\int_{\mathbb{R}^2}\nabla^\perp_i\nabla_j \Delta^{-1}
\big(\nabla^\perp_i\nabla^\perp_j V\nabla_kV^{(\alpha,a)} \big)
\nabla_kV^{(\alpha,a)} e^q dx.
\end{align*}
Then for $I_{13}$, we write
\begin{align*}
I_{13}&=\int_{\mathbb{R}^2}
\nabla\nabla^\perp\cdot\nabla\cdot \Delta^{-1}
\big(\nabla^\perp H^{(\alpha,a)} \otimes\nabla^\perp H \big)
\cdot\nabla V^{(\alpha,a)} e^q dx\nonumber\\
&=\int_{\mathbb{R}^2}
\nabla_k\nabla^\perp_i\nabla_j \Delta^{-1}
\big(\nabla^\perp_i H^{(\alpha,a)} \cdot\nabla^\perp_j H \big)
\nabla_kV^{(\alpha,a)} e^q dx\nonumber\\
&=\int_{\mathbb{R}^2}
\nabla^\perp_i\nabla_j \Delta^{-1}
\big(\nabla_k\nabla^\perp_i H^{(\alpha,a)} \cdot\nabla^\perp_j H \big)
\nabla_kV^{(\alpha,a)} e^q dx\nonumber\\
&\quad+\int_{\mathbb{R}^2}
\nabla^\perp_i\nabla_j \Delta^{-1}
\big(\nabla^\perp_i H^{(\alpha,a)} \cdot\nabla_k\nabla^\perp_j H \big)
\nabla_kV^{(\alpha,a)} e^q dx\nonumber\\
&=\int_{\mathbb{R}^2}
\nabla^\perp_i\nabla^\perp_i\nabla_j \Delta^{-1}
\big(\nabla_k H^{(\alpha,a)} \cdot\nabla^\perp_j H \big)
\nabla_kV^{(\alpha,a)} e^q dx\nonumber\\
&\quad-\int_{\mathbb{R}^2}
\nabla^\perp_i\nabla_j \Delta^{-1}
\big(\nabla_k H^{(\alpha,a)} \cdot\nabla^\perp_i\nabla^\perp_j H \big)
\nabla_kV^{(\alpha,a)} e^q dx\nonumber\\
&\quad+\int_{\mathbb{R}^2}
\nabla^\perp_i\nabla_j \Delta^{-1}
\big(\nabla^\perp_i H^{(\alpha,a)} \cdot\nabla_k\nabla^\perp_j H \big)
\nabla_kV^{(\alpha,a)} e^q dx\nonumber\\
&=-\int_{\mathbb{R}^2}
\nabla_k H^{(\alpha,a)} \cdot\nabla^\perp_j H\nabla_j\big(
\nabla_kV^{(\alpha,a)} e^q \big)dx\nonumber\\
&\quad-\int_{\mathbb{R}^2}
\nabla^\perp_i\nabla_j \Delta^{-1}
\big(\nabla_k H^{(\alpha,a)} \cdot\nabla^\perp_i\nabla^\perp_j H \big)
\nabla_kV^{(\alpha,a)} e^q dx\nonumber\\
&\quad+\int_{\mathbb{R}^2}
\nabla^\perp_i\nabla_j \Delta^{-1}
\big(\nabla^\perp_i H^{(\alpha,a)} \cdot\nabla_k\nabla^\perp_j H \big)
\nabla_kV^{(\alpha,a)} e^q dx.
\end{align*}
For $I_{14}$, we deduce similarly  to $I_{12}$   that
\begin{align*}
I_{14}&=\int_{\mathbb{R}^2}\nabla
\nabla^\perp\cdot\nabla\cdot \Delta^{-1}
\big(\nabla^\perp H\otimes \nabla^\perp H^{(\alpha,a)} \big)
\cdot\nabla V^{(\alpha,a)} e^q dx\nonumber\\
&=\int_{\mathbb{R}^2}
\nabla_k\nabla^\perp_i\nabla_j \Delta^{-1}
\big(\nabla^\perp_i H \cdot\nabla^\perp_j H^{(\alpha,a)} \big)
\nabla_kV^{(\alpha,a)} e^q dx\nonumber\\
&=\int_{\mathbb{R}^2}
\nabla^\perp_i\nabla_j \Delta^{-1}
\big(\nabla_k\nabla^\perp_i H \cdot\nabla^\perp_j H^{(\alpha,a)} \big)
\nabla_kV^{(\alpha,a)} e^q dx\nonumber\\
&\quad+\int_{\mathbb{R}^2}
\nabla^\perp_i\nabla_j \Delta^{-1}
\big(\nabla^\perp_i H \cdot\nabla_k\nabla^\perp_j H^{(\alpha,a)} \big)
\nabla_kV^{(\alpha,a)} e^q dx\nonumber\\
&=\int_{\mathbb{R}^2}
\nabla^\perp_i\nabla_j \Delta^{-1}
\big(\nabla_k\nabla^\perp_i H \cdot\nabla^\perp_j H^{(\alpha,a)} \big)
\nabla_kV^{(\alpha,a)} e^q dx\nonumber\\
&\quad-\int_{\mathbb{R}^2}
\nabla^\perp_i\nabla_j \Delta^{-1}
\big(\nabla^\perp_i \nabla^\perp_j H \cdot\nabla_k H^{(\alpha,a)} \big)
\nabla_kV^{(\alpha,a)} e^q dx.
\end{align*}
For $I_{15}$, we have
\begin{align*}
I_{15}&=\int_{\mathbb{R}^2}\nabla_k(\nabla^\perp_j H^{(\alpha,a)}_i\nabla_j V
+\nabla^\perp_j H_i\nabla_j V^{(\alpha,a)})\nabla_kH^{(\alpha,a)}_i e^qdx\nonumber\\
&=\int_{\mathbb{R}^2}(\nabla_k\nabla^\perp_j H^{(\alpha,a)}_i\nabla_j V
+\nabla^\perp_j H^{(\alpha,a)}_i\nabla_k\nabla_j V)\nabla_kH^{(\alpha,a)}_i e^qdx\nonumber\\
&\quad+\int_{\mathbb{R}^2}(\nabla^\perp_j H_i\nabla_k\nabla_j V^{(\alpha,a)}
+\nabla_k\nabla^\perp_j H_i\nabla_j V^{(\alpha,a)})\nabla_kH^{(\alpha,a)}_i e^qdx\nonumber\\
&=-\frac12 \int_{\mathbb{R}^2}|\nabla H^{(\alpha,a)}|^2 \nabla^\perp_j(\nabla_j Ve^q)dx
+\int_{\mathbb{R}^2}\nabla^\perp_j H_i\nabla_k\nabla_j V^{(\alpha,a)}\nabla_kH^{(\alpha,a)}_i e^qdx\nonumber\\
&\quad+\int_{\mathbb{R}^2}(\nabla^\perp_j H^{(\alpha,a)}_i\nabla_k\nabla_jV
+\nabla_k\nabla^\perp_j H_i\nabla_j V^{(\alpha,a)})
\nabla_kH^{(\alpha,a)}_i e^qdx.
\end{align*}
Inserting the above equalities from $I_{11}$ to $I_{15}$ into $I_1$,
we get
\begin{align}\label{I1}
I_1=&\int_{\mathbb{R}^2}\nabla^\perp_i\nabla_j \Delta^{-1}
\big(\nabla^\perp_i H^{(\alpha,a)} \cdot\nabla_k\nabla^\perp_j H-\nabla^\perp_iV^{(\alpha,a)}\nabla_k\nabla^\perp_j V \big)
\nabla_kV^{(\alpha,a)} e^q dx\nonumber\\
&+\int_{\mathbb{R}^2}\nabla^\perp_i\nabla_j \Delta^{-1}
\big(\nabla_kV^{(\alpha,a)}\nabla^\perp_i\nabla^\perp_j V
-\nabla_k H^{(\alpha,a)} \cdot\nabla^\perp_i\nabla^\perp_j H\big)
\nabla_kV^{(\alpha,a)} e^q dx\nonumber\\
&+\int_{\mathbb{R}^2}\nabla^\perp_i\nabla_j \Delta^{-1}
\big(\nabla_k\nabla^\perp_i H \cdot\nabla^\perp_j H^{(\alpha,a)}
-\nabla_k\nabla^\perp_i V\nabla^\perp_jV^{(\alpha,a)} \big)
\nabla_kV^{(\alpha,a)} e^q dx\nonumber\\
&+\int_{\mathbb{R}^2}\nabla^\perp_i\nabla_j \Delta^{-1}
\big(\nabla^\perp_i\nabla^\perp_j V\nabla_kV^{(\alpha,a)}
-\nabla^\perp_i \nabla^\perp_j H \cdot\nabla_k H^{(\alpha,a)} \big)
\nabla_kV^{(\alpha,a)} e^q dx\nonumber\\
&+\frac{1}{2}\int_{\mathbb{R}^2}
|\nabla V^{(\alpha,a)}|^2
\nabla_j(\nabla^\perp_j V e^q) dx
-\frac12 \int_{\mathbb{R}^2}|\nabla H^{(\alpha,a)}|^2
\nabla^\perp_j(\nabla_j Ve^q)dx\nonumber\\
&-\int_{\mathbb{R}^2}\nabla_k H^{(\alpha,a)} \cdot\nabla^\perp_j H\nabla_j\big(
\nabla_kV^{(\alpha,a)} e^q \big)dx
+\int_{\mathbb{R}^2}\nabla^\perp_j H_i\nabla_k\nabla_j V^{(\alpha,a)}\nabla_kH^{(\alpha,a)}_i e^qdx\nonumber\\
&+\int_{\mathbb{R}^2}(\nabla^\perp_j H^{(\alpha,a)}_i\nabla_k\nabla_jV
+\nabla_k\nabla^\perp_j H_i\nabla_j V^{(\alpha,a)})
\nabla_kH^{(\alpha,a)}_i e^qdx.
\end{align}
In view of \eqref{I1}, one can see that we have gained one derivative
compared with the original expression. We still need to make the strong null condition
appear.

We start by treating  the first four lines
on the right hand side of \eqref{I1}. It is obvious that they  have the same structure so we only treat the first one.
By the $L^2$ boundness of the Riesz transform,  the first term is  bounded by
\begin{align}\label{HI1}
\sum_{1\leq i, j, k\leq 2}\|\nabla_i H^{(\alpha,a)} \cdot\nabla_k\nabla_j H-\nabla_iV^{(\alpha,a)}\nabla_k\nabla_j V\|_{L^2}
\|\nabla V^{(\alpha,a)}\|_{L^2}.
\end{align}
Now,  to see the strong   null condition,
we employ the orthogonal decomposition into radial and transverse directions:
\begin{align}\label{othor-decomp1}
&\nabla_i H^{(\alpha,a)} \cdot\nabla_k\nabla_j H-\nabla_iV^{(\alpha,a)}\nabla_k\nabla_j V
\nonumber\\
&=\nabla_i H^{(\alpha,a)} \cdot\omega \nabla_k\nabla_j H\cdot \omega
+\nabla_i H^{(\alpha,a)} \cdot\omega^\perp \nabla_k\nabla_j H\cdot \omega^\perp
-\nabla_iV^{(\alpha,a)}\nabla_k\nabla_j V\nonumber\\
&=(\nabla_iV^{(\alpha,a)}
+\nabla_i H^{(\alpha,a)}\cdot\omega) \nabla_k\nabla_j H\cdot\omega \nonumber\\
&\quad-\nabla_iV^{(\alpha,a)}
( \nabla_k\nabla_j V+\nabla_k\nabla_j H\cdot\omega)
+\nabla_i H^{(\alpha,a)} \cdot\omega^\perp \nabla_k\nabla_j H\cdot \omega^\perp.
\end{align}
Now we are ready to estimate \eqref{HI1}.
When integrating over the  domain  $ \{   r\geq \langle t\rangle/2 \}$,
 we use  \eqref{K-S-2},
Lemma \ref{lemmaWeighted-Energy} and Lemma \ref{LemmaGood-L-infy}   to  get
\begin{align}\label{II1}
&\sum_{1\leq i, j, k\leq 2}
\|\nabla_i H^{(\alpha,a)} \cdot\nabla_k\nabla_j H-\nabla_iV^{(\alpha,a)}\nabla_k\nabla_j V
\|_{L^2(r\geq \langle t\rangle/2)}
\mathcal{E}_\kappa^{\frac 12} \nonumber\\[-4mm]\nonumber\\
&\lesssim\sum_{1\leq i, j, k\leq 2}\Big(
\big\|\frac{\nabla_iV^{(\alpha,a)}+\nabla_i H^{(\alpha,a)}\cdot\omega}{\langle r-t\rangle} \big\|_{L^2}
\|\langle r-t\rangle\nabla_k\nabla_j H\cdot\omega\|_{L^\infty(r\geq \langle t\rangle/2)}
\mathcal{E}_\kappa^{\frac 12} \nonumber\\
&\qquad\qquad+\|\nabla_iV^{(\alpha,a)}\|_{L^2}
\| \nabla_k\nabla_j V+\nabla_k\nabla_j H\cdot\omega\|_{L^\infty(r\geq \langle t\rangle/2)}
\mathcal{E}_\kappa^{\frac 12} \nonumber\\[-4mm]\nonumber\\
&\qquad\qquad+\|\nabla_i H^{(\alpha,a)} \cdot\omega^\perp\|_{L^2}
\|\nabla_k\nabla_j H\cdot \omega^\perp\|_{L^\infty(r\geq \langle t\rangle/2)}
\mathcal{E}_\kappa^{\frac 12} \Big)\nonumber\\[-4mm]\nonumber\\
&\lesssim  \eta G_\kappa+C_{\eta}\langle t \rangle ^{-1}\mathcal{E}_\kappa E_{\kappa-3}
+\langle t \rangle ^{-\frac 32}\mathcal{E}_\kappa E_{\kappa-3}^{\frac 12},
\end{align}
for any  $\eta > 0$.
For the region  $\{  r\leq \langle t\rangle/2  \} $,  the bound for \eqref{HI1} is easier.
By \eqref{K-S-3} and Lemma \ref{lemmaWeighted-Energy}, one has
\begin{align}\label{II2}
&\sum_{1\leq i, j, k\leq 2}\|\nabla_i H^{(\alpha,a)} \cdot\nabla_k\nabla_j H
 -\nabla_iV^{(\alpha,a)}\nabla_k\nabla_j V\|_{L^2(r\leq \langle t\rangle/2)}
\nonumber\\
&\leq \big\| |\nabla H^{(\alpha,a)}| |\nabla^2 H|+ |\nabla V^{(\alpha,a)}| |\nabla^2 V|
\big\|_{L^2(r\leq \langle t\rangle/2)}
 \nonumber\\[-4mm]\nonumber\\
&\leq
\| \nabla H^{(\alpha,a)}\|_{L^2} \|\nabla^2 H\|_{L^\infty(r\leq \langle t\rangle/2)}
+\|\nabla V^{(\alpha,a)}\|_{L^2} \|\nabla^2 V\|_{L^\infty(r\leq \langle t\rangle/2)}
 \nonumber\\[-4mm]\nonumber\\
&\lesssim \langle t \rangle ^{-1}\mathcal{E}_\kappa^{\frac 12}  X_{\kappa-4}^{\frac 12}
\lesssim \langle t \rangle ^{-1}\mathcal{E}_\kappa^{\frac 12}  E_{\kappa-3}^{\frac 12}.
\end{align}
Consequently, inserting the above \eqref{II1} and \eqref{II2}
into \eqref{HI1} gives that
\begin{align*}
&\sum_{1\leq i, j, k\leq 2}\|\nabla_i H^{(\alpha,a)} \cdot\nabla_k\nabla_j H-\nabla_iV^{(\alpha,a)}\nabla_k\nabla_j V\|_{L^2}
\|\nabla V^{(\alpha,a)}\|_{L^2} \\
&\quad \lesssim  \eta G_\kappa+C_{\eta} \langle t \rangle^{-1} \mathcal{E}_\kappa E_{\kappa-3}^{\frac 12}.
\end{align*}
Here we have used  the \textit{a prior} assumption $E_{\kappa-3}\ll 1$.

Then, we estimate the fifth and sixth lines  of \eqref{I1}.
For these two lines, the decay
is much better. Direct integration by parts   shows that they are equal  to
\begin{align*}
&-\frac12\int_{\mathbb{R}^2}(
|\nabla H^{(\alpha,a)}|^2+|\nabla V^{(\alpha,a)}|^2)\partial_\theta V
\frac{e^q}{r\langle t-r \rangle^2}dx\nonumber\\
&+\int_{\mathbb{R}^2}\nabla_k H^{(\alpha,a)} \cdot
\partial_\theta H \nabla_kV^{(\alpha,a)}
 \frac{e^q}{r\langle t-r \rangle^2} dx.
\end{align*}
By \eqref{K-S-1}, the above can be further estimated by
\begin{align*}
\langle t \rangle ^{-\frac 32}\mathcal{E}_\kappa E_{\kappa-3}^{\frac 12}.
\end{align*}

Finally, we treat the last line of \eqref{I1}.
To show the null structure, we employ the orthogonal decomposition into radial and transverse directions to get that
\begin{align} \label{othor-decomp2}
&(\nabla^\perp_j H^{(\alpha,a)}_i\nabla_k\nabla_jV
+\nabla_k\nabla^\perp_j H_i\nabla_j V^{(\alpha,a)})
\nabla_kH^{(\alpha,a)}_i \nonumber\\[-4mm]\nonumber\\
&=(\nabla^\perp_j H^{(\alpha,a)}\cdot\omega\nabla_k\nabla_jV
+\nabla_k\nabla^\perp_j H\cdot\omega\nabla_j V^{(\alpha,a)})
\nabla_kH^{(\alpha,a)}\cdot\omega\nonumber\\[-4mm]\nonumber\\
&+(\nabla^\perp_j H^{(\alpha,a)}\cdot\omega^\perp\nabla_k\nabla_jV
+\nabla_k\nabla^\perp_j H\cdot\omega^\perp\nabla_j V^{(\alpha,a)})
\nabla_kH^{(\alpha,a)}\cdot\omega^\perp.
\end{align}
For the expression inside the bracket on the first line
of the right-hand side of \eqref{othor-decomp2},
one can further reorganize them as follows:
\begin{align*} 
&\nabla^\perp_j H^{(\alpha,a)}\cdot\omega\nabla_k\nabla_jV
+\nabla_k\nabla^\perp_j H\cdot\omega\nabla_j V^{(\alpha,a)} \nonumber\\[-4mm]\nonumber\\
&=(\nabla^\perp_j H^{(\alpha,a)}\cdot\omega
+\nabla^\perp_j V^{(\alpha,a)}) \nabla_k\nabla_jV
+(\nabla_k\nabla^\perp_j H\cdot\omega +\nabla_k\nabla^\perp_j V)
\nabla_j V^{(\alpha,a)} \nonumber\\[-4mm]\nonumber\\
&\quad -\nabla^\perp_j V^{(\alpha,a)}\nabla_k\nabla_jV
-\nabla_k\nabla^\perp_j V\nabla_j V^{(\alpha,a)}\nonumber\\[-4mm]\nonumber\\
&=(\nabla^\perp_j H^{(\alpha,a)}\cdot\omega
+\nabla^\perp_j V^{(\alpha,a)}) \nabla_k\nabla_jV
+(\nabla_k\nabla^\perp_j H\cdot\omega +\nabla_k\nabla^\perp_j V)
\nabla_j V^{(\alpha,a)}.
\end{align*}
Here we have used the fact that
\begin{align*}
\nabla^\perp_j V^{(\alpha,a)}\nabla_k\nabla_jV
+\nabla_k\nabla^\perp_j V\nabla_j V^{(\alpha,a)}=0.
\end{align*}
Hence, for the last line of \eqref{I1}, by \eqref{K-S-2},
Lemma \ref{lemmaWeighted-Energy} and Lemma \ref{LemmaGood-L-infy},
we can estimate the integral  over the region  $ \{ r\geq \langle t\rangle/2 \} $ by
\begin{align*}
&\|(\nabla^\perp_j H^{(\alpha,a)}\cdot\omega
+\nabla^\perp_j V^{(\alpha,a)}) \nabla_k\nabla_jV\|_{L^2(r\geq \langle t\rangle/2)}
\|\nabla_kH^{(\alpha,a)}\cdot\omega\|_{L^2}\nonumber\\[-4mm]\nonumber\\
&+\|(\nabla_k\nabla^\perp_j H\cdot\omega +\nabla_k\nabla^\perp_j V)
\nabla_j V^{(\alpha,a)}\|_{L^2(r\geq \langle t\rangle/2)}
\|\nabla_kH^{(\alpha,a)}\cdot\omega\|_{L^2} \nonumber\\[-4mm]\nonumber\\
&+\|(\nabla^\perp_j H^{(\alpha,a)}\cdot\omega^\perp\nabla_k\nabla_jV
+\nabla_k\nabla^\perp_j H\cdot\omega^\perp\nabla_j V^{(\alpha,a)})
\nabla_kH^{(\alpha,a)}\cdot\omega^\perp\|_{L^1(r\geq \langle t\rangle/2)}\nonumber\\[-4mm]\nonumber\\
&\lesssim
\big\|\frac{\nabla^\perp_j H^{(\alpha,a)}\cdot\omega
+\nabla^\perp_j V^{(\alpha,a)}}{\langle t-r\rangle} \big\|_{L^2}
\|\langle t-r\rangle\nabla^2 V\|_{L^\infty(r\geq \langle t\rangle/2)}
\mathcal{E}_{\kappa}^{\frac 12}\\
&\quad+\|\nabla_k\nabla^\perp_j H\cdot\omega +\nabla_k\nabla^\perp_j V\|_{L^\infty(r\geq \langle t\rangle/2)}
\|\nabla V^{(\alpha,a)}\|_{L^2}
\mathcal{E}_\kappa^{\frac 12}\nonumber\\[-4mm]\nonumber\\
&\quad+\big\|\frac{\nabla^\perp_k H^{(\alpha,a)}\cdot\omega^\perp}
{\langle t-r\rangle} \big\|_{L^2}
\big\|\langle t-r\rangle |\nabla U^{(\alpha,a)}| |\nabla^2 U|
\big\|_{L^2(r\geq \langle t\rangle/2)}\\
&\lesssim  \eta G_\kappa+C_{\eta}\langle t \rangle ^{-1}\mathcal{E}_\kappa E_{\kappa-3}
+\langle t \rangle ^{-\frac 32}\mathcal{E}_\kappa E_{\kappa-3}^{\frac 12}.
\end{align*}
In the region  $\{ r\leq \langle t\rangle/2 \}$, we can easily estimate the last line of \eqref{I1},
similarly  to \eqref{II2},  to deduce that  it is   controlled by
\begin{equation*}
\langle t \rangle ^{-1}\mathcal{E}_\kappa E_{\kappa-3}^{\frac 12}.
\end{equation*}
Thus we gather the estimates in step 1 to conclude that
\begin{equation*}
I_1 \lesssim  \eta G_\kappa+C_{\eta} \langle t \rangle^{-1} \mathcal{E}_\kappa E_{\kappa-3}^{\frac 12}.
\end{equation*}

Step 2: Estimate of the lower order term $I_2$.

We introduce $\tilde{f}^{\alpha a}$ given by
\begin{align*}
\tilde{f}_{ijk}^{\alpha a}=\sum_{\tiny\begin{matrix} \beta+\gamma=\alpha, b + c = a
\\|\beta|+|b|,|\gamma|+|c|<|\alpha|+|a| \end{matrix}}
C_{\alpha}^\beta C_a^b
\nabla_k(\nabla_iV^{(\beta,b)} \nabla_jV^{(\gamma,c)}
-\nabla_iH^{(\beta,b)}\cdot \nabla_jH^{(\gamma,c)}),
\end{align*}
and
\begin{align*}
\tilde{f}^{\alpha a}_2=\sum_{\tiny\begin{matrix} \beta+\gamma=\alpha, b + c = a
\\|\beta|+|b|,|\gamma|+|c|<|\alpha|+|a| \end{matrix}}
C_\alpha^\beta C_a^b
\nabla(\nabla^\perp H^{(\beta,b)}\nabla V^{(\gamma,c)} ).
\end{align*}
Using these notations, we can control $I_2$
from the $L^2$ boundness of the Riesz transform by
\begin{align*}
 \sum_{ijk} \|\tilde{f}_{ijk}^{\alpha a}\|_{L^2}\|\nabla V^{(\alpha,a)}\|_{L^2}
+\|\tilde{f}_{2}^{\alpha a}\|_{L^2}\|\nabla H^{(\alpha,a)}\|_{L^2}.
\end{align*}
Thus to estimate $I_2$, we only need to take care of $\tilde{f}_{ijk}^{\alpha a}$ and $\tilde{f}_{2}^{\alpha a}$.

First we treat $\tilde{f}_{ijk}^{\alpha a}$. One easily has
\begin{align*}
&\nabla_k(\nabla_iV^{(\beta,b)} \nabla_jV^{(\gamma,c)}
-\nabla_iH^{(\beta,b)}\cdot \nabla_jH^{(\gamma,c)})\\
&=\nabla_k\nabla_iV^{(\beta,b)} \nabla_jV^{(\gamma,c)}
-\nabla_k\nabla_iH^{(\beta,b)}\cdot \nabla_jH^{(\gamma,c)}\\
&\quad+\nabla_iV^{(\beta,b)} \nabla_k\nabla_jV^{(\gamma,c)}
-\nabla_iH^{(\beta,b)}\cdot \nabla_k\nabla_jH^{(\gamma,c)}.
\end{align*}
In view of fact that the last two lines in the above are similar,
we only concentrate on the first one.
To estimate of $\| \tilde{f}_{ijk}^{\alpha a}\|_{L^2}$,
we still divide the integral domain $\mathbb{R}^2$ into two different subdomains.

In the region  $\{  r\leq \langle t\rangle/2  \}$,
we have
\begin{align} \label{f_tilde2_}
\|\tilde{f}_{ijk}^{\alpha a}\|_{L^2(r\leq \langle t\rangle/2)}
\leq
\sum_{\tiny\begin{matrix} \beta+\gamma=\alpha, b + c = a
\\|\beta|+|b|,|\gamma|+|c|<|\alpha|+|a| \end{matrix}}
\big\| |\nabla^2 U^{(\beta,b)}| |\nabla U^{(\gamma,c)}|\big\|_{L^2(r\leq \langle t\rangle/2)}.
\end{align}
Here and in what follows, thanks to the fully nonlinear effect of the new formulation,
we always have one derivative in the lower order terms.
Thus one has room to use the weighted $L^2$ norm $X_{\kappa}$
even though  we are facing the derivative loss $X_{\kappa-2}\lesssim E_{\kappa-1}$.

For \eqref{f_tilde2_}, if $|c|+|\gamma|\leq|b|+|\beta|$,
then there holds $|b|+|\beta|+2\leq \kappa$, $|\gamma|+|c|+3\leq [(|\alpha|+|a|)/2]+3\leq \kappa-4$.
By \eqref{K-S-3} and Lemma \ref{lemmaWeighted-Energy},
we have
\begin{align*}
&\sum_{\tiny\begin{matrix} \beta+\gamma=\alpha, b + c = a
\\|\gamma|+|c|\leq |\beta|+|b|<|\alpha|+|a| \end{matrix}}
 \big\| |\nabla^2 U^{(\beta,b)}|
        |\nabla U^{(\gamma,c)}| \big\|_{L^2(r\leq \langle t\rangle/2)}\\
&\lesssim \sum_{\tiny\begin{matrix} \beta+\gamma=\alpha, b + c = a
\\|\gamma|+|c|\leq |\beta|+|b|<|\alpha|+|a| \end{matrix}}
\langle t \rangle^{-1}
\| \nabla^2 U^{(\beta,b)}\|_{L^2}
\| \langle t\rangle \nabla U^{(\gamma,c)}\|_{L^\infty(r\leq \langle t\rangle/2)}\\
&\lesssim \sum_{\tiny\begin{matrix} \beta+\gamma=\alpha, b + c = a
\\|\gamma|+|c|\leq |\beta|+|b|<|\alpha|+|a| \end{matrix}}
\langle t\rangle^{-1} \mathcal{E}_{|\beta|+|b|+2}^{\frac 12} X_{|\gamma|+|c|+3 }^{\frac 12}
\lesssim \langle t \rangle^{-1}  \mathcal{E}_{\kappa}^{\frac 12} E_{\kappa-3}^{\frac 12}.
\end{align*}
If $|b|+|\beta| < |c|+|\gamma|$,
then $|\gamma|+|c|+1\leq \kappa$, $|b|+|\beta|+4\leq [(|\alpha|+|a|)/2]+4\leq \kappa-4$.
We can similarly obtain
\begin{align*}
& \sum_{\tiny\begin{matrix} \beta+\gamma=\alpha, b + c = a
\\|\beta|+|b|<|\gamma|+|c|<|\alpha|+|a| \end{matrix}}
 \big\| |\nabla^2 U^{(\beta,b)}|
        |\nabla U^{(\gamma,c)}| \big\|_{L^2(r\leq \langle t\rangle/2)}\\
&\lesssim
 \sum_{\tiny\begin{matrix} \beta+\gamma=\alpha, b + c = a
\\|\beta|+|b|<|\gamma|+|c|<|\alpha|+|a| \end{matrix}}
\langle t\rangle^{-1}\mathcal{E}_{|\gamma|+|c|+1}^{\frac 12} X_{|\beta|+|b|+4 }^{\frac 12}
\lesssim \langle t \rangle^{-1}  \mathcal{E}_{\kappa}^{\frac 12} E_{\kappa-3}^{\frac 12}.
\end{align*}
Thus we arrive at
\begin{align*}
\|\tilde{f}_{ijk}^{\alpha a}\|_{L^2(r\leq \langle t\rangle/2)}
\lesssim \langle t \rangle^{-1}  \mathcal{E}_{\kappa}^{\frac 12} E_{\kappa-3}^{\frac 12}.
\end{align*}

In the region  $\{ r\geq \langle t\rangle/2 \} $,
we need to employ the null structure to get some  extra decay in time.
A natural idea is to use a variant version of Lemma \ref{Lem-Good-f},  however  this doesn't work
due to the derivative loss $Y_{\kappa-2}\lesssim E_{\kappa-1}$.
To solve this problem, we will use  the ghost weight energy at  all  derivative  levels.

For $\tilde{f}_{ijk}^{\alpha a}$, we organize similarly to the decomposition  \eqref{othor-decomp1}:
\begin{align*}
&\nabla_k\nabla_iV^{(\beta,b)} \nabla_jV^{(\gamma,c)}
-\nabla_k\nabla_iH^{(\beta,b)}\cdot \nabla_jH^{(\gamma,c)}
\nonumber\\
&=(\nabla_k\nabla_iV^{(\beta,b)}
 +\nabla_k\nabla_iH^{(\beta,b)}\cdot\omega)
 \nabla_jV^{(\gamma,c)}\\
&\quad-\nabla_k\nabla_iH^{(\beta,b)}\cdot\omega
 (\nabla_jV^{(\gamma,c)}+
 \nabla_jH^{(\gamma,c)}\cdot\omega)\\
&\quad-\nabla_k\nabla_iH^{(\beta,b)}\cdot\omega^\perp
 \nabla_j H^{(\gamma,c)}\cdot\omega^\perp.
\end{align*}
Thus
\begin{align} \label{f_tilde2}
&\|\tilde{f}_{ijk}^{\alpha a}\|_{L^2( r\leq \langle t\rangle/2 )} \|\nabla V^{(\alpha,a)}\|_{L^2}
  \nonumber  \leq
\sum_{\tiny\begin{matrix} \beta+\gamma=\alpha, b + c = a
\\|\beta|+|b|,|\gamma|+|c|<|\alpha|+|a| \end{matrix}} \nonumber\\
& \qquad\big(
\|(\nabla_k\nabla_iV^{(\beta,b)}
 +\nabla_k\nabla_iH^{(\beta,b)}\cdot\omega)
 \nabla_jV^{(\gamma,c)}\|_{L^2(r\leq \langle t\rangle/2)}\mathcal{E}_\kappa^\frac{1}{2}
\nonumber\\
&\qquad+\|\nabla_k\nabla_iH^{(\beta,b)}\cdot\omega
 (\nabla_jV^{(\gamma,c)}+
 \nabla_jH^{(\gamma,c)}\cdot\omega)\|_{L^2(r\leq \langle t\rangle/2)}\mathcal{E}_\kappa^\frac{1}{2}
\nonumber\\[-4mm]\nonumber\\
&\qquad+\|\nabla_k\nabla_iH^{(\beta,b)}\cdot\omega^\perp
\nabla_j H^{(\gamma,c)}\cdot\omega^\perp\|_{L^2(r\leq \langle t\rangle/2)}\mathcal{E}_\kappa^\frac{1}{2} \big).
\end{align}
When $|\gamma|+|c|\leq|\beta|+|b|$, by \eqref{K-S-2},
Lemma \ref{lemmaWeighted-Energy} and Lemma \ref{LemmaGood-L-infy},
the right hand side of \eqref{f_tilde2} can be  bounded by
\begin{align*}
&\qquad\sum_{1\leq i,j,k\leq 2}
\sum_{\tiny\begin{matrix} |\beta|+|b|<|\alpha|+|a|
\\|\gamma|+|c|\leq [(|\alpha|+|a|)/2] \end{matrix}} \\
&\quad\big(\|\frac{\nabla_k\nabla_iV^{(\beta,b)}
 +\nabla_k\nabla_iH^{(\beta,b)}\cdot\omega}
{\langle t-r\rangle}\|_{L^2}
\|\langle t-r\rangle\nabla_jV^{(\gamma,c)}\|_{L^\infty(r\leq \langle t\rangle/2)}
\mathcal{E}_k^\frac{1}{2} \nonumber\\[-4mm]\nonumber\\
&+\|\nabla_k\nabla_iH^{(\beta,b)}\cdot\omega\|_{L^2}
 \|(\nabla_jV^{(\gamma,c)}+
 v_jH^{(\gamma,c)}\cdot\omega)\|_{L^\infty(r\leq \langle t\rangle/2)}
\mathcal{E}_\kappa^\frac{1}{2} \nonumber\\[-4mm]\nonumber\\
&+\|\nabla_k\nabla_iH^{(\beta,b)}\cdot\omega^\perp\|_{L^2}
 \|\nabla_j H^{(\gamma,c)}\cdot\omega^\perp\|_{L^\infty(r\leq \langle t\rangle/2)}
\mathcal{E}_\kappa^\frac{1}{2} \big) \nonumber\\[-4mm]\nonumber\\
&\lesssim \eta G_\kappa+ C_{\eta}\langle t \rangle ^{-1}\mathcal{E}_\kappa E_{\kappa-3}
+\langle t \rangle ^{-\frac 32}\mathcal{E}_\kappa E_{\kappa-3}^{\frac 12}.
\end{align*}
Repeating  the above procedure, we can control the right hand side of \eqref{f_tilde2}
in  the case   $|\beta|+|b|\leq |\gamma|+|c|$ by
\begin{align*}
\eta G_\kappa +C_{\eta}\langle t \rangle ^{-1}\mathcal{E}_\kappa E_{\kappa-3}
+\langle t \rangle ^{-\frac 32}\mathcal{E}_\kappa E_{\kappa-3}^{\frac 12}.
\end{align*}
Thus,  we   get
\begin{align*}
\|\tilde{f}_{ijk}^{\alpha a}\|_{L^2}\|\nabla V^{(\alpha,a)}\|_{L^2}
\lesssim \eta G_\kappa +C_{\eta}\langle t \rangle ^{-1}\mathcal{E}_\kappa E_{\kappa-3}^{\frac 12}.
\end{align*}
Here we have used the \textit{a priori} estimate that $E_{\kappa-3}\ll 1$.

Next we turn our attention to $\| \tilde{f}^{\alpha a}_2 \|_{L^2}$.
Since the estimate is similar to $\| \tilde{f}^{\alpha a}_{ijk} \|_{L^2}$,
we only sketch the main line.
We still divide the integral domain $\mathbb{R}^2$ into two different parts to estimate
them separately. For the integral over the  domain  $ \{ r\leq \langle t\rangle/2  \} $, the estimate
is exactly the same as the one for  $\| \tilde{f}^{\alpha a}_{ijk} \|_{L^2(r\leq \langle t\rangle/2)}$.
For the region  $\{  r\geq \langle t\rangle/2\} $, we still need to make full use of the  appropriate null structure.
The estimate is similar to one for  $\| \tilde{f}^{\alpha a}_{ijk} \|_{L^2(r\geq \langle t\rangle/2)}$
once the null structure of $\tilde{f}^{\alpha a}_2$ is present.
Hence we only  show the  strong null structure of $\tilde{f}^{\alpha a}_2$ below.

Employing the orthogonal decomposition into radial and transverse directions,  any term
in the sum  defining  $I_2$ can be decomposed as:
\begin{align*}
&\sum_{\tiny\begin{matrix} \beta+\gamma=\alpha, b + c = a
\\|\beta|+|b|,|\gamma|+|c|<|\alpha|+|a| \end{matrix}}
C_\alpha^\beta C_a^b
\nabla_i(\nabla^\perp_j H^{(\beta,b)}\nabla_j V^{(\gamma,c)} ) \\
&=\sum_{\tiny\begin{matrix} \beta+\gamma=\alpha, b + c = a
\\|\beta|+|b|,|\gamma|+|c|<|\alpha|+|a| \end{matrix}}
C_\alpha^\beta C_a^b
\nabla_i\nabla^\perp_j H^{(\beta,b)}\nabla_j V^{(\gamma,c)}
+\nabla^\perp_j H^{(\beta,b)}\nabla_i\nabla_j V^{(\gamma,c)}\\
&=\sum_{\tiny\begin{matrix} \beta+\gamma=\alpha, b + c = a
\\|\beta|+|b|,|\gamma|+|c|<|\alpha|+|a| \end{matrix}}  \Big[  C_\alpha^\beta C_a^b
\big(\nabla_i\nabla^\perp_j H^{(\beta,b)}\cdot \omega \nabla_j V^{(\gamma,c)}
+\nabla^\perp_j H^{(\beta,b)}\cdot\omega\nabla_i\nabla_j V^{(\gamma,c)}\big)
\omega\\
&\quad \quad \quad \quad +
C_\alpha^\beta C_a^b
\big(\nabla_i\nabla^\perp_j H^{(\beta,b)}\cdot\omega^\perp\nabla_j V^{(\gamma,c)}
+\nabla^\perp_j H^{(\beta,b)}\cdot\omega^\perp\nabla_i\nabla_j V^{(\gamma,c)}
\big)\cdot\omega^\perp  \Big]   \\
&=\sum_{\tiny\begin{matrix} \beta+\gamma=\alpha, b + c = a
\\|\beta|+|b|,|\gamma|+|c|<|\alpha|+|a| \end{matrix}}   \Big[
C_\alpha^\beta C_a^b
\big(\nabla_i\nabla^\perp_j H^{(\beta,b)}\cdot \omega
      +\nabla_i\nabla^\perp_j V^{(\beta,b)}\big) \nabla_j V^{(\gamma,c)}\\
& \quad \quad \quad \quad  +C_\alpha^\beta C_a^b
\big(\nabla^\perp_j H^{(\beta,b)}\cdot\omega+\nabla^\perp_j V^{(\beta,b)}\big)
\nabla_i\nabla_j V^{(\gamma,c)}\\
& \quad \quad \quad \quad  +
C_\alpha^\beta C_a^b
\big(\nabla_i\nabla^\perp_j H^{(\beta,b)}\cdot\omega^\perp\nabla_j V^{(\gamma,c)}
+\nabla^\perp_j H^{(\beta,b)}\cdot\omega^\perp\nabla_i\nabla_j V^{(\gamma,c)}
\big)\cdot\omega^\perp  \Big].
\end{align*}
Here we have used the fact that
\begin{align*}
&\sum_{\tiny\begin{matrix} \beta+\gamma=\alpha, b + c = a
\\|\beta|+|b|,|\gamma|+|c|<|\alpha|+|a| \end{matrix}}
C_\alpha^\beta C_a^b
\nabla_i\nabla^\perp_j V^{(\beta,b)} \nabla_j V^{(\gamma,c)}\\
&+\sum_{\tiny\begin{matrix} \beta+\gamma=\alpha, b + c = a
\\|\beta|+|b|,|\gamma|+|c|<|\alpha|+|a| \end{matrix}}
C_\alpha^\beta C_a^b
\nabla_i\nabla_j V^{(\gamma,c)} \nabla^\perp_j V^{(\beta,b)}=0.
\end{align*}
Thus, we can estimate $\tilde{f}^{\alpha a}_2$ as $\tilde{f}^{\alpha a}_{ijk}$
to get that
\begin{align*}
\|\tilde{f}_{2}^{\alpha a}\|_{L^2} \|\nabla H^{(\alpha,a)}\|_{L^2}
\lesssim \eta G_\kappa +C_{\eta}\langle t \rangle ^{-1}\mathcal{E}_\kappa E_{\kappa-3}^{\frac 12}.
\end{align*}

Finally, we gather our estimates  for \eqref{EN-Es1} to derive that
\begin{eqnarray}
&&\frac 12\frac{d}{dt}
\int_{\mathbb{R}^2}(|\nabla V^{(\alpha,a)}|^2
    +|\nabla H^{(\alpha,a)}|^2)e^qdx\nonumber\\
&&-\int_{\mathbb{R}^2}
\mu\nabla\Delta\sum\limits_{l=0}^{\alpha} (-1)^{\alpha-l}V^{(l,a)}
\cdot\nabla V^{(\alpha,a)}e^q dx \nonumber\\
&&+\frac 12 \int_{\mathbb{R}^2}\frac{|\nabla V^{(\alpha,a)}
+\nabla H^{(\alpha,a)}\cdot\omega|^2
+|\nabla H^{(\alpha,a)}\cdot\omega^\perp|^2}
{\langle t-r \rangle^2}e^qdx\nonumber\\
&&\lesssim \eta G_\kappa +C_{\eta}\langle t \rangle ^{-1}\mathcal{E}_\kappa E_{\kappa-3}^{\frac 12}.
\end{eqnarray}
Moreover, the viscosity terms can be estimated as follows:
\begin{align*}
&-\int_{\mathbb{R}^2}
\mu\nabla\Delta\sum\limits_{l=0}^{\alpha}  C_\alpha^l (-1)^{\alpha-l}V^{(l,a)}
\cdot\nabla V^{(\alpha,a)}e^q dx \nonumber\\
&=\mu\int_{\mathbb{R}^2}|\Delta V^{(\alpha,a)}|^2e^q dx
+\int_{\mathbb{R}^2}
\mu\Delta\sum\limits_{l=0}^{\alpha-1} C_\alpha^l (-1)^{\alpha-l}V^{(l,a)}
\cdot\Delta V^{(\alpha,a)}e^q dx \nonumber\\
&\quad +\int_{\mathbb{R}^2}
\mu\Delta\sum\limits_{l=0}^{\alpha} C_\alpha^l (-1)^{\alpha-l}V^{(l,a)}
\nabla V^{(\alpha,a)}\cdot\nabla e^q dx \nonumber\\
&\geq\mu\int_{\mathbb{R}^2}|\Delta V^{(\alpha,a)}|^2e^q dx
-\mu \sum\limits_{l=0}^{\alpha-1} (C_\alpha^l)^2
\int_{\mathbb{R}^2}|\Delta V^{(l,a)}|^2 e^q dx
-\frac 14\mu\int_{\mathbb{R}^2} |\Delta V^{(\alpha,a)}|^2 e^qdx\nonumber\\
&\quad-\frac 14\mu\sum\limits_{l=0}^{\alpha} (C_\alpha^l)^2
\int_{\mathbb{R}^2}
|\Delta V^{(l,a)}|^2 e^{q} dx
 -\mu\int_{\mathbb{R}^2}
|\nabla V^{(\alpha,a)}|^2 e^{q} dx\nonumber\\
&\geq\frac 12\mu\int_{\mathbb{R}^2}|\Delta V^{(\alpha,a)}|^2e^q dx
-2\mu\sum\limits_{l=0}^{\alpha-1}(C_\alpha^l)^2\int_{\mathbb{R}^2}
|\Delta V^{(l,a)}|^2 e^q dx
-\mu\int_{\mathbb{R}^2}
|\nabla V^{(\alpha,a)}|^2 e^{q}dx.
\end{align*}
Consequently,
\begin{eqnarray*}
&&\frac 12\frac{d}{dt}
\int_{\mathbb{R}^2}(|\nabla V^{(\alpha,a)}|^2
    +|\nabla H^{(\alpha,a)}|^2)e^qdx\nonumber\\
&&+\frac 12\mu\int |\Delta V^{(\alpha,a)}|^2 e^qdx
-2\mu\sum\limits_{l=0}^{\alpha-1}(C_\alpha^l)^2\int_{\mathbb{R}^2}
|\Delta V^{(l,a)}|^2e^qdx
-\mu\int_{\mathbb{R}^2}
|\nabla V^{(\alpha,a)}|^2 e^q dx \nonumber\\
&&+\frac 12 \int_{\mathbb{R}^2}\frac{|\nabla V^{(\alpha,a)}
+\nabla H^{(\alpha,a)}\cdot\omega|^2
+|\nabla H^{(\alpha,a)}\cdot\omega^\perp|^2}
{\langle t-r \rangle^2}e^qdx\nonumber\\
&&\leq
\eta G_{\kappa} +
C_{\eta}\langle t \rangle ^{-1}\mathcal{E}_\kappa E_{\kappa-3}^{\frac 12}.
\end{eqnarray*}
Integrating   both sides of the above inequality in time on $[0,t)$, we get
\begin{align} \label{ModHighEneEs}
&\frac 12
\int_{\mathbb{R}^2}(|\nabla V^{(\alpha,a)}(t)|^2
    +|\nabla H^{(\alpha,a)}(t)|^2)e^qdx
+\frac 12\mu \int_0^t\int_{\mathbb{R}^2} |\Delta V^{(\alpha,a)}(\tau)|^2 e^qdxd\tau
\nonumber\\
&-2\mu\sum\limits_{l=0}^{\alpha-1}
(C_\alpha^l)^2\int_0^t\int_{\mathbb{R}^2}
|\Delta V^{(l,a)}(\tau)|^2e^qdxd\tau
-\mu\int_0^t\int_{\mathbb{R}^2}
|\nabla V^{(\alpha,a)}(\tau)|^2 e^q dxd\tau \nonumber\\
&+\frac 12  \int_0^t\int_{\mathbb{R}^2}\frac{|\nabla V^{(\alpha,a)}(\tau)
+\nabla H^{(\alpha,a)}(\tau)\cdot\omega|^2
+|\nabla H^{(\alpha,a)}(\tau)\cdot\omega^\perp|^2}
{\langle \tau-r \rangle^2}e^qdxd\tau \nonumber\\
&\leq \eta \int_0^t G_\kappa(\tau)d\tau+C_{\eta}
\int_0^t \langle \tau \rangle ^{-1}\mathcal{E}_\kappa(\tau) E_{\kappa-3}^{\frac 12}(\tau)d\tau
\nonumber\\
&\quad +\frac 12\int_{\mathbb{R}^2}(|\nabla V^{(\alpha,a)}(0)|^2
    +|\nabla H^{(\alpha,a)}(0)|^2)e^qdx .
\end{align}
 We notice that the  two  viscous terms on the second line  have a negative sign.
 These two terms will be absorbed by the viscous dissipation coming from the lower-orders and the standard energy estimate of the next subsection  (see \eqref{HighEnergyEs2}).

\subsection{Highest-order Standard Energy Estimate}
Now we proceed with  the highest-order standard energy estimate.
Here, we have one less regular derivative to estimate and  we will not  use the ghost weight.
Hence,  we don't need to handle the commutators between the ghost weight and the viscosity terms,
but only handle the commutators between the scaling operator and the viscosity terms.
We remark that  here the estimate of the nonlinearities is  slightly different because of the absence of
the extra  regular derivative.

Let $\kappa \geq12$, $|\alpha|+|a| \leq \kappa-1$ and
let us take  the $L^2$ inner product of the first and the second equation of \eqref{VisElas-Gamma}
with $V^{(\alpha,a)}$ and $H^{(\alpha,a)}$, respectively. Then adding  up the resulting equations, we get
\begin{eqnarray}\label{EN-EsH}
&&\frac 12\frac{d}{dt}
\int_{\mathbb{R}^2}(|V^{(\alpha,a)}|^2+|H^{(\alpha,a)}|^2)dx
-\int_{\mathbb{R}^2}\mu\Delta\sum\limits_{l=0}^{\alpha} C_\alpha^l (-1)^{\alpha-l}V^{(l,a)}
\cdot V^{(\alpha,a)} dx \nonumber\\
&&=\int_{\mathbb{R}^2} f^1_{\alpha a}  V^{(\alpha,a)}dx
+  \int_{\mathbb{R}^2} f^2_{\alpha a} \cdot H^{(\alpha,a)}dx.
\end{eqnarray}
Since the estimate for the first term and the second one on the right hand side of \eqref{EN-EsH}
are very similar, we only give the details for
the first  one.

It follows easily from the $L^2$ boundness of the Riesz transform that
\begin{align*}
\int_{\mathbb{R}^2} f^1_{\alpha a}  V^{(\alpha,a)}dx
\lesssim
\sum_{\tiny\begin{matrix} \beta+\gamma=\alpha\\ b + c = a \end{matrix}}
\sum_{1\leq i,j \leq 2}
\|\nabla_i V^{(\beta,b)} \nabla_j V^{(\gamma,c)}
  -\nabla_i H^{(\beta,b)}\cdot \nabla_j H^{(\gamma,c)}\|_{L^2}
\|V^{(\alpha,a)}\|_{L^2}.
\end{align*}

In the region  $\{  r\leq \langle t\rangle/2\} $, we have
\begin{align*}
&\sum_{\tiny\begin{matrix} \beta+\gamma=\alpha\\ b + c = a \end{matrix}}
\sum_{1\leq i,j \leq 2}
\|\nabla_i V^{(\beta,b)} \nabla_j V^{(\gamma,c)}
  -\nabla_i H^{(\beta,b)}\cdot \nabla_j H^{(\gamma,c)}\|_{L^2(r\leq \langle t\rangle/2)}\\
&\ \lesssim
\sum_{\tiny\begin{matrix} \beta+\gamma=\alpha\\ b + c = a \end{matrix}}
\big\||\nabla U^{(\beta,b)}||\nabla U^{(\gamma,c)}|
 \big\|_{L^2(r\leq \langle t \rangle/2)} \\
&\ \lesssim
\sum_{\tiny\begin{matrix}b + c = a,\beta+\gamma=\alpha
 \\|b|+|\beta|\geq|c|+|\gamma|\end{matrix}}
\big\| |\nabla U^{(\beta,b)}| |\nabla U^{(\gamma,c)}|
      \big\|_{L^2(r\leq \langle t \rangle/2)}.
\end{align*}
Here,  we used the symmetry between the index
in the last inequality. Note that due to the derivative loss
$X_{\kappa-2}\lesssim E_{\kappa-1}$ and
since $|b|+|\beta|\geq |c|+|\gamma|$,
one has $|\gamma|+|c|+3\leq [(|\alpha|+|a|)/2]+3\leq \kappa-4$.
By \eqref{K-S-3} and Lemma \ref{lemmaWeighted-Energy},
the above quantities can be  controlled by
\begin{align*}
&\sum_{\tiny\begin{matrix}b + c = a,\beta+\gamma=\alpha
 \\|b|+|\beta|\geq|c|+|\gamma|\end{matrix}}
\langle t \rangle^{-1} \| \nabla U^{(\beta,b)}\|_{L^2}
   \|\langle t \rangle \nabla U^{(\gamma,c)}
   \|_{L^\infty(r\leq \langle t \rangle/2)} \\
&\quad\lesssim\sum_{\tiny\begin{matrix}b + c = a,\beta+\gamma=\alpha
 \\|b|+|\beta|\geq|c|+|\gamma|\end{matrix}}
\langle t \rangle^{-1}
\mathcal{E}^{\frac{1}{2}}_{|\beta|+|b|+1}X^{\frac{1}{2}}_{|\gamma|+|c|+3} \\
&\quad\leq \langle t \rangle^{-1} \mathcal{E}^{\frac{1}{2}}_{\kappa} X^{\frac{1}{2}}_{\kappa-4}
\lesssim \langle t \rangle^{-1} \mathcal{E}^{\frac{1}{2}}_{\kappa} E^{\frac{1}{2}}_{\kappa-3}.
\end{align*}
In the region  $\{  r\geq \langle t \rangle /2\} $,
we need to employ the null structure to get  extra  time  decay.
An important trick here is that we need to use the appropriate null structure.
The situation is similar to the estimate of $I_2$ in  the last subsection.
A natural idea is to use Lemma \ref{Lem-Good-f}. But this doesn't work due to the derivative loss $Y_{\kappa-2}\lesssim E_{\kappa-1}$.
To solve this problem, we combine the highest-order standard energy estimate and the highest-order modified energy estimate.
More precisely,  we will use the good term $G_\kappa$  that comes from the
ghost weight energy obtained in the modified energy estimate.

Employ the orthogonal decomposition into radial and transverse directions,
we have
\begin{align*}
&\nabla_i V^{(\beta,b)}\nabla_jV^{(\gamma,c)}
-\nabla_i H^{(\beta,b)}\cdot
\nabla_j H^{(\gamma,c)} \nonumber\\
&=\nabla_i V^{(\beta,b)}\nabla_jV^{(\gamma,c)}
-\nabla_i H^{(\beta,b)}\cdot\omega
\nabla_j H^{(\gamma,c)}\cdot\omega
 -\nabla_i H^{(\beta,b)}\cdot\omega^\perp
\nabla_j H^{(\gamma,c)}\cdot\omega^\perp\nonumber\\
&=(\nabla_i V^{(\beta,b)}+\nabla_i H^{(\beta,b)}\cdot\omega)
\nabla_jV^{(\gamma,c)}
-\nabla_i H^{(\beta,b)}\cdot\omega
(\nabla_jV^{(\gamma,c)}
+\nabla_j H^{(\gamma,c)}\cdot\omega) \nonumber\\
&\quad -\nabla_i H^{(\beta,b)}\cdot\omega^\perp
\nabla_j H^{(\gamma,c)}\cdot\omega^\perp.
\end{align*}
Consequently,
\begin{align}\label{est-good-f}
&\sum_{\tiny\begin{matrix} \beta+\gamma=\alpha\\ b + c = a \end{matrix}}
\sum_{1\leq i,j \leq 2}
\|\nabla_i V^{(\beta,b)} \nabla_j V^{(\gamma,c)}
  -\nabla_i H^{(\beta,b)}\cdot \nabla_j H^{(\gamma,c)}\|_{L^2(r\geq \langle t \rangle /2)}
\|V^{(\alpha,a)}\|_{L^2} \nonumber\\
&\lesssim
\sum_{\tiny\begin{matrix} \beta+\gamma=\alpha\\ b + c = a \end{matrix}}
\sum_{1\leq i \leq 2}
\big\| |\nabla_i H^{(\beta,b)}\cdot \omega+\nabla_i V^{(\beta,b)}|
|\nabla U^{(\gamma,c)}| \big\|_{L^2(r\geq \langle t \rangle/2)} E^{\frac12}_{\kappa-1} \nonumber\\
&\quad+\sum_{\tiny\begin{matrix} \beta+\gamma=\alpha\\ b + c = a \end{matrix}}
\sum_{1\leq i,j \leq 2}
\|\nabla_i H^{(\beta,b)}\cdot\omega^\perp
\nabla_j H^{(\gamma,c)}\cdot\omega^\perp\|_{L^2(r\geq \langle t \rangle/2)} E^{\frac12}_{\kappa-1}.
\end{align}
In the above inequality, we  used the symmetry between the index $b$ and $c$ and the symmetry between $\beta$ and $\gamma$. For \eqref{est-good-f}, if $|\beta|+|b|\geq |\gamma|+|c|$, by \eqref{K-S-2}, Lemma \ref{lemmaWeighted-Energy} and Lemma \ref{LemmaGood-L-infy}, it can be further bounded by
\begin{align*}
&\sum_{\tiny\begin{matrix} \beta+\gamma=\alpha, b + c = a\\ |\beta|+|b|\geq |\gamma|+|c| \end{matrix}}
\sum_{1\leq i \leq 2}
\big\|\frac{\nabla_iV^{(\beta,b)}+\nabla_iH^{(\beta,b)}\cdot\omega}{\langle t-r \rangle}\big\|_{L^2}
\|\langle t-r \rangle\nabla U^{(\gamma,c)}\|_{L^\infty(r\geq \langle t \rangle/2)} E^{\frac12}_{\kappa-1}\\
&+ \sum_{\tiny\begin{matrix} \beta+\gamma=\alpha, b + c = a\\ |\beta|+|b|\geq |\gamma|+|c| \end{matrix}}
\sum_{1\leq i \leq 2}
\|\nabla H^{(\beta,b)}\|_{L^2}
\|\nabla_i H^{(\gamma,c)}\cdot\omega^\perp\|_{L^\infty(r\geq \langle t \rangle/2)}
E^{\frac12}_{\kappa-1}\\
&\qquad \lesssim
\eta G_\kappa
+C_{\eta}\langle t \rangle ^{-1} E_{\kappa-1} E_{\kappa-3}
+\langle t \rangle ^{-\frac 32}\mathcal{E}_\kappa^{\frac12} E_{\kappa-1}^{\frac12} E_{\kappa-3}^{\frac12}.
\end{align*}
If $|\beta|+|b|< |\gamma|+|c|$,
we can repeat a  similar procedure to deduce that
the right hand side of \eqref{est-good-f} can be bounded by
\begin{align*}
\langle t \rangle ^{-\frac 32}\mathcal{E}_\kappa^{\frac12} E_{\kappa-1}^{\frac12} E_{\kappa-3}^{\frac12}.
\end{align*}
It then follows by gathering the above estimates   that
\begin{align*}
\int_{\mathbb{R}^2} f^1_{\alpha a}  V^{(\alpha,a)}dx
\lesssim  \eta G_\kappa
+C_{\eta}\langle t \rangle ^{-1}(\mathcal{E}_{\kappa}+E_{\kappa-1} )E_{\kappa-3}^{\frac 12}.
\end{align*}

The estimate of $\int_{\mathbb{R}^2} f^2_{\alpha a} \cdot H^{(\alpha,a)}dx$ is similar to $\int_{\mathbb{R}^2} f^1_{\alpha a} \cdot V^{(\alpha,a)}dx$.
The key point is to explore the appropriate null structure for $f^2_{\alpha a}$.
We will prove that
\begin{align} \label{f2est}
|f^2_{\alpha a}|\lesssim
&\sum\limits_{\tiny\begin{matrix}b+c=a\\\beta+\gamma=\alpha\end{matrix}}
|\nabla^\perp_j H^{(\beta,b)}\cdot\omega^\perp
\nabla_j V^{(\gamma,c)}| \nonumber\\
&+\sum\limits_{\tiny\begin{matrix}b+c=a\\\beta+\gamma=\alpha\end{matrix}}
|(\nabla^\perp_j H^{(\beta,b)}\cdot\omega
+\nabla^\perp_j V^{(\beta,b)})
\nabla_j V^{(\gamma,c)}\omega|,
\end{align}
from which we can   deduce that
\begin{align*}
\int_{\mathbb{R}^2} f^2_{\alpha a}  H^{(\alpha,a)}dx
\lesssim  \eta G_\kappa
+C_{\eta}\langle t \rangle ^{-1}(\mathcal{E}_{\kappa}+E_{\kappa-1} )E_{\kappa-3}^{\frac 12}.
\end{align*}
Hence in the sequel, we only show \eqref{f2est}.

Employing the orthogonal decomposition onto radial and transverse directions,
we have
\begin{align*}
f^2_{\alpha a}&=\sum\limits_{\tiny\begin{matrix}b+c=a\\\beta+\gamma=\alpha\end{matrix}}
C_{\alpha}^\beta C_a^b
(\nabla^\perp_j H^{(\beta,b)}
\nabla_j V^{(\gamma,c)}) \\
&=\sum\limits_{\tiny\begin{matrix}b+c=a\\\beta+\gamma=\alpha\end{matrix}}
C_{\alpha}^\beta C_a^b
(\nabla^\perp_j H^{(\beta,b)}\cdot\omega
\nabla_j V^{(\gamma,c)})\omega \\
&+\sum\limits_{\tiny\begin{matrix}b+c=a\\\beta+\gamma=\alpha\end{matrix}}
C_{\alpha}^\beta C_a^b
(\nabla^\perp_j H^{(\beta,b)}\cdot\omega^\perp
\nabla_j V^{(\gamma,c)})\omega^\perp.
\end{align*}
For the first line on the right hand side in the above, we reorganize them as
\begin{align*}
&\sum\limits_{\tiny\begin{matrix}b+c=a\\\beta+\gamma=\alpha\end{matrix}}
C_{\alpha}^\beta C_a^b
(\nabla^\perp_j H^{(\beta,b)}\cdot\omega
\nabla_j V^{(\gamma,c)})\omega\\
&=\sum\limits_{\tiny\begin{matrix}b+c=a\\\beta+\gamma=\alpha\end{matrix}}
C_{\alpha}^\beta C_a^b
(\nabla^\perp_j H^{(\beta,b)}\cdot\omega
+\nabla^\perp_j V^{(\beta,b)})
\nabla_j V^{(\gamma,c)}\omega.
\end{align*}
Here we have used the fact that
\begin{equation*}
\sum\limits_{\tiny\begin{matrix}b+c=a\\\beta+\gamma=\alpha\end{matrix}}
C_{\alpha}^\beta C_a^b
\nabla^\perp_j V^{(\beta,b)}
\nabla_j V^{(\gamma,c)}\omega=0.
\end{equation*}
This yields  \eqref{f2est}.

Finally, we gather our estimate for \eqref{EN-EsH} to derive that
\begin{eqnarray}\label{EN-Es3}
&&\frac 12\frac{d}{dt}
\int_{\mathbb{R}^2}(|V^{(\alpha,a)}|^2
    +|H^{(\alpha,a)}|^2)dx\nonumber\\
&&-\int_{\mathbb{R}^2}
\mu\Delta\sum\limits_{l=0}^{\alpha} C_\alpha^l (-1)^{\alpha-l}V^{(l,a)}
\cdot V^{(\alpha,a)} dx \nonumber\\
&&\leq
\eta G_\kappa+C_{\eta}\langle t \rangle ^{-1}
(\mathcal{E}_{\kappa}+E_{\kappa-1} )E_{\kappa-3}^{\frac 12}.
\end{eqnarray}
For the diffusion terms in \eqref{EN-Es3}, we estimate them as follows:
\begin{align}\label{min-ener}
&-\int_{\mathbb{R}^2}
\mu\Delta\sum\limits_{l=0}^{\alpha} C_\alpha^l (-1)^{\alpha-l}V^{(l,a)}
\cdot V^{(\alpha,a)} dx\nonumber\\
&=\mu\int_{\mathbb{R}^2} |\nabla V^{(\alpha,a)}|^2 dx
+\int_{\mathbb{R}^2} \mu\nabla\sum\limits_{l=0}^{\alpha-1}C_\alpha^l (-1)^{\alpha-l}V^{(l,a)}
\cdot \nabla V^{(\alpha,a)} dx \nonumber\\
&\geq \frac12\mu\int_{\mathbb{R}^2} |\nabla V^{(\alpha,a)}|^2 dx
-\frac12 \mu\sum\limits_{l=0}^{\alpha-1} (C_\alpha^l)^2\int_{\mathbb{R}^2} |\nabla V^{(l,a)}|^2 dx .
\end{align}
Hence  we can deduce that
\begin{eqnarray*}
&&\frac 12\frac{d}{dt}
\int_{\mathbb{R}^2}(|V^{(\alpha,a)}|^2
    +|H^{(\alpha,a)}|^2)dx\nonumber\\
&&+\frac12\mu\int_{\mathbb{R}^2} |\nabla V^{(\alpha,a)}|^2 dx
-\frac12 \mu\sum\limits_{l=0}^{\alpha-1} (C_\alpha^l)^2\int_{\mathbb{R}^2} |\nabla V^{(l,a)}|^2 dx  \nonumber\\
&&\leq  \eta G_\kappa(t)+C_{\eta}\langle t \rangle ^{-1}
(\mathcal{E}_{\kappa}+E_{\kappa-1} )E_{\kappa-3}^{\frac 12}.
\end{eqnarray*}
Integrating   both sides of the above inequality in time over $[0,t)$, we get
\begin{eqnarray*}
&&\frac 12
\int_{\mathbb{R}^2}(|V^{(\alpha,a)}(t)|^2
    +|H^{(\alpha,a)}(t)|^2)dx\nonumber\\
&&+\frac12\mu\int_0^t\int_{\mathbb{R}^2} |\nabla V^{(\alpha,a)}(\tau)|^2 dxd\tau
-\frac12\mu \sum\limits_{l=0}^{\alpha-1} (C_\alpha^l)^2\int_0^t\int_{\mathbb{R}^2} |\nabla V^{(l,a)}(\tau)|^2 dxd\tau  \nonumber\\
&&\leq \frac 12
\int_{\mathbb{R}^2}(|V^{(\alpha,a)}(0)|^2
    +|H^{(\alpha,a)}(0)|^2)dx\nonumber\\
&&+\eta \int_0^t G_\kappa(\tau)d\tau+
C_{\eta} \int_0^t \langle \tau \rangle ^{-1}
\big(\mathcal{E}_{\kappa}(\tau)+E_{\kappa-1}(\tau)\big)E_{\kappa-3}^{\frac 12}(\tau)d\tau.
\end{eqnarray*}
Using Lemma \ref{lemItera}, we deduce that
\begin{align}\label{HighEnergyEst}
&\int_{\mathbb{R}^2}(|V^{(\alpha,a)}(t)|^2+|H^{(\alpha,a)}(t)|^2)dx
+\mu\int_0^t\int_{\mathbb{R}^2} |\nabla V^{(\alpha,a)}(\tau)|^2 dxd\tau \nonumber\\
&\lesssim E_{\kappa-1}(0)
+\eta \int_0^t G_\kappa(\tau)d\tau+
C_{\eta} \int_0^t \langle \tau \rangle ^{-1}
\big(\mathcal{E}_{\kappa}(\tau)+E_{\kappa-1}(\tau)\big)E_{\kappa-3}^{\frac 12}(\tau)d\tau.
\end{align}

Now we are going to combine the highest-order modified energy estimate of the previous subsection
 with the standard one
to deal with the diffusion energy with  the negative sign in \eqref{ModHighEneEs}.
Multiplying \eqref{HighEnergyEst} by $4\max_{\sigma\in\mathbb{R}}{e^{q(\sigma)}}$,
then adding \eqref{ModHighEneEs},
we get
\begin{align}\label{HighEnergyEs2}
&\int_{\mathbb{R}^2}(|\nabla V^{(\alpha,a)}(t)|^2
    +|\nabla H^{(\alpha,a)}(t)|^2)e^qdx
+\int_{\mathbb{R}^2}(| V^{(\alpha,a)}(t)|^2
    +| H^{(\alpha,a)}(t)|^2)dx \nonumber\\
&+\mu \int_0^t\int_{\mathbb{R}^2} |\Delta V^{(\alpha,a)}(\tau)|^2 e^qdxd\tau
-\sum\limits_{l=0}^{\alpha-1}
\mu C\int_{\mathbb{R}^2}
|\Delta V^{(l,a)}|^2e^qdx \nonumber\\
&+\mu  \int_0^t\int_{\mathbb{R}^2} |\nabla V^{(\alpha,a)}(\tau)|^2 dxd\tau
\nonumber\\
&+\sum_{1\leq k \leq 2} \int_0^t\int_{\mathbb{R}^2}\frac{|\nabla_kV^{(\alpha,a)}(\tau)
+\nabla_kH^{(\alpha,a)}(\tau)\cdot\omega|^2
+|\nabla_kH^{(\alpha,a)}(\tau)\cdot\omega^\perp|^2}
{\langle t-r \rangle^2}e^qdxd\tau \nonumber\\
&\lesssim \eta \int_0^t G_\kappa(\tau)d\tau+C_{\eta}
\int_0^t \langle \tau \rangle ^{-1} (\mathcal{E}_\kappa(\tau)+E_{\kappa-1} )E_{\kappa-3}^{\frac 12}(\tau)d\tau
+\mathcal{E}_{\kappa}(0)+E_{\kappa-1}(0) .
\end{align}
Summing over  all $|\alpha|+|a|\leq \kappa-1$  and using  Lemma \ref{lemItera} to handle the negative sign diffusion energy
on the left hand side of \eqref{HighEnergyEs2}, we get that
\begin{align*}
&\mathcal{E}_{\kappa}(t) + E_{\kappa-1}(t)
+ \int_0^t G_\kappa(\tau)d\tau\nonumber\\ &\quad\quad
+\ \mu\sum_{|\alpha|+|a|\leq \kappa-1}
\int_0^t\int_{\mathbb{R}^2} |\Delta V^{(\alpha,a)}(\tau)|^2
+ |\nabla V^{(\alpha,a)}(\tau)|^2 dxd\tau
\nonumber\\
&\lesssim \eta \int_0^t G_\kappa(\tau)d\tau+C_{\eta}
\int_0^t \langle \tau \rangle ^{-1}(\mathcal{E}_\kappa(\tau)+E_{\kappa-1}(\tau))
 E_{\kappa-3}^{\frac 12}(\tau)d\tau\nonumber\\
&\quad\quad +\ \mathcal{E}_{\kappa}(0)+E_{\kappa-1}(0) .
\end{align*}
Taking  $\eta$ small enough, we conclude that
\begin{align*}
&\mathcal{E}_{\kappa}(t) + E_{\kappa-1}(t)
+ \int_0^t G_\kappa(\tau)d\tau\nonumber\\
&\quad\quad +\ \sum_{|\alpha|+|a|\leq \kappa-1}
\mu \int_0^t\int_{\mathbb{R}^2} |\Delta V^{(\alpha,a)}(\tau)|^2
+ |\nabla V^{(\alpha,a)}(\tau)|^2 dxd\tau\nonumber\\
&\lesssim
\int_0^t \langle \tau \rangle ^{-1}(\mathcal{E}_\kappa(\tau)+E_{\kappa-1}(\tau))
 E_{\kappa-3}^{\frac 12}(\tau)d\tau
+\mathcal{E}_{\kappa}(0)+E_{\kappa-1}(0) .
\end{align*}
This is the desired \textit{a priori} estimate \eqref{PriorEsti1}.

\subsection{Lower-order Standard Energy Estimate}
In this last subsection, we present the lower-order standard energy estimate.
A  trick here is that we need to earn the maximum decay in time.
In order to achieve this, we are going to take full advantage of the inherent strong null structure.

Let $|\alpha|+|a| \leq \kappa-3$. Taking the $L^2$ inner product of the
first and  second equation of \eqref{VisElas-Gamma} with $V^{(\alpha,a)}$ and $H^{(\alpha,a)}$, respectively,
we get
\begin{align*}
&\frac 12\frac{d}{dt}
\int_{\mathbb{R}^2}(|V^{(\alpha,a)}|^2+|H^{(\alpha,a)}|^2)dx
-\int_{\mathbb{R}^2}\mu\Delta\sum\limits_{l=0}^{\alpha} C_\alpha^l (-1)^{\alpha-l}V^{(l,a)}
\cdot V^{(\alpha,a)} dx \nonumber\\
&=\int_{\mathbb{R}^2} f^1_{\alpha a} V^{(\alpha,a)}
+ f^2_{\alpha a} \cdot H^{(\alpha,a)} dx\\
&\leq   \|f_{ij}^{\alpha a}\|_{L^2} \|V^{(\alpha,a)}\|_{L^2} + \|f^2_{\alpha a}\|_{L^2} \|H^{(\alpha,a)} \|_{L^2}
.
\end{align*}
We have used the $L^2$ boundness of the Riesz transform in the last bound.
Here we recall that $f^{\alpha a}_{ij}$ was  defined in \eqref{fij}.

Now we are going to treat $\|f_{ij}^{\alpha a}\|_{L^2}$.
First, we have
\begin{equation*}
\| f_{ij}^{\alpha a}\|_{L^2(r\leq \langle t \rangle/2)}
\lesssim
\sum_{\tiny{\begin{matrix}b + c = a,\\ \beta+\gamma=\alpha \end{matrix}}}
\big\| |\nabla V^{(\beta,b)}| |\nabla V^{(\gamma,c)}|+
|\nabla H^{(\beta,b)}| |\nabla H^{(\gamma,c)}|
\big\|_{L^2(r\leq \langle t \rangle/2)}.
\end{equation*}
Since the index $(\beta,b)$ and $(\gamma,c)$ in the above quantity are symmetric,
we can  assume  that $|\gamma|+|c|\leq|\beta|+|b|$ without loss of generality.
Thus $|\gamma|+|c|+3\leq [(|\alpha|+|a|)/2]+3\leq \kappa-4$.
In view of \eqref{K-S-3} and Lemma \ref{lemmaWeighted-Energy}, we get
\begin{align*}
&\| f_{ij}^{\alpha a}\|_{L^2(r\leq \langle t \rangle/2)}
 \lesssim
\sum_{\tiny\begin{matrix}b + c = a,\beta+\gamma=\alpha
 \\ |c|+|\gamma| \leq|b|+|\beta|\end{matrix}}
 \big\|  |\nabla U^{(\beta,b)}| |\nabla U^{(\gamma,c)}|
   \big\|_{L^2(r\leq \langle t \rangle/2)}\\
&\lesssim
\sum_{\tiny\begin{matrix}b + c = a,\beta+\gamma=\alpha
 \\ |c|+|\gamma| \leq|b|+|\beta|\end{matrix}}
\big\|\langle t \rangle^{-2}  \langle t-r \rangle^2
|\nabla U^{(\beta,b)}| |\nabla U^{(\gamma,c)}|
\big\|_{L^2(r\leq \langle t \rangle/2)} \\
&\lesssim\sum_{\tiny\begin{matrix}b + c = a,\beta+\gamma=\alpha
 \\ |c|+|\gamma| \leq|b|+|\beta|\end{matrix}}
\langle t \rangle^{-2}
X^{\frac{1}{2}}_{|\beta|+|b|+1}X^{\frac{1}{2}}_{|\gamma|+|c|+3}
\leq \langle t \rangle^{-2} X^{\frac{1}{2}}_{\kappa-2}X^{\frac{1}{2}}_{\kappa-4} \\
&\lesssim \langle t \rangle^{-2} E^{\frac 12}_{\kappa-1}(t) E^{\frac 12}_{\kappa-3}(t).
\end{align*}
Moreover, in the region  $\{ r\geq \langle t \rangle /2\} $, by \eqref{good-f1}, we get
\begin{align}\label{est-egood-f}
\|f^{ij}_{\alpha a}\|_{L^2(r\geq \langle t \rangle/2)}
& \lesssim \big\| \frac 1r
\sum\limits_{\tiny\begin{matrix}|b|+|c|\leq |a|\\ |\beta|+|\gamma|\leq |\alpha|\end{matrix}}
|V^{(|\beta|,|b|+1)}  V^{(|\gamma|,|c|+1)}|  + |H^{(|\beta|,|b|+1)}| |H^{|\gamma|,|c|+1}|    \big\|_{L^2(r\geq \langle t \rangle/2)}
\nonumber\\
&+\big\|\sum\limits_{\tiny\begin{matrix}b+c=a\\ \beta+\gamma=\alpha\end{matrix}}
|\partial_rV^{(\beta,b)}+\partial_rH^{(\beta,b)}\cdot\omega|
(|\nabla V^{(\gamma,c)}|+|\nabla H^{(\gamma,c)}|)\big\|_{L^2(r\geq \langle t \rangle/2)}
\nonumber\\
&+\|\sum\limits_{\tiny\begin{matrix}b+c=a\\ \beta+\gamma=\alpha\end{matrix}}
\partial_rH^{(\beta,b)}\cdot\omega^\perp \partial_rH^{(\gamma,c)}\cdot\omega^\perp
\|_{L^2(r\geq \langle t \rangle/2)}.
\end{align}
For the first line on the right hand side of \eqref{est-egood-f}, by the
symmetry between    the index $(\beta,b)$ and $(\gamma,c)$, we can  assume  that
$|b|+|\beta|\leq|c|+|\gamma|$. Thus $|b|+|\beta|+3 \leq [( |a|+|\alpha|)/2]+3\leq \kappa-4 $. By \eqref{K-S-1},
the first line can be   estimated by
\begin{align*}
& \langle t\rangle^{-1}
\sum_{\tiny\begin{matrix}|b|+|c|+|\beta|+|\gamma|\leq |a|+|\alpha|
 \\|b|+|\beta|\leq|c|+|\gamma|\end{matrix}}
\|U^{(|\beta|,|b|+1)}\|_{L^\infty(r\geq \langle t\rangle/2) } \|U^{(|\gamma|,|c|+1)}\|_{L^2 }\\
&\lesssim \langle t\rangle^{-\frac 32}
\sum_{\tiny\begin{matrix}|b|+|c|+|\beta|+|\gamma|\leq |a|+|\alpha|
 \\|b|+|\beta|\leq|c|+|\gamma|\end{matrix}} E_{|c|+|\gamma|+1}^{\frac 12}E_{|b|+|\beta|+3}^{\frac 12}
\lesssim  \langle t\rangle^{-\frac 32}E_{\kappa-3}^{\frac 12} E_{\kappa-1}^{\frac 12}.
\end{align*}
For the second line on the right-hand side of \eqref{est-egood-f}, if $|b|+|\beta|\geq|c|+|\gamma|$,
then by Lemma \ref{lemmaWeighted-Energy}, we have
\begin{align*}
&\big\|\sum\limits_{\tiny\begin{matrix}b+c=a,\beta+\gamma=\alpha\\|b|+|\beta|\geq|c|+|\gamma| \end{matrix}}
(\partial_rV^{(\beta,b)}+\partial_rH^{(\beta,b)}\cdot\omega)
(|\nabla V^{(\gamma,c)}|+|\nabla H^{(\gamma,c)}) \big\|_{L^2(r\geq \langle t \rangle/2)}\nonumber\\
&\lesssim \langle t \rangle^{-1}
\sum\limits_{\tiny\begin{matrix}b+c=a,\beta+\gamma=\alpha\\|b|+|\beta|\geq|c|+|\gamma| \end{matrix}}
\| r(\partial_rV^{(\beta,b)}+\partial_rH^{(\beta,b)}\cdot\omega) \|_{L^2}
\|\nabla U^{(\gamma,c)}\|_{L^\infty(r\geq \langle t \rangle/2)}\nonumber\\
&\lesssim \langle t \rangle^{-\frac32}
\sum\limits_{\tiny\begin{matrix}b+c=a,\beta+\gamma=\alpha\\|b|+|\beta|\geq|c|+|\gamma| \end{matrix}}
Y_{|\beta|+|b|+1}^{\frac 12}  E_{|\gamma|+c+3}^{\frac 12}
\lesssim  \langle t\rangle^{-\frac 32}E_{\kappa-3}^{\frac 12} E_{\kappa-1}^{\frac 12}.
\end{align*}
Otherwise, if $|b|+|\beta|\leq|c|+|\gamma|$, then by \eqref{der-decomp} and Lemma \ref{LemmaGood-L-infy}, we have
\begin{align*}
&\sum\limits_{\tiny\begin{matrix}b+c=a,\beta+\gamma=\alpha\\ |b|+|\beta|\leq|c|+|\gamma| \end{matrix}}
\big\| (\partial_rV^{(\beta,b)}
+\partial_rH^{(\beta,b)}\cdot\omega)
(|\nabla V^{(\gamma,c)}|+|\nabla H^{(\gamma,c)}) \big\|_{L^2(r\geq \langle t \rangle/2)}\nonumber\\
&\lesssim
\sum\limits_{\tiny\begin{matrix}b+c=a,\beta+\gamma=\alpha\\ |b|+|\beta|\leq|c|+|\gamma| \end{matrix}}
\|\partial_rV^{(\beta,b)}
+\partial_rH^{(\beta,b)}\cdot\omega \|_{L^\infty(r\geq \langle t \rangle/2)}
\| \nabla U^{(\gamma,c)}\|_{L^2}\nonumber\\
&\lesssim   \langle t\rangle^{-\frac 32}E_{\kappa-3}^{\frac 12}  E_{\kappa-1}^{\frac 12}.
\end{align*}
The estimate of the third line of \eqref{est-egood-f} can be treated exactly as the second line.
Thus we conclude by gathering the estimates  that
\begin{align*}
&\| f_{ij}^{\alpha a}\|_{L^2}\lesssim
 \langle t \rangle^{-\frac32} E^{\frac 12}_{\kappa-1}(t) E^{\frac 12}_{\kappa-3}(t).
\end{align*}

Next for $\|f^2_{\alpha a}\|_{L^2}$, we can use the same strategy as the estimate of $\| f_{ij}^{\alpha a}\|_{L^2}$
to get the same bound.
Thus, we gather all the estimate in this subsection to deduce that
\begin{align*}
&\frac 12\frac{d}{dt}
\int_{\mathbb{R}^2}(|V^{(\alpha,a)}|^2+|H^{(\alpha,a)}|^2)dx\\
&-\int_{\mathbb{R}^2}\mu\Delta\sum\limits_{l=0}^{\alpha} C_\alpha^l (-1)^{\alpha-l}V^{(l,a)}
\cdot V^{(\alpha,a)} dx
\lesssim \langle t\rangle^{-\frac 32}E_{\kappa-3} E_{\kappa-1}^{\frac 12}.
\end{align*}
For the viscosity terms, by \eqref{min-ener}, we get
\begin{align*}
&\frac 12\frac{d}{dt}
\int_{\mathbb{R}^2}(|V^{(\alpha,a)}|^2+|H^{(\alpha,a)}|^2)dx \\
&+\frac 12\mu\int_{\mathbb{R}^2} |\nabla V^{(\alpha,a)}|^2 dx
-\sum\limits_{l=0}^{\alpha-1}
\mu C\int_{\mathbb{R}^2}
|\nabla V^{(l,a)}|^2dx
\lesssim \langle t\rangle^{-\frac 32}E_{\kappa-3} E_{\kappa-1}^{\frac 12}.
\end{align*}
We can integrate in time on $[0,t)$ over the above inequality,
then use Lemma \ref{lemItera} to absorb the diffusion energy with negative sign.
Finally, summing over $|\alpha|+|a| \leq \kappa-3$, we get
\begin{align*}
&E_{\kappa-3}(t)+ \sum_{|\alpha| +|a| \leq \kappa-3}
\mu\int_0^t\int_{\mathbb{R}^2} |\nabla V^{(\alpha,a)}(\tau)|^2 dxd\tau\\
&\lesssim E_{\kappa-3}(0)+\int_0^t \langle \tau\rangle^{-\frac 32}
E_{\kappa-3}(\tau) E_{\kappa-1}^{\frac 12}(\tau) d\tau.
\end{align*}
This is the desired \textit{a priori} estimate \eqref{PriorEsti2}.

\section{Acknowledgement}
Part of this work was carried out while Y. Cai was visiting the Courant Institute.
He would like to thank the hospitality of the institute.
The first two authors were in part supported by the NSFC grant (Grant No. 11421061 and 11222107),
the National Support Program for Young Top-Notch Talents, the Shanghai Shu Guang project,
the Shanghai Talent Development Fund, and SGST 09DZ2272900.
Cai was also sponsored by the China Scholarship Council (No. 201606100111) for one year
at New York University, Courant Institute of Mathematics Sciences. Nader Masmoudi was in part supported by  NSF grant DMS-1211806. F.-H. Lin was in part supported by the NSF grant DMS-1501000.

\end{document}